\pdfoutput=1

\documentclass[12pt]{article}
\usepackage[T1]{fontenc}
\usepackage[utf8]{inputenc}
\usepackage[a4paper,margin=1in]{geometry}
\usepackage{amsmath,amssymb,amsthm,mathtools,mathrsfs}
\usepackage{cite}
\usepackage{hyperref}
\usepackage{xcolor}
\colorlet{blue}{black}
\usepackage{enumitem}
\usepackage{bm}
\usepackage{lmodern}

\hypersetup{
  colorlinks=true,
  linkcolor=black,
  citecolor=black,
  urlcolor=black,
  pdftitle={Threshold asymptotics and decay for massive Maxwell on subextremal Reissner--Nordstrom},
  pdfauthor={Bobby Eka Gunara},
  pdfstartview={FitH}
}

\numberwithin{equation}{section}
\allowdisplaybreaks[2]
\theoremstyle{plain}
\newtheorem{theorem}{Theorem}[section]
\newtheorem{proposition}[theorem]{Proposition}
\newtheorem{lemma}[theorem]{Lemma}
\newtheorem{corollary}[theorem]{Corollary}

\theoremstyle{definition}

\theoremstyle{remark}
\newtheorem{remark}{Remark}[section]

\newcommand{\R}{\mathbb R}
\newcommand{\C}{\mathbb C}
\newcommand{\N}{\mathbb N}
\newcommand{\Z}{\mathbb Z}
\newcommand{\RN}{\mathrm{RN}}
\newcommand{\ii}{\mathrm i}
\newcommand{\ee}{\mathrm e}
\newcommand{\dd}{\mathrm d}
\newcommand{\disc}{\operatorname{disc}}
\newcommand{\Res}{\operatorname{Res}}
\newcommand{\diag}{\operatorname{diag}}

\newcommand{\bc}{\mathrm{bc}}
\newcommand{\qb}{\mathrm{qb}}
\newcommand{\fast}{\mathrm{fast}}
\newcommand{\hor}{\mathrm{hor}}
\newcommand{\out}{\mathrm{out}}
\newcommand{\cG}{\mathcal G}
\newcommand{\cR}{\mathcal R}
\newcommand{\cW}{\mathcal W}
\newcommand{\cA}{\mathcal A}
\newcommand{\cB}{\mathcal B}
\newcommand{\cC}{\mathcal C}
\newcommand{\cS}{\mathcal S}
\newcommand{\cM}{\mathcal M}
\newcommand{\cE}{\mathcal E}
\newcommand{\la}{\lambda}
\newcommand{\om}{\omega}
\newcommand{\vp}{\varpi}
\newcommand{\kap}{\kappa}
\newcommand{\eps}{\varepsilon}
\newcommand{\angles}[1]{\left\langle #1 \right\rangle}
\newcommand{\abs}[1]{\left|#1\right|}
\newcommand{\norm}[1]{\left\|#1\right\|}
\newcommand{\Id}{\mathrm{Id}}
\newcommand{\chg}[1]{#1}

\newcommand{\appheading}[1]{\par\medskip\noindent\textbf{#1}\par\smallskip\noindent}
\newcommand{\applabelheading}[2]{\par\medskip\phantomsection\label{#2}\noindent\textbf{#1}\par\smallskip\noindent}

\title{\textbf{\chg{Threshold asymptotics and decay for massive Maxwell on subextremal Reissner--Nordstr\"om}}}
\author{Bobby Eka Gunara\\[1ex]
\small Theoretical Physics Laboratory,\\
\small Theoretical High Energy Physics Research Division,\\
\small Faculty of Mathematics and Natural Sciences,\\
\small Institut Teknologi Bandung,\\
\small Jl. Ganesha no. 10 Bandung, Indonesia, 40132\\[0.5ex]
\small\texttt{Email: bobby@itb.ac.id}}
\date{}

\begin{document}
\maketitle

\begin{abstract}

We study the neutral massive Maxwell (Proca) equation on subextremal Reissner--Nordstr\"om exteriors.  After spherical-harmonic decomposition, the odd sector is scalar, while the even sector remains a genuinely coupled $2\times2$ system.  Our starting point is that this even system admits an exact asymptotic polarization splitting at spatial infinity.  The three resulting channels carry effective angular momenta $\ell-1$, $\ell$, and $\ell+1$, and these are precisely the indices that govern the late-time thresholds.
For each fixed angular momentum we develop a threshold spectral theory for the cut-off resolvent.  We prove meromorphic continuation across the massive branch cut, rule out upper-half-plane modes and threshold resonances, and obtain explicit small- and large-Coulomb expansions for the branch-cut jump.  Inverting this jump yields polarization-resolved intermediate tails together with the universal very-late $t^{-5/6}$ branch-cut law.
At the full-field level, high-order angular regularity allows us to sum the modewise leading terms on compact radial sets and obtain a two-regime asymptotic expansion for the radiative branch-cut component of the Proca field, with explicit coefficient fields and quantitative remainders.  We also analyze the quasibound resonance branches created by stable timelike trapping, prove residue and reconstruction bounds, and derive a fully self-contained dyadic packet estimate.  As a result, the unsplit full Proca field obeys logarithmic compact-region decay, while the radiative branch-cut contribution retains explicit polynomial asymptotics and explicit leading coefficients.

\end{abstract}

\tableofcontents

\section{Introduction and main results}

\subsection{Background and scope}

Massive fields on black-hole exteriors behave differently from massless ones in two coupled ways.  The mass creates branch points at $\om=\pm\mu$, so late times are governed by oscillatory threshold tails rather than the familiar massless picture.  At the same time, it produces stable timelike trapping and therefore a discrete family of exponentially long-lived quasibound states.  For scalar fields on subextremal Reissner--Nordstr\"om, this combined picture is now understood both mode by mode and after summation over angular momenta \cite{PasqualottoShlapentokhRothmanVanDeMoortel2024,ShlapentokhRothmanVanDeMoortel2026}.  The Proca equation is the next natural model, but it is also the first genuinely vectorial one: after spherical-harmonic reduction, the odd sector is scalar whereas the even sector remains a coupled $2\times2$ system.

Our goal is to give a rigorous late-time theory for the neutral Proca equation on subextremal Reissner--Nordstr\"om.  We work on a fixed exterior
\begin{equation}\label{eq:intro_subextremal}
   M>0,
   \qquad
   0\le \abs{Q}<M,
\end{equation}
with event and Cauchy radii
\begin{equation}\label{eq:intro_rpm}
   r_\pm=M\pm\sqrt{M^2-Q^2}.
\end{equation}
Unless stated otherwise we restrict to angular momenta $\ell\ge1$, so that the odd channel and both even polarizations are present.  All pointwise statements are uniform on compact radial sets $K\Subset(r_+,\infty)$ and are formulated either for a fixed mode or after summing over all angular momenta.  The exceptional monopole $\ell=0$ supports only the even electric channel and is treated separately in Appendix~\ref{app:small_mass}.  We exclude the extremal case $\abs{Q}=M$, where the red-shift degenerates, the near-horizon $\mathrm{AdS}_2\times S^2$ throat changes the threshold structure, and Aretakis-type horizon phenomena enter \cite{LuciettiMurataReallTanahashi2013}.

What makes Proca harder than the scalar problem is not a single complication but a combination of three.  First, the mass removes gauge freedom, so the field itself is the natural unknown.  Second, the even sector is genuinely matrix valued at threshold.  Third, the same polarization structure must remain compatible with the large-angular-momentum semiclassical analysis needed to control the quasibound family.  A full decay theory must therefore connect fixed-mode threshold analysis, trapping-driven resonance theory, and unsplit full-field reconstruction in one coherent framework.

This is exactly what we do on subextremal Reissner--Nordstr\"om.  For each fixed angular momentum we prove a threshold spectral theorem for the odd and even sectors, derive explicit polarization-resolved branch-cut asymptotics, sum those leading terms into full-field radiative profiles with quantitative remainders, construct the quasibound poles generated by stable timelike trapping, and finally combine the continuous and discrete spectral contributions to obtain a self-contained compact-region decay theorem for the full Proca field.

\subsection{Relation to the literature and the main innovation}

Earlier work already points toward the picture one should expect.  For Schwarzschild, Rosa and Dolan wrote the odd/even reduction for Proca explicitly \cite{RosaDolan2012}, while Konoplya, Zhidenko, and Molina predicted polarization-dependent intermediate tails and the universal very-late $t^{-5/6}$ law by formal and numerical arguments \cite{KonoplyaZhidenkoMolina2007}.  The static Reissner--Nordstr\"om problem studied here is also the charged zero-rotation limit of the separated Kerr--Newman Proca system of \cite{CayusoDiasGrayKubiznakMargalitSantosSouzaThiele2020}.  On the scalar side, Pasqualotto, Shlapentokh-Rothman, and Van de Moortel established exact fixed-mode massive tails for stationary spherically symmetric black holes \cite{PasqualottoShlapentokhRothmanVanDeMoortel2024}, and Shlapentokh-Rothman and Van de Moortel later proved compact-region decay on subextremal Reissner--Nordstr\"om after summing over angular momenta \cite{ShlapentokhRothmanVanDeMoortel2026}.

What was still missing was a single rigorous framework that could simultaneously handle the massive branch cut, the vectorial polarization structure, the quasibound resonance family created by stable timelike trapping, and the unsplit full-field problem for Proca on a charged black-hole background.  To the best of our knowledge, the present paper is the first to provide such a framework.

The key new point is structural.  The even Proca sector admits an exact asymptotic polarization splitting at spatial infinity: once one passes from the standard pair $(u_2,u_3)$ to the constant combinations \eqref{eq:intro_even_pol_basis}, the leading inverse-square part diagonalizes into the three effective angular momenta
\[
   L_{-1}=\ell-1,
   \qquad
   L_0=\ell,
   \qquad
   L_{+1}=\ell+1.
\]
This is the step that makes the problem manageable.  It identifies the correct threshold indices, shows that the charge-dependent coupling is shorter range than the leading asymptotic dynamics, and provides exactly the variables in which the large-$\ell$ semiclassical analysis can be closed.

Seen from that perspective, the contribution of the paper is not just a collection of estimates.  It is a single mechanism that links threshold analysis, resonance theory, and full-field reconstruction.  Concretely, the main results may be summarized as follows:
\begin{enumerate}[label=\textnormal{(\roman*)}]
   \item We establish a matrix-valued fixed-mode threshold theory for neutral Proca on subextremal Reissner--Nordstr\"om, including meromorphic continuation of the cut-off resolvent across the massive branch cut, exclusion of upper-half-plane modes and threshold resonances, and explicit small- and large-Coulomb threshold expansions.
   \item We derive explicit polarization-resolved late-time asymptotics for the branch-cut contribution, including the intermediate tails and the universal very-late $t^{-5/6}$ law.
   \item We sum the modewise leading branch-cut amplitudes to obtain a two-regime asymptotic expansion for the radiative part of the full Proca field on compact radial sets, with explicit leading coefficient fields and quantitative remainders in both the intermediate and very-late regimes.
   \item We construct the quasibound resonance branches generated by stable timelike trapping, together with uniform residue and reconstruction bounds strong enough for angular summation.
   \item We combine the continuous and discrete spectral pieces to prove a fully self-contained compact-region decay theorem for the unsplit Proca field.  The branch-cut part decays polynomially and admits explicit asymptotic profiles, while the discrete quasibound contribution is controlled by a logarithmic packet summation argument based only on tunnelling widths and high-order angular regularity.
\end{enumerate}

At the global summation stage the present version uses no imported packet theorem.  Instead we work only with ingredients established in the paper itself: two-sided tunnelling-width bounds for the quasibound poles and polynomial residue/reconstruction estimates strong enough to convert angular regularity of the data into arbitrary negative powers of the packet scale.  Summing the resulting packet bounds yields logarithmic decay for the quasibound contribution.  If one later reinstates the arithmetic packet theorem of \cite{ShlapentokhRothmanVanDeMoortel2026}, then the sharper polynomial rates from the earlier draft reappear as an optional upgrade.

\subsection{Notation and conventions}
For the reader's convenience we collect here all recurring notation used throughout the paper.  Auxiliary symbols that appear only inside one proof (for example a temporary cutoff, a local coefficient matrix, or a short-lived current) are reintroduced at first appearance.

\medskip\noindent\textbf{Background parameters and geometry.}\par\smallskip\noindent
The parameters $M>0$, $Q\in\R$, and $\mu>0$ denote respectively the black-hole mass, the electric charge of the background, and the Proca mass.  Throughout the paper we assume the subextremal condition $0\le \abs{Q}<M$.  The horizon radii are
\[
   r_\pm=M\pm\sqrt{M^2-Q^2},
\]
and the static coefficient is
\[
   f(r)=1-\frac{2M}{r}+\frac{Q^2}{r^2}=\frac{(r-r_+)(r-r_-)}{r^2}.
\]
The exterior region is $\{r>r_+\}$.  The tortoise coordinate $r_*$ is defined by $\dd r_*/\dd r=f(r)^{-1}$, so $r_*\to-\infty$ at the future event horizon and $r_*\to+\infty$ at spatial infinity.  The surface gravity of the event horizon is
\[
   \kappa_+:=\frac{r_+-r_-}{2r_+^2}>0.
\]
Whenever $K\Subset(r_+,\infty)$ appears, it denotes a fixed compact radial set in the exterior.

\medskip\noindent\textbf{Field variables and harmonic labels.}\par\smallskip\noindent
The unknown $A$ is the neutral Proca one-form, $F=\dd A$ is its field strength, and the field equation is
\[
   \nabla^\beta F_{\alpha\beta}+\mu^2A_\alpha=0.
\]
The angular mode labels are $(\ell,m)$ with $\ell\in\N$ and $m=-\ell,\dots,\ell$.  We write
\[
   \lambda_\ell:=\ell(\ell+1),
   \qquad
   \la:=\sqrt{\lambda_\ell}
\]
when it is convenient to use either $\lambda_\ell$ or $\la^2$ in the reduced equations.  The odd radial amplitude is denoted by $u_4$, the even amplitudes by $(u_2,u_3)$, and the exceptional monopole electric mode at $\ell=0$ by $u_0$.  Scalar spherical harmonics are written $Y_{\ell m}$, while $Y^{(P)}_{\ell m}$ denotes the polarization-adapted vector harmonics used in the full-field reconstruction.  When a point of the unit sphere is needed in pointwise estimates we write $\vartheta\in S^2$; in a few formulas the same role is played by an unadorned $\omega$, which should then be read as an angular variable rather than as a spectral frequency.

\medskip\noindent\textbf{Polarization variables.}\par\smallskip\noindent
We set $v_0:=u_4$ for the odd channel and define the even polarization basis by
\[
   v_{-1}:=\frac{\ell u_2+u_3}{2\ell+1},
   \qquad
   v_{+1}:=\frac{(\ell+1)u_2-u_3}{2\ell+1}.
\]
The polarization index is $P\in\{-1,0,+1\}$, so that $v_P$ means $v_{-1}$, $v_0$, or $v_{+1}$ according to the value of $P$.  The corresponding effective angular momenta are
\[
   L_{-1}=\ell-1,
   \qquad
   L_0=\ell,
   \qquad
   L_{+1}=\ell+1,
\]
and we often write this compactly as $L_P=\ell+P$.  The constant matrix $T_\ell$ is the even-sector change of basis from $(u_2,u_3)$ to $(v_{-1},v_{+1})$.  The diagonal matrix
\[
   D_\ell=\diag\bigl(L_{-1}(L_{-1}+1),L_{+1}(L_{+1}+1)\bigr)
\]
is the leading $r^{-2}$ even potential in the polarization basis, while $C_\ell$ is the shorter-range even coupling matrix entering at orders $r^{-3}$ and $r^{-4}$.  The finite-order differential operator $\mathcal B_\ell$ denotes the reconstruction map from the scalarized channel variables back to the physical Proca coefficients.

\medskip\noindent\textbf{Operators, energy spaces, and cutoffs.}\par\smallskip\noindent
For each fixed $\ell\ge1$, the reduced odd and even problems define the channel Hamiltonians
\[
   H_{\ell}^{\mathrm{odd}}=-\partial_{r_*}^2+V_{\ell}^{\mathrm{odd}},
   \qquad
   H_{\ell}^{\mathrm{ev}}=-\partial_{r_*}^2\Id_2+V_{\ell}^{\mathrm{ev}},
\]
and we write $H_\ell^{\sharp}$ when a statement applies to either channel.  The corresponding $L^2$ spaces are
\[
   \mathfrak h_{\ell}^{\mathrm{odd}}=L^2(\R_{r_*}),
   \qquad
   \mathfrak h_{\ell}^{\mathrm{ev}}=L^2(\R_{r_*};\C^2),
\]
while $\mathfrak h_\ell$ denotes the fixed-mode finite-energy space generically.  The selfadjoint first-order generator of the reduced time evolution is denoted by $\cA_\ell$.  Compactly supported radial cutoffs are written $\chi,\widetilde\chi\in C_0^\infty(r_+,\infty)$.  High-order initial-data norms are denoted by $\mathcal E_N[A[0]]$; these are the Sobolev-type mode energies controlling $\angles{\ell}^{N}$-weighted angular sums.

\medskip\noindent\textbf{Spectral and threshold parameters.}\par\smallskip\noindent
The complex spectral frequency is $\om$.  When the operator-theoretic spectral variable is needed we write
\[
   \lambda:=\om^2.
\]
Near the massive thresholds we use
\[
   \vp:=\sqrt{\mu^2-\om^2},
   \qquad
   \kap:=\frac{M\mu^2}{\vp},
\]
with the branch of $\vp$ chosen by the condition $\Re\vp>0$ on $\{\Im\om>0\}$.  The notation $\disc$ means the jump across the cut $[-\mu,\mu]$, and $\Res$ denotes residue.  The quantity $\nu_{\ell,P}$ is the threshold index determined by the inverse-square coefficient in the far-zone normal form.  The cut-off channel resolvent is written $R_{\ell,P}(\om)$, its Green kernel is $\cG_{\ell,P}(\om;r,r')$, the horizon and infinity basis solutions are marked by the subscripts $\hor$ and $\out$, $\cW_{\ell,P}$ denotes the associated Wronskian, and $\cE_\ell$ is the Evans determinant whose zeros coincide with poles of the meromorphically continued resolvent.

\medskip\noindent\textbf{Threshold amplitudes and time scales.}\par\smallskip\noindent
The functions $a_{\ell,P}(r,r')$ and $b_{\ell,P}^\pm(r,r')$ are the leading amplitudes in the small-$\kap$ and large-$\kap$ expansions of $\disc\,\cG_{\ell,P}$.  After inverse Fourier transform, $A_{\ell,P}(r,r';Q)$ and $B_{\ell,P}(r,r';Q)$ denote the corresponding intermediate- and very-late-time amplitudes, while $\delta_{\ell,P}(Q)$ and $\delta_{\ell,P,0}(Q)$ are the associated oscillatory phases.  The small-mass parameter is
\[
   \eps_{\mu,Q}:=(M\mu)^2+(Q\mu)^2.
\]
The two threshold time scales are
\[
   \kap_*(t)=M\mu^{3/2}t^{1/2},
   \qquad
   \vp_0(t)=\Bigl(\frac{2\pi M\mu^3}{t}\Bigr)^{1/3},
   \qquad
   \kap_0(t)=\frac{M\mu^2}{\vp_0(t)}.
\]

\medskip\noindent\textbf{Time-domain decomposition.}\par\smallskip\noindent
For each fixed $(\ell,m,P)$, $u_{\ell m,P}^{\bc}$ denotes the branch-cut contribution after subtraction of discrete pole residues, $u_{\ell m,P}^{\qb}$ denotes the quasibound contribution, and $u^{\fast}$ denotes the exponentially decaying remainder coming from poles and contour pieces bounded away from the real axis.  At the full-field level we write
\[
   A=A^{\bc}+A^{\qb}+A^{\fast},
\]
and sometimes $u_{\ell m}^{\mathrm{poles}}=u_{\ell m}^{\qb}+u_{\ell m}^{\fast}$.  The modal coefficients in the branch-cut harmonic expansion are written $a_{\ell m,P}^{\bc}$.  The smallest threshold index is
\[
   \nu_*:=\inf_{\ell\ge1,\ P\in\{-1,0,+1\}}\nu_{\ell,P},
\]
and the decay exponents used later are
\[
   \gamma_*:=\min\Bigl\{\nu_*+1,\frac56\Bigr\},
   \qquad
   \gamma_{\mathrm{qb}}:=\frac56-\frac1{23},
   \qquad
   \gamma_{\mathrm{full}}:=\min\{\gamma_*,\gamma_{\mathrm{qb}}\}.
\]

\medskip\noindent\textbf{Semiclassical and quasibound notation.}\par\smallskip\noindent
In the large-angular-momentum analysis we use the semiclassical parameter
\[
   h=(\ell+\tfrac12)^{-1}.
\]
The compact interval $I_{\mathrm{trap}}\Subset(0,\mu)$ is the trapped energy window, and $I\Subset I_{\mathrm{trap}}$ denotes a smaller fixed interval.  The scalarized semiclassical channel operators are written $P_{h,P}(\omega)$, and $U_h$ denotes the analytic diagonalizer for the even $2\times2$ system near the trapped region.  The quasibound poles are $\omega_{\ell,n,P}$, indexed by integers $n\in\mathcal N_{\ell,P}(I)$.  The quantities $\mathscr S_P$, $\mathscr J_P$, and $d_P$ are respectively the well action, the tunnelling action, and the Agmon distance; $\vartheta_P$ is the subprincipal phase correction; $\Gamma_P$ is the prefactor in the width formula; and
\[
   \Pi_{\ell,n,P}:=-\Res_{\omega=\omega_{\ell,n,P}}R_{\ell,P}(\omega)
\]
is the residue projector.  Dyadic packets of poles are denoted by $\mathfrak P_j$.

\medskip\noindent\textbf{Asymptotic notation and local proof symbols.}\par\smallskip\noindent
We write $\angles{x}=(1+x^2)^{1/2}$, in particular $\angles{\ell}=(1+\ell^2)^{1/2}$.  The notation $A\lesssim B$ means $A\le CB$ for a constant depending only on the background parameters, the Proca mass $\mu$, the chosen compact set $K$, and finitely many cutoff norms.  All pointwise estimates are uniform for $r,r'\in K$.  Our inverse Fourier transform uses the phase $\ee^{-\ii\om t}$.  The remaining proof-local symbols, such as the radial current $\mathfrak J$, the threshold cutoffs $\eta_{t,\pm}$, and the bounded coefficient matrices $B_{\ell,j}$ and $E_{\ell,j}$, are defined again at first appearance and are not reused outside their local arguments.

\subsection{Main results}

The results come in two layers.  At fixed angular momentum we develop the threshold spectral theory and read off the late-time tails from the branch cut.  We then lift this modewise information to the full field, where one must control both large angular momenta and the quasibound family generated by stable timelike trapping.

The first structural fact is the asymptotic polarization splitting at spatial infinity.  If $(u_2,u_3)$ denotes the standard even pair in the vector-spherical-harmonic decomposition, then the constant combinations
\begin{equation}\label{eq:intro_even_pol_basis}
   v_{-1}=\frac{\ell u_2+u_3}{2\ell+1},
   \qquad
   v_{+1}=\frac{(\ell+1)u_2-u_3}{2\ell+1}
\end{equation}
diagonalize the leading $r^{-2}$ even matrix exactly.  The three resulting effective angular momenta are
\begin{equation}\label{eq:intro_Lvalues}
   L_{-1}=\ell-1,
   \qquad
   L_0=\ell,
   \qquad
   L_{+1}=\ell+1.
\end{equation}
In this basis the charge $Q$ enters only one order shorter range, at $r^{-4}$, so the leading threshold dynamics separate cleanly into three channels.

The natural late-time object is the branch-cut piece after the discrete poles have been removed:
\[
   u_{\ell m}(t)=u^{\mathrm{poles}}_{\ell m}(t)+u^{\bc}_{\ell m}(t).
\]
This distinction is essential in the massive problem.  The oscillatory tails come from the cut $[-\mu,\mu]$, while quasinormal and quasibound states belong to the meromorphic part.  All explicit tail statements below are therefore formulated for $u^{\bc}_{\ell m,P}$.

We begin with the fixed-mode spectral theorem, which drives every later time-domain statement.

\begin{theorem}[Fixed-mode spectral and threshold theorem]\label{thm:spectral_package}
Fix a subextremal Reissner--Nordstr\"om exterior \eqref{eq:intro_subextremal} and an angular momentum $\ell\ge1$.  For each polarization $P\in\{-1,0,+1\}$ there exists a threshold index
\[
   \nu_{\ell,P}>-\frac12
\]
such that the following hold.
\begin{enumerate}[label=\textnormal{(\roman*)}]
   \item The fixed-mode cut-off resolvent extends meromorphically from $\{\Im\om>0\}$ to a slit strip
   \[
      \{\abs{\Im\om}<\eta\}\setminus[-\mu,\mu]
   \]
   for some $\eta>0$, with at most finitely many poles in compact subsets.
   \item There is no real pole and no threshold resonance at $\om=\pm\mu$.
   \item In the small-Coulomb-parameter regime $\kap=M\mu^2/\vp\ll1$, with $\vp=\sqrt{\mu^2-\om^2}$, the threshold discontinuity obeys
   \begin{equation}\label{eq:intro_small_kappa_general}
      \disc\,\cG_{\ell,P}(\om;r,r')
      =
      a_{\ell,P}(r,r')\,\vp^{2\nu_{\ell,P}}
      +O\!\bigl(\kap\,\vp^{2\nu_{\ell,P}}\bigr)
      +O\!\bigl(\vp^{2\nu_{\ell,P}+2}\bigr)
   \end{equation}
   uniformly for $r,r'$ in compact subsets of $(r_+,\infty)$.
   \item In the large-Coulomb-parameter regime $\kap\gg1$, the threshold discontinuity obeys
   \begin{equation}\label{eq:intro_large_kappa_general}
      \disc\,\cG_{\ell,P}(\om;r,r')
      =
      b^+_{\ell,P}(r,r')\,\ee^{2\pi\ii\kap}
      +b^-_{\ell,P}(r,r')\,\ee^{-2\pi\ii\kap}
      +O(\kap^{-1})+O(\vp)
   \end{equation}
   uniformly for $r,r'$ in compact subsets of $(r_+,\infty)$.
   \item The threshold indices satisfy
   \begin{equation}\label{eq:intro_nu_smallmass}
      \nu_{\ell,P}=L_P+\frac12+O\bigl((M\mu)^2+(Q\mu)^2\bigr)
   \end{equation}
   in the small-mass regime $(M\mu)^2+(Q\mu)^2\ll1$.
\end{enumerate}
\end{theorem}

The time-domain tail theorem is then obtained by oscillatory inversion of the branch-cut jump.

\begin{theorem}[Modewise branch-cut tails on subextremal Reissner--Nordstr\"om]\label{thm:main}
Fix $\ell\ge1$ and let $u^{\bc}_{\ell m,P}(t,r,r')$ be the branch-cut contribution for one fixed $(\ell,m)$ and one polarization $P\in\{-1,0,+1\}$.

\smallskip
\noindent\textbf{(a) Intermediate regime.}
If
\begin{equation}\label{eq:intro_kappastar}
   \kap_*(t):=M\mu^{3/2}t^{1/2}\to0,
\end{equation}
then
\begin{equation}\label{eq:intro_intermediate_general}
   u^{\bc}_{\ell m,P}(t,r,r')
   =
   A_{\ell,P}(r,r';Q)\,t^{-(\nu_{\ell,P}+1)}\sin(\mu t+\delta_{\ell,P}(Q))
   +O\!\bigl(\kap_*(t)t^{-(\nu_{\ell,P}+1)}\bigr)
   +O\!\bigl(t^{-(\nu_{\ell,P}+2)}\bigr).
\end{equation}

\smallskip
\noindent\textbf{(b) Small-mass explicit exponents.}
If in addition
\[
   \eps_{\mu,Q}:=(M\mu)^2+(Q\mu)^2\ll1,
\]
then the threshold indices obey \eqref{eq:intro_nu_smallmass}; hence the leading intermediate decay exponents reduce to
\[
   \ell+\frac12,\qquad
   \ell+\frac32,\qquad
   \ell+\frac52
\]
for $P=-1,0,+1$, respectively.

\smallskip
\noindent\textbf{(c) Very-late regime.}
Let
\begin{equation}\label{eq:intro_vp0}
   \vp_0(t):=\Bigl(\frac{2\pi M\mu^3}{t}\Bigr)^{1/3},
   \qquad
   \kap_0(t):=\frac{M\mu^2}{\vp_0(t)}.
\end{equation}
If $\kap_0(t)\to\infty$, then
\begin{align}\label{eq:intro_late_general}
   u^{\bc}_{\ell m,P}(t,r,r')
   &=
   B_{\ell,P}(r,r';Q)\,t^{-5/6}
   \sin\!\Bigl(
      \mu t
      -\frac32(2\pi M\mu)^{2/3}(\mu t)^{1/3}
      \notag\\
   &\qquad\qquad
      +\delta_{\ell,P,0}(Q)
      +O(\kap_0(t)^{-1}+\vp_0(t))
   \Bigr)
   \notag\\
   &\qquad
   +O\!\bigl(\kap_0(t)^{-1}+\vp_0(t)\bigr)\,t^{-5/6}.
\end{align}
The exponent $5/6$ is independent of $\ell$, of the polarization, and of $Q$.
\end{theorem}

\begin{remark}\label{rem:scope_intro}
Theorem~\ref{thm:main} is deliberately stated for the branch-cut contribution after subtraction of discrete pole residues.  This is the correct fixed-mode formulation in the massive problem, because long-lived quasibound poles may dominate the full signal before the cut contribution becomes visible.
\end{remark}

\medskip\noindent\textbf{Large angular momenta and branch-cut full-field decay.}\par\smallskip\noindent
The summed angular-momentum problem is the point of contact with \cite{ShlapentokhRothmanVanDeMoortel2026}.  In the scalar theory, one proves a compact-region polynomial decay theorem for the full field after overcoming the obstruction created by stable timelike trapping.  For Proca, the argument naturally splits into a continuous spectral part and a discrete quasibound part.  The first ingredient is a uniform large-$\ell$ bound for the branch-cut kernels; the second is a semiclassical description and summation of the quasibound resonance family.

\begin{theorem}[Uniform angular summability of branch-cut kernels]\label{thm:uniform_angular}
For every compact $K\Subset(r_+,\infty)$ there exist an integer $N_0\ge0$ and a constant $C_K$ such that, for every admissible $(\ell,m,P)$,
\begin{equation}\label{eq:uniform_angular_kernel_bound}
   \sup_{r,r'\in K}\abs{u_{\ell m,P}^{\bc}(t,r,r')}
   \le
   C_K\angles{\ell}^{N_0}
   \begin{cases}
      t^{-(\nu_{\ell,P}+1)}, & \kap_*(t)\le1,\\[0.5ex]
      t^{-5/6}, & \kap_0(t)\ge1,
   \end{cases}
\end{equation}
with the same intermediate and very-late remainder structure established in Section~\ref{sec:threshold_to_time}, uniformly in $(\ell,m,P)$.  Moreover,
\begin{equation}\label{eq:uniform_angular_amplitudes}
   \sup_{r,r'\in K}
   \Bigl(
      \abs{A_{\ell,P}(r,r';Q)}
      +
      \abs{B_{\ell,P}(r,r';Q)}
   \Bigr)
   \le C_K\angles{\ell}^{N_0}.
\end{equation}
\end{theorem}

\begin{theorem}[Full-field pointwise branch-cut decay]\label{thm:full_field_pointwise}
Let $A^{\bc}$ be generated by smooth compactly supported initial data.  Let $N_0$ be the integer from Theorem~\ref{thm:uniform_angular}.  For every compact $K\Subset(r_+,\infty)$ and every integer $N>N_0+2$ there exists $C_{K,N}$ such that
\begin{equation}\label{eq:full_field_intermediate}
   \sup_{r\in K,\,\omega\in S^2}\abs{A^{\bc}(t,r,\omega)}
   \le
   C_{K,N}\,\mathcal E_N[A[0]]^{1/2}\,t^{-(\nu_*+1)},
   \qquad
   \kap_*(t)\le1.
\end{equation}
where
\[
   \nu_*:=\inf_{\ell\ge1,\ P\in\{-1,0,+1\}}\nu_{\ell,P}.
\]
If $\kap_0(t)\ge1$, then
\begin{equation}\label{eq:full_field_late}
   \sup_{r\in K,\,\omega\in S^2}\abs{A^{\bc}(t,r,\omega)}
   \le
   C_{K,N}\,\mathcal E_N[A[0]]^{1/2}\,t^{-5/6}.
\end{equation}
In the Schwarzschild case, $\nu_*=1/2$ and the intermediate full-field bound is $t^{-3/2}$; in the small-mass Reissner--Nordstr\"om regime, $\nu_*=1/2+O((M\mu)^2+(Q\mu)^2)$.
\end{theorem}

\begin{corollary}[Polynomial time decay for the radiative Proca field]\label{cor:poly_decay_radiative}
Define
\[
   \gamma_*:=\min\Bigl\{\nu_*+1,\frac56\Bigr\}.
\]
Then for every compact $K\Subset(r_+,\infty)$ and every integer $N>N_0+2$ there exists $C_{K,N}$ such that
\begin{equation}\label{eq:intro_poly_decay_radiative}
   \sup_{r\in K,\,\omega\in S^2}\abs{A^{\bc}(t,r,\omega)}
   \le
   C_{K,N}\,\mathcal E_N[A[0]]^{1/2}\,t^{-\gamma_*}.
\end{equation}
In the Schwarzschild case one has $\gamma_*=5/6$.  More generally, if $(M\mu)^2+(Q\mu)^2$ is sufficiently small, then $\gamma_*=5/6$ as well.
\end{corollary}

The corollary isolates the continuous spectral contribution.  The unsplit full-field theorem proved later in Theorem~\ref{thm:full_field_unsplit} adds the quasibound resonance family and upgrades this branch-cut estimate to a decay theorem for the whole Proca field.

\medskip\noindent\textbf{Quasibound poles, residue bounds, and unsplit full-field decay.}\par\smallskip\noindent
One of the central lessons of \cite{ShlapentokhRothmanVanDeMoortel2026} is that stable timelike trapping produces discrete bad frequencies and an unbounded Fourier transform even though pointwise decay still holds.  The same geometric issue is present for massive spin-$1$ fields.  The present paper proves the corresponding discrete spectral theorem as well: after mode stability has ruled out upper-half-plane and threshold pathologies, the lower-half-plane quasibound family is constructed semiclassically, its residues are controlled uniformly, and the resulting time series is summed to recover the unsplit full field.  The starting point is the following mode-stability theorem.

\begin{theorem}[Mode stability and threshold exclusion]\label{thm:mode_stability}
Fix $\ell\ge1$.  There is no nontrivial separated solution of the form
\[
   A=\ee^{-\ii\om t}\widehat A(r,\omega)
\]
which is ingoing at the future event horizon and outgoing or decaying at infinity, with frequency $\om$ in the closed upper half-plane.  More precisely:
\begin{enumerate}[label=\textnormal{(\roman*)}]
   \item if $\Im\om>0$, no such mode exists;
   \item if $\om\in[-\mu,\mu]$, there is no real cut pole and no threshold resonance at $\om=\pm\mu$.
\end{enumerate}
Consequently the Evans determinant $\cE_\ell$ is zero-free on $\{\Im\om>0\}$ and zero-free in punctured threshold neighborhoods of $\om=\pm\mu$.
\end{theorem}

The continuous spectral results above are complemented by a discrete resonance package.  Let
\[
   A(t)=A^{\bc}(t)+A^{\qb}(t)+A^{\fast}(t)
\]
denote the decomposition of the full Proca field into the branch-cut contribution, the quasibound pole contribution, and the exponentially decaying remainder coming from poles and contour terms bounded away from the real axis.

\begin{theorem}[Quasibound branches and Bohr--Sommerfeld quantization]\label{thm:qb_branches}
Fix a compact energy interval $I\Subset(0,\mu)$ contained in the stably trapped regime.  There exist $\ell_{\mathrm{sc}}\in\N$ and $c_I>0$ such that for every $\ell\ge \ell_{\mathrm{sc}}$ and every polarization $P\in\{-1,0,+1\}$, the poles of the fixed-mode resolvent with
\[
   \Re\omega\in I,
   \qquad
   0>\Im\omega>-\ee^{-c_I\ell}
\]
are simple and may be indexed by integers $n\in\mathcal N_{\ell,P}(I)$ as $\omega_{\ell,n,P}$.  Writing $h=(\ell+\tfrac12)^{-1}$, one has the Bohr--Sommerfeld law
\[
   \mathscr S_P(\Re\omega_{\ell,n,P};h)
   =
   2\pi h\Bigl(n+\frac12\Bigr)
   +h\,\vartheta_P(\Re\omega_{\ell,n,P};h)
   +O(h^2),
\]
with $\vartheta_P=O(1)$, and the tunnelling width formula
\[
   \Im\omega_{\ell,n,P}
   =
   -\Gamma_P(\Re\omega_{\ell,n,P};h)
   \exp\!\Bigl(-\frac{2\mathscr J_P(\Re\omega_{\ell,n,P};h)}{h}\Bigr)
   (1+O(h)).
\]
\end{theorem}

\begin{theorem}[Residue projectors and Agmon localization]\label{thm:qb_residues}
For every compact radial set $K\Subset(r_+,\infty)$ there exist integers $N_{\qb},N_{\mathrm{rec}}\ge0$ and a constant $C_K$ such that every quasibound residue projector
\[
   \Pi_{\ell,n,P}:=-\Res_{\omega=\omega_{\ell,n,P}}R_{\ell,P}(\omega)
\]
obeys
\[
   \sup_{r,r'\in K}\abs{\Pi_{\ell,n,P}(r,r')}
   \le
   C_K\angles{\ell}^{N_{\qb}},
\]
and, for compactly supported initial data,
\[
   \sup_{r\in K,\,\omega\in S^2}\abs{\Pi_{\ell,n,P}A[0](r,\omega)}
   \le
   C_{K,N}\angles{\ell}^{N_{\mathrm{rec}}}\mathcal E_N[A[0]]^{1/2}.
\]
If $K$ is disjoint from the classically allowed well corresponding to $\Re\omega_{\ell,n,P}$, then the residue kernel satisfies an Agmon-type exponential bound on $K\times K$.
\end{theorem}

\begin{theorem}[Self-contained summed quasibound decay]\label{thm:qb_sum}
For smooth compactly supported initial data and every logarithmic power $L>0$, there exists an integer $N_{\qb}^{\mathrm{sum}}(L)$ such that for every compact $K\Subset(r_+,\infty)$,
\[
   \sup_{r\in K,\,\omega\in S^2}\abs{A^{\qb}(t,r,\omega)}
   \le
   C_{K,N,L}\,\mathcal E_N[A[0]]^{1/2}\,(\log(2+t))^{-L}
\]
for all $N\ge N_{\qb}^{\mathrm{sum}}(L)$ and all $t\ge2$.
\end{theorem}

\begin{theorem}[Self-contained full-field decay]\label{thm:full_field_unsplit}
For smooth compactly supported initial data, every logarithmic power $L>0$, and every compact $K\Subset(r_+,\infty)$, there exists $N_{\mathrm{full}}(L)$ such that
\[
   \sup_{r\in K,\,\omega\in S^2}\abs{A(t,r,\omega)}
   \le
   C_{K,N,L}\,\mathcal E_N[A[0]]^{1/2}\,
   \Bigl(t^{-\gamma_*}+(\log(2+t))^{-L}\Bigr)
\]
for all $N\ge N_{\mathrm{full}}(L)$ and all $t\ge2$, where $\gamma_*$ is the branch-cut exponent from Corollary~\ref{cor:poly_decay_radiative}.  In particular,
\[
   \sup_{r\in K,\,\omega\in S^2}\abs{A(t,r,\omega)}
   \le
   C_{K,N,L}\,\mathcal E_N[A[0]]^{1/2}\,(\log(2+t))^{-L}.
\]
\end{theorem}

\begin{theorem}[Two-regime asymptotic expansion of the full Proca field]\label{thm:full_field_asymptotic_expansion}
Fix a compact set $K\Subset(r_+,\infty)$ and let $A$ be generated by smooth compactly supported initial data.  Then there exist an integer $\ell_{\mathrm{as}}\ge1$, finite-rank fields
\[
   A^{\bc}_{\mathrm{int,low}}(t),
   \qquad
   A^{\bc}_{\mathrm{late,low}}(t),
\]
coming from the finitely many modes $0\le\ell<\ell_{\mathrm{as}}$ and depending linearly on $A[0]$, and an integer $N_{\mathrm{as}}>N_0+2$ such that for every $N\ge N_{\mathrm{as}}$ there exist coefficient functions, linear in $A[0]$,
\[
   \mathcal A_{\ell m,P}[A[0]](r),
   \qquad
   \mathcal B_{\ell m,P}[A[0]](r),
   \qquad
   \ell\ge\ell_{\mathrm{as}},\quad |m|\le\ell,\quad P\in\{-1,0,+1\},
\]
with
\begin{equation}\label{eq:intro_asymptotic_coeff_bounds}
   \sup_{r\in K}
   \Bigl(
      \abs{\mathcal A_{\ell m,P}[A[0]](r)}
      +
      \abs{\mathcal B_{\ell m,P}[A[0]](r)}
   \Bigr)
   \le
   C_{K,N}\angles{\ell}^{N_0}\mathcal E_{\ell m,P}[A[0]]^{1/2},
\end{equation}
where $N_0$ is the angular-loss exponent from Theorem~\ref{thm:uniform_angular}.  Define
\begin{align}
   A^{\bc}_{\mathrm{int}}(t,r,\vartheta)
   &:=
   A^{\bc}_{\mathrm{int,low}}(t,r,\vartheta)
   \notag\\
   &\qquad+
   \sum_{\ell\ge\ell_{\mathrm{as}}}\sum_{m=-\ell}^{\ell}\sum_{P\in\{-1,0,+1\}}
   \mathcal A_{\ell m,P}[A[0]](r)
   Y^{(P)}_{\ell m}(\vartheta)
   \,t^{-(\nu_{\ell,P}+1)}
   \sin(\mu t+\delta_{\ell,P}(Q)),
   \label{eq:intro_full_intermediate_profile}
   \\
   A^{\bc}_{\mathrm{late}}(t,r,\vartheta)
   &:=
   A^{\bc}_{\mathrm{late,low}}(t,r,\vartheta)
   \notag\\
   &\qquad+
   t^{-5/6}
   \sum_{\ell\ge\ell_{\mathrm{as}}}\sum_{m=-\ell}^{\ell}\sum_{P\in\{-1,0,+1\}}
   \mathcal B_{\ell m,P}[A[0]](r)
   Y^{(P)}_{\ell m}(\vartheta)
   \notag\\
   &\qquad\qquad\times
   \sin\!\Bigl(
      \mu t
      -\frac32(2\pi M\mu)^{2/3}(\mu t)^{1/3}
      +\delta_{\ell,P,0}(Q)
   \Bigr).
   \label{eq:intro_full_late_profile}
\end{align}
Then both series converge in $C^0(K\times S^2)$ and satisfy
\begin{align}
   \sup_{r\in K,\,\vartheta\in S^2}
   \abs{A^{\bc}_{\mathrm{int}}(t,r,\vartheta)}
   &\le
   C_{K,N}\mathcal E_N[A[0]]^{1/2}t^{-(\nu_*+1)},
   \label{eq:intro_profile_bounds_a}
   \\
   \sup_{r\in K,\,\vartheta\in S^2}
   \abs{A^{\bc}_{\mathrm{late}}(t,r,\vartheta)}
   &\le
   C_{K,N}\mathcal E_N[A[0]]^{1/2}t^{-5/6}.
   \label{eq:intro_profile_bounds_b}
\end{align}
If $\kap_*(t)\le1$, then
\begin{equation}\label{eq:intro_full_intermediate_expansion}
   A(t)=A^{\bc}_{\mathrm{int}}(t)+A^{\qb}(t)+A^{\fast}(t)+R_{\mathrm{int}}(t)
\end{equation}
on $K$, with remainder bound
\begin{equation}\label{eq:intro_full_intermediate_remainder}
   \sup_{r\in K,\,\vartheta\in S^2}\abs{R_{\mathrm{int}}(t,r,\vartheta)}
   \le
   C_{K,N}\mathcal E_N[A[0]]^{1/2}
   \Bigl(
      \kap_*(t)t^{-(\nu_*+1)}
      +
      t^{-(\nu_*+2)}
   \Bigr).
\end{equation}
If $\kap_0(t)\ge1$, then
\begin{equation}\label{eq:intro_full_late_expansion}
   A(t)=A^{\bc}_{\mathrm{late}}(t)+A^{\qb}(t)+A^{\fast}(t)+R_{\mathrm{late}}(t)
\end{equation}
on $K$, with remainder bound
\begin{equation}\label{eq:intro_full_late_remainder}
   \sup_{r\in K,\,\vartheta\in S^2}\abs{R_{\mathrm{late}}(t,r,\vartheta)}
   \le
   C_{K,N}\mathcal E_N[A[0]]^{1/2}
   \bigl(\kap_0(t)^{-1}+\vp_0(t)\bigr)t^{-5/6}.
\end{equation}
Thus the full solution admits an explicit radiative asymptotic expansion in both time regimes, with the only additional late-time obstruction carried by the separately controlled quasibound term $A^{\qb}$.
\end{theorem}

\begin{corollary}[Explicit leading coefficient fields]\label{cor:explicit_leading_coefficients}
Fix a compact set $K\Subset(r_+,\infty)$ and let $A$ be generated by smooth compactly supported initial data.  Let
\[
   \widetilde{\mathcal A}_{\ell m,P}[A[0]](r),
   \qquad
   \widetilde{\mathcal B}_{\ell m,P}[A[0]](r),
\]
denote the full modal coefficient families furnished by the proof of Theorem~\ref{thm:full_field_asymptotic_expansion}; these depend linearly on $A[0]$, and for $\ell=0$ only the electric monopole $P=+1$ occurs.  Then there exist a finite nonempty set
\[
   \Sigma_*:=\{(\ell,P):\ell\ge1,\ P\in\{-1,0,+1\},\ \nu_{\ell,P}=\nu_*\}
\]
and a number $\rho_*>0$ such that
\[
   \nu_{\ell,P}\ge\nu_*+\rho_*
   \qquad
   \text{whenever }(\ell,P)\notin\Sigma_*.
\]
Define the intermediate coefficient fields
\begin{align}
   \cS_*(r,\vartheta)
   &:=
   \sum_{(\ell,P)\in\Sigma_*}\sum_{m=-\ell}^{\ell}
   \widetilde{\mathcal A}_{\ell m,P}[A[0]](r)
   \cos\delta_{\ell,P}(Q)\,
   Y_{\ell m}^{(P)}(\vartheta),
   \label{eq:intro_explicit_Sstar}
   \\
   \cC_*(r,\vartheta)
   &:=
   \sum_{(\ell,P)\in\Sigma_*}\sum_{m=-\ell}^{\ell}
   \widetilde{\mathcal A}_{\ell m,P}[A[0]](r)
   \sin\delta_{\ell,P}(Q)\,
   Y_{\ell m}^{(P)}(\vartheta),
   \label{eq:intro_explicit_Cstar}
\end{align}
and the very-late coefficient fields
\begin{align}
   \cS_{\mathrm{late}}(r,\vartheta)
   &:=
   \sum_{\ell\ge0}\sum_{m=-\ell}^{\ell}\sum_{P\in\mathcal P_\ell}
   \widetilde{\mathcal B}_{\ell m,P}[A[0]](r)
   \cos\delta_{\ell,P,0}(Q)\,
   Y_{\ell m}^{(P)}(\vartheta),
   \label{eq:intro_explicit_Slate}
   \\
   \cC_{\mathrm{late}}(r,\vartheta)
   &:=
   \sum_{\ell\ge0}\sum_{m=-\ell}^{\ell}\sum_{P\in\mathcal P_\ell}
   \widetilde{\mathcal B}_{\ell m,P}[A[0]](r)
   \sin\delta_{\ell,P,0}(Q)\,
   Y_{\ell m}^{(P)}(\vartheta),
   \label{eq:intro_explicit_Clate}
\end{align}
where $\mathcal P_0=\{+1\}$ and $\mathcal P_\ell=\{-1,0,+1\}$ for $\ell\ge1$.  Then $\cS_*$ and $\cC_*$ are finite-rank, the series defining $\cS_{\mathrm{late}}$ and $\cC_{\mathrm{late}}$ converge in $C^0(K\times S^2)$, and there exists $N_{\mathrm{lead}}$ such that for every $N\ge N_{\mathrm{lead}}$,
\begin{align}
   \sup_{r\in K,\,\vartheta\in S^2}
   \Bigl(
      \abs{\cS_*(r,\vartheta)}
      +\abs{\cC_*(r,\vartheta)}
      +\abs{\cS_{\mathrm{late}}(r,\vartheta)}
      +\abs{\cC_{\mathrm{late}}(r,\vartheta)}
   \Bigr)
   &\le
   C_{K,N}\mathcal E_N[A[0]]^{1/2},
   \label{eq:intro_explicit_coeff_bound}
\end{align}
and, on $K$,
\begin{align}
   A(t)
   &=
   t^{-(\nu_*+1)}
   \bigl(
      \cS_*(r,\vartheta)\sin(\mu t)
      +
      \cC_*(r,\vartheta)\cos(\mu t)
   \bigr)
   +A^{\qb}(t)+A^{\fast}(t)+\widetilde R_{\mathrm{int}}(t),
   \label{eq:intro_explicit_intermediate_coeff}
   \\
   A(t)
   &=
   t^{-5/6}
   \bigl(
      \cS_{\mathrm{late}}(r,\vartheta)\sin\Theta(t)
      +
      \cC_{\mathrm{late}}(r,\vartheta)\cos\Theta(t)
   \bigr)
   +A^{\qb}(t)+A^{\fast}(t)+\widetilde R_{\mathrm{late}}(t),
   \label{eq:intro_explicit_late_coeff}
\end{align}
where
\[
   \Theta(t):=\mu t-\frac32(2\pi M\mu)^{2/3}(\mu t)^{1/3},
   \qquad
   \sigma_*:=\min\{1,\rho_*\}.
\]
If $\kap_*(t)\le1$, then
\begin{equation}\label{eq:intro_explicit_intermediate_remainder}
   \sup_{r\in K,\,\vartheta\in S^2}
   \abs{\widetilde R_{\mathrm{int}}(t,r,\vartheta)}
   \le
   C_{K,N}\mathcal E_N[A[0]]^{1/2}
   \Bigl(
      \kap_*(t)t^{-(\nu_*+1)}
      +
      t^{-(\nu_*+1+\sigma_*)}
   \Bigr).
\end{equation}
If $\kap_0(t)\ge1$, then
\begin{equation}\label{eq:intro_explicit_late_remainder}
   \sup_{r\in K,\,\vartheta\in S^2}
   \abs{\widetilde R_{\mathrm{late}}(t,r,\vartheta)}
   \le
   C_{K,N}\mathcal E_N[A[0]]^{1/2}
   \bigl(\kap_0(t)^{-1}+\vp_0(t)\bigr)t^{-5/6}.
\end{equation}
In particular, the dominant intermediate coefficient is finite-rank and is supported exactly on the channels with threshold index $\nu_*$.
\end{corollary}

\begin{remark}[Scope of the leading coefficient statement]\label{rem:leading_coefficient_scope}
Corollary~\ref{cor:explicit_leading_coefficients} identifies the leading coefficient fields of the radiative branch-cut contribution after the quasibound term has been separated.  It does not claim that the displayed oscillatory term dominates the unsplit full field for arbitrary data, since the self-contained quasibound estimate is logarithmic.  Moreover the coefficient fields are data dependent and may vanish identically for special initial data; in that case the first nontrivial radiative term is obtained by repeating the same argument with the smallest threshold index actually activated by the data.
\end{remark}

\begin{remark}[Optional external arithmetic upgrade]\label{rem:external_packet_upgrade}
If one supplements the present paper with the abstract arithmetic packet theorem of \cite{ShlapentokhRothmanVanDeMoortel2026}, then the logarithmic quasibound estimate above upgrades to the polynomial bound $t^{-5/6+1/23}$ from the earlier draft, and under the exponent-pair conjecture to $t^{-5/6+\eps}$ for every $\eps>0$.  This stronger refinement is not used anywhere in the self-contained core developed below.
\end{remark}

\subsection{Self-contained summation and proof architecture}

The logic of the paper is easiest to read in four passes.  We begin by reducing the Proca equation to one odd scalar channel and one even $2\times2$ system, and by identifying the polarization basis in which the large-$r$ behavior becomes transparent.  We then construct the fixed-mode resolvent, continue it across the massive branch cut, and extract the threshold expansions that control the late-time signal.  In parallel we develop the semiclassical theory of the quasibound poles created by stable timelike trapping.  The final step is to put the continuous and discrete pieces back together at the full-field level.

A useful feature of the present version is that the discrete summation is now genuinely self-contained.  The only inputs are ones proved in the paper itself: two-sided tunnelling-width bounds and polynomial residue/reconstruction estimates.  On a dyadic packet the tunnelling formula contributes an explicit damping factor $\exp(-ct\ee^{-C2^j})$, while angular regularity of the initial data pays for any negative power of the packet scale $2^j$.  A short dyadic summation lemma then turns these packet bounds into logarithmic decay for the quasibound contribution.  The sharper polynomial rates from the earlier draft remain available only as an external arithmetic upgrade; see Remark~\ref{rem:external_packet_upgrade}.

For readers skimming the overall argument, it is helpful to keep the paper in five modules:
\begin{enumerate}[label=\textnormal{(\arabic*)}]
   \item Section~\ref{sec:reduction} constructs the odd and even channel equations and identifies the asymptotic polarization basis.
   \item Sections~\ref{sec:functional_framework}--\ref{sec:spectral_theorem_proof} develop the operator-theoretic fixed-mode framework and prove the threshold spectral theorem.
   \item Sections~\ref{sec:mode_stability} and \ref{sec:threshold_to_time} convert that spectral theorem into mode stability and explicit branch-cut tails.
   \item Section~\ref{sec:full_field} carries out the large-angular-momentum analysis needed for branch-cut summation on compact radial sets.
   \item Sections~\ref{sec:qb} and \ref{sec:unsplit} construct the quasibound pole family, prove the self-contained packet estimate, sum the discrete series, and establish the unsplit full-field decay theorem.
\end{enumerate}

\subsection{Organization of the paper}

The paper is written so that the fixed-mode theory comes first and the full-field reconstruction comes later.  Sections~\ref{sec:reduction}--\ref{sec:spectral_theorem_proof} set up the odd/even reduction, the selfadjoint channel framework, the Green kernels, and the threshold spectral theorem.  Sections~\ref{sec:mode_stability} and \ref{sec:threshold_to_time} then turn that spectral input into mode stability and explicit intermediate and very-late branch-cut asymptotics.  Section~\ref{sec:full_field} provides the large-angular-momentum analysis needed to sum the branch-cut contribution on compact radial sets.  Sections~\ref{sec:qb} and \ref{sec:unsplit} develop the semiclassical quasibound resonance theory, prove the residue and reconstruction bounds, derive the self-contained packet estimate, and combine the discrete and continuous pieces to obtain decay for the unsplit Proca field.  Section~\ref{sec:future} closes with the directions that lie beyond the static subextremal setting.

\subsection{Proof dependency diagram}

The diagram below records the main logical dependencies in a way that is easy to check while reading or refereeing the paper:

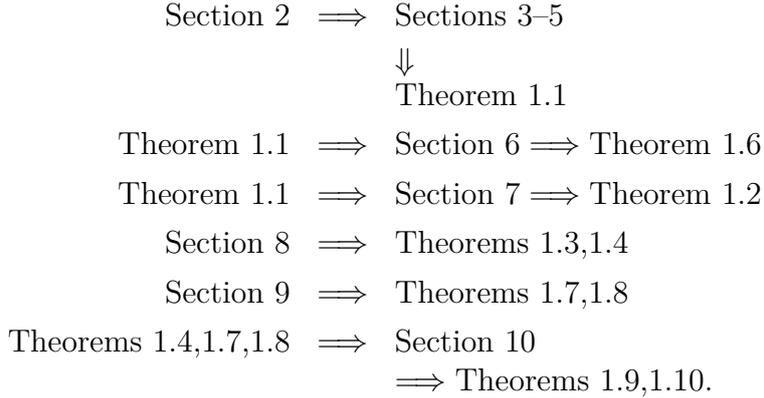
\begin{equation*}
\begin{array}{rcl}
\text{Section~\ref{sec:reduction}}
&\Longrightarrow&
\text{Sections~\ref{sec:functional_framework}--\ref{sec:spectral_theorem_proof}}
\\[0.5ex]
&&\Downarrow
\\[-0.3ex]
&&\text{Theorem~\ref{thm:spectral_package}}
\\[0.8ex]
\text{Theorem~\ref{thm:spectral_package}}
&\Longrightarrow&
\text{Section~\ref{sec:mode_stability}}\Longrightarrow\text{Theorem~\ref{thm:mode_stability}}
\\[0.8ex]
\text{Theorem~\ref{thm:spectral_package}}
&\Longrightarrow&
\text{Section~\ref{sec:threshold_to_time}}\Longrightarrow\text{Theorem~\ref{thm:main}}
\\[0.8ex]
\text{Section~\ref{sec:full_field}}
&\Longrightarrow&
\text{Theorems~\ref{thm:uniform_angular},\ref{thm:full_field_pointwise}}
\\[0.8ex]
\text{Section~\ref{sec:qb}}
&\Longrightarrow&
\text{Theorems~\ref{thm:qb_branches},\ref{thm:qb_residues}}
\\[0.8ex]
\text{Theorems~\ref{thm:full_field_pointwise},\ref{thm:qb_branches},\ref{thm:qb_residues}}
&\Longrightarrow&
\text{Section~\ref{sec:unsplit}}
\\
&&\Longrightarrow\text{Theorems~\ref{thm:qb_sum},\ref{thm:full_field_unsplit}}.
\end{array}
\end{equation*}

In other words, the paper has four real bottlenecks: the fixed-mode spectral theorem, the time-domain inversion of the branch cut, the large-angular-momentum uniformity theorem, and the quasibound summation theorem.  Everything else is arranged to feed into one of these points or to draw consequences from them.

\subsection{Appendix guide}

The appendices are grouped by function rather than by the order in which they are used.  Appendix~\ref{app:vsh} records the vector-spherical-harmonic reduction together with the boundary constructions at the horizon and infinity.  Appendix~\ref{app:special_functions} collects the threshold model equations, special-function asymptotics, oscillatory lemmas, and fixed-mode technical complements.  Appendix~\ref{app:large_ell} gathers the large-angular-momentum, reconstruction, local-energy, and dyadic bookkeeping estimates used in the full-field argument.  Appendix~\ref{app:small_mass} treats the small-mass regime, the monopole and static-threshold input, and the cutoff-independence argument for the branch-cut field on compact radial sets.

\section{Geometry, mode reduction, and polarization splitting}\label{sec:reduction}

\subsection{Geometry and the field equation}

We begin by fixing the background geometry and writing the Proca system in the form that will be used throughout the paper.

\begin{equation}\label{eq:rn_metric}
\begin{aligned}
   g_{\RN}
   &=
   -f(r)\,\dd t^2
   +f(r)^{-1}\,\dd r^2
   +r^2\dd\Omega^2,
\\
   f(r)&=1-\frac{2M}{r}+\frac{Q^2}{r^2}
   =
   \frac{(r-r_+)(r-r_-)}{r^2},
   \qquad
   r>r_+.
\end{aligned}
\end{equation}
with
\[
   r_\pm=M\pm\sqrt{M^2-Q^2},
   \qquad
   \kappa_+:=\frac{r_+-r_-}{2r_+^2}>0.
\]
The tortoise coordinate is defined by $\dd r_*=f^{-1}\dd r$.  Near the event horizon
\begin{equation}\label{eq:rn_tortoise_horizon}
   r_*
   =
   \frac{1}{2\kappa_+}\log(r-r_+)+O(1),
   \qquad r\downarrow r_+,
\end{equation}
whereas at infinity
\begin{equation}\label{eq:rn_tortoise_infty}
   r_*=r+2M\log r+O(1).
\end{equation}
The massive Maxwell--Proca equation is
\begin{equation}\label{eq:proca_covariant}
   \nabla^\mu F_{\mu\nu}-\mu^2A_\nu=0,
   \qquad
   F=\dd A,
   \qquad
   \mu>0.
\end{equation}
Taking the divergence and using antisymmetry of $F$ gives the forced Lorenz constraint
\begin{equation}\label{eq:forced_lorenz}
   \nabla^\nu A_\nu=0,
\end{equation}
so there is no gauge freedom in the massive theory.

\subsection{Reduced odd/even mode equations}

Unless stated otherwise, we work here with $\ell\ge1$; the monopole is postponed because only the even electric channel survives there.

We use the standard vector spherical harmonic decomposition.  For each $(\ell,m)$ the odd sector is carried by one scalar amplitude $u_4(t,r)$, while the even sector is encoded by a pair $(u_2(t,r),u_3(t,r))$.  Writing $\la^2=\ell(\ell+1)$, the neutral Proca equations on a static spherically symmetric background reduce to the same odd/even pattern as in Schwarzschild, with the curvature term
\[
   1-\frac{3M}{r}
\]
replaced by the Reissner--Nordstr\"om combination
\[
   f(r)-\frac{r f'(r)}{2}
   =
   1-\frac{3M}{r}+\frac{2Q^2}{r^2}.
\]
The reduced time-domain equations are therefore

\begin{equation}\label{eq:odd_time_domain}
   \Bigl(
      -\partial_t^2+\partial_{r_*}^2
      -f\Bigl(\mu^2+\frac{\la^2}{r^2}\Bigr)
   \Bigr)u_4=0
\end{equation}
and
\begin{align}
   \Bigl(
      -\partial_t^2+\partial_{r_*}^2
      -f\Bigl(\mu^2+\frac{\la^2}{r^2}\Bigr)
   \Bigr)u_2
   -\frac{2f}{r^2}\Bigl(f-\frac{r f'}{2}\Bigr)(u_2-u_3)&=0,
   \label{eq:even_time_domain_1}
   \\
   \Bigl(
      -\partial_t^2+\partial_{r_*}^2
      -f\Bigl(\mu^2+\frac{\la^2}{r^2}\Bigr)
   \Bigr)u_3
   +\frac{2f\la^2}{r^2}u_2&=0.
   \label{eq:even_time_domain_2}
\end{align}
After Fourier transformation with time dependence $\ee^{-\ii\om t}$, these become
\begin{equation}\label{eq:odd_frequency_domain}
   u_4''+\Bigl(\om^2-f\mu^2-f\frac{\la^2}{r^2}\Bigr)u_4=0
\end{equation}
and
\begin{align}
   u_2''+\Bigl(\om^2-f\mu^2-f\frac{\la^2}{r^2}\Bigr)u_2
   -\frac{2f}{r^2}\Bigl(f-\frac{r f'}{2}\Bigr)(u_2-u_3)&=0,
   \label{eq:even_frequency_domain_1}
   \\
   u_3''+\Bigl(\om^2-f\mu^2-f\frac{\la^2}{r^2}\Bigr)u_3
   +\frac{2f\la^2}{r^2}u_2&=0,
   \label{eq:even_frequency_domain_2}
\end{align}
where primes denote $\partial_{r_*}$.  The reduction is consistent with the Schwarzschild formulas of Rosa--Dolan \cite{RosaDolan2012}, and the fact that Proca separates in the full Kerr--Newman family shows that this static RN model is the correct zero-rotation limit of the general charged black-hole problem \cite{CayusoDiasGrayKubiznakMargalitSantosSouzaThiele2020}.

\subsection{Exact asymptotic polarization diagonalization}

The odd sector already gives one physical polarization; we denote it by $P=0$.  The even sector is more interesting.  The right variables are not the original pair $(u_2,u_3)$, but the constant combinations below.  In those coordinates the leading inverse-square term becomes diagonal, and the three effective angular momenta $\ell-1$, $\ell$, and $\ell+1$ appear transparently.

\begin{equation}\label{eq:odd_pol}
   v_0:=u_4,
   \qquad
   L_0:=\ell.
\end{equation}
The even sector is diagonalized at leading order by the same constant transformation as in Schwarzschild:
\begin{equation}\label{eq:T_matrix}
   T_\ell=
   \begin{pmatrix}
      1&1\\
      \ell+1&-\ell
   \end{pmatrix},
   \qquad
   T_\ell^{-1}
   =
   \frac{1}{2\ell+1}
   \begin{pmatrix}
      \ell&1\\
      \ell+1&-1
   \end{pmatrix},
\end{equation}
with
\begin{equation}\label{eq:polarization_variables}
   \binom{v_{-1}}{v_{+1}}
   =
   T_\ell^{-1}\binom{u_2}{u_3}
   =
   \frac{1}{2\ell+1}
   \binom{\ell u_2+u_3}{(\ell+1)u_2-u_3}.
\end{equation}
Equivalently,
\begin{equation}\label{eq:inverse_polarization_variables}
   u_2=v_{-1}+v_{+1},
   \qquad
   u_3=(\ell+1)v_{-1}-\ell v_{+1}.
\end{equation}

\begin{proposition}[Reissner--Nordstr\"om even-sector polarization form]\label{prop:even_polarization_form}
For every $\ell\ge1$, the even frequency-domain system
\eqref{eq:even_frequency_domain_1}--\eqref{eq:even_frequency_domain_2} takes the matrix form
\begin{equation}\label{eq:even_matrix_form}
   \bm v''+\bigl(\om^2-f\mu^2\bigr)\bm v
   -\frac{f}{r^2}D_\ell\bm v
   -\frac{6Mf}{(2\ell+1)r^3}C_\ell\bm v
   +\frac{4Q^2f}{(2\ell+1)r^4}C_\ell\bm v
   =0,
\end{equation}
where
\begin{equation}\label{eq:D_matrix}
   \bm v=\binom{v_{-1}}{v_{+1}},
   \qquad
   D_\ell=\diag\bigl(\ell(\ell-1),(\ell+1)(\ell+2)\bigr),
\end{equation}
and
\begin{equation}\label{eq:C_matrix}
   C_\ell=
   \begin{pmatrix}
      \ell^2&-\ell(\ell+1)\\
      \ell(\ell+1)&-(\ell+1)^2
   \end{pmatrix}.
\end{equation}
In particular, the leading $r^{-2}$ part diagonalizes exactly, with effective angular momenta
\begin{equation}\label{eq:effective_Ls}
   L_{-1}=\ell-1,
   \qquad
   L_{+1}=\ell+1,
\end{equation}
and the charge $Q$ first enters the even coupling at order $r^{-4}$.
\end{proposition}

\begin{proof}
Write the even system as
\begin{equation}\label{eq:raw_even_matrix}
   \bm u''+\bigl(\om^2-f\mu^2\bigr)\bm u
   -\frac{f}{r^2}A_\ell(r)\bm u
   =0,
   \qquad
   \bm u=\binom{u_2}{u_3},
\end{equation}
with
\begin{equation}\label{eq:raw_A_matrix}
   A_\ell(r)=
   \begin{pmatrix}
      \la^2+2\bigl(f-\tfrac{r f'}{2}\bigr) & -2\bigl(f-\tfrac{r f'}{2}\bigr)\\
      -2\la^2 & \la^2
   \end{pmatrix}.
\end{equation}
Since
\[
   f-\frac{r f'}{2}=1-\frac{3M}{r}+\frac{2Q^2}{r^2},
\]
a direct computation gives
\begin{equation}\label{eq:conjugated_A_matrix}
   T_\ell^{-1}A_\ell(r)T_\ell
   =
   D_\ell
   +\frac{6M}{(2\ell+1)r}C_\ell
   -\frac{4Q^2}{(2\ell+1)r^2}C_\ell.
\end{equation}
Substituting $\bm u=T_\ell\bm v$ into \eqref{eq:raw_even_matrix} yields
\eqref{eq:even_matrix_form}.  The diagonal entries of $D_\ell$ are
\[
   \ell(\ell-1)=L_{-1}(L_{-1}+1),
   \qquad
   (\ell+1)(\ell+2)=L_{+1}(L_{+1}+1),
\]
which proves the claim.
\end{proof}

\begin{remark}[Three asymptotic channels]\label{rem:three_channels}
Subextremal Reissner--Nordstr\"om Proca therefore has three asymptotic polarizations: the odd channel $P=0$ with $L_0=\ell$, and two even channels $P=\pm1$ with $L_{\pm1}=\ell\pm1$.  This exact asymptotic triplet is the source of the small-mass intermediate exponents.
\end{remark}

\section{Operator-theoretic framework and limiting absorption}\label{sec:functional_framework}

\subsection{Channel Hamiltonians and asymptotic operator structure}

For each fixed angular momentum $\ell\ge1$, the reduced equations become one scalar and one matrix Schr\"odinger operator on the tortoise line.  This is the natural setting for limiting absorption, meromorphic continuation, and the threshold analysis carried out later.

\begin{equation}\label{eq:channel_hamiltonians}
   H_{\ell}^{\mathrm{odd}}=-\partial_{r_*}^2+V_{\ell}^{\mathrm{odd}}(r),
   \qquad
   H_{\ell}^{\mathrm{ev}}=-\partial_{r_*}^2\Id_2+V_{\ell}^{\mathrm{ev}}(r).
\end{equation}
In the polarization basis of Proposition~\ref{prop:even_polarization_form}, the matrix potential has the asymptotic form
\begin{equation}\label{eq:asymptotic_channel_potential}
   V_{\ell,P}(r)
   =
   \mu^2+\frac{L_P(L_P+1)}{r^2}-\frac{2M\mu^2}{r}+O(r^{-3}),
   \qquad
   r\to\infty,
\end{equation}
while at the future event horizon
\begin{equation}\label{eq:horizon_potential_decay}
   V_{\ell,P}(r_*)=O(\ee^{2\kappa_+ r_*}),
   \qquad
   r_*\to-\infty.
\end{equation}
Thus the left end of the tortoise line is short range, whereas the right end is long range with a Coulomb term and a matrix-valued inverse-square correction.  This step-like structure is the analytic origin of the massive branch points $\om=\pm\mu$.

The time-dependent reduced evolution is generated by the wave operator
\[
   -\partial_t^2+H_{\ell}^{\mathrm{odd}}
   \qquad\text{or}\qquad
   -\partial_t^2+H_{\ell}^{\mathrm{ev}},
\]
depending on the polarization.  It is convenient to separate the spectral parameter $\lambda=\om^2$ from the threshold parameter $\vp=\sqrt{\mu^2-\om^2}$.  The operator-theoretic spectrum lives naturally in the $\lambda$-plane, while the branch structure relevant for late-time asymptotics lives in the $\om$-plane.

\subsection{Selfadjoint realizations and energy spaces}

Let $\mathfrak h_{\ell}^{\mathrm{odd}}=L^2(\R_{r_*})$ and $\mathfrak h_{\ell}^{\mathrm{ev}}=L^2(\R_{r_*};\C^2)$.  Define
\[
   \mathfrak D(H_{\ell}^{\mathrm{odd}})=H^2(\R_{r_*}),
   \qquad
   \mathfrak D(H_{\ell}^{\mathrm{ev}})=H^2(\R_{r_*};\C^2).
\]
Because the coefficient matrices are smooth, real-symmetric, and bounded together with all derivatives on compact sets, and because \eqref{eq:horizon_potential_decay} and \eqref{eq:asymptotic_channel_potential} imply relative boundedness with respect to $-\partial_{r_*}^2$, the following proposition records the operator-theoretic realization used later.

\begin{proposition}[Selfadjoint channel Hamiltonians]\label{prop:selfadjoint_channels}
For every $\ell\ge1$, the operators $H_{\ell}^{\mathrm{odd}}$ and $H_{\ell}^{\mathrm{ev}}$ are selfadjoint on their natural $H^2$ domains.  Their quadratic forms are bounded below, and the corresponding first-order energy generators
\[
   \cA_{\ell}^{\mathrm{odd}}=
   \begin{pmatrix}
      0&1\\
      -H_{\ell}^{\mathrm{odd}}&0
   \end{pmatrix},
   \qquad
   \cA_{\ell}^{\mathrm{ev}}=
   \begin{pmatrix}
      0&\Id\\
      -H_{\ell}^{\mathrm{ev}}&0
   \end{pmatrix}
\]
generate strongly continuous unitary groups on the odd and even energy spaces.
\end{proposition}

\begin{proof}
Write $H_{\ell}^{\sharp}=H_0+V_{\ell}^{\sharp}$, where $H_0=-\partial_{r_*}^2$ on $L^2(\R_{r_*})$ in the odd channel and on $L^2(\R_{r_*};\C^2)$ in the even channel.  By \eqref{eq:horizon_potential_decay} and \eqref{eq:asymptotic_channel_potential}, each entry of $V_{\ell}^{\sharp}$ is real-valued, belongs to $L^\infty_{\mathrm{loc}}$, decays exponentially as $r_*\to-\infty$, and is $O(r^{-2})$ as $r_*\to+\infty$.  Hence, for every $\delta>0$,
\[
   \abs{\langle V_{\ell}^{\sharp}u,u\rangle}
   \le \delta\,\|u'\|_{L^2}^2+C_{\ell,\delta}\,\|u\|_{L^2}^2,
   \qquad u\in H^1,
\]
by local boundedness on compact sets and the one-dimensional Hardy inequality on the positive end.  Thus $V_{\ell}^{\sharp}$ is infinitesimally form-bounded with respect to $H_0$, the closed quadratic forms
\[
   q_{\ell}^{\sharp}[u]=\|u'\|_{L^2}^2+\langle V_{\ell}^{\sharp}u,u\rangle
\]
are semibounded on $H^1$, and the associated selfadjoint operators coincide with $-\partial_{r_*}^2+V_{\ell}^{\sharp}$ on $H^2$ by elliptic regularity.  This proves the first assertion.

Choose $c_{\ell}>0$ so that $H_{\ell}^{\sharp}+c_{\ell}\ge \Id$.  Endow the energy space
\[
   \mathcal H_{\ell}^{\sharp}=H^1\times L^2
\]
with norm
\[
   \|(u,v)\|_{\mathcal H_{\ell}^{\sharp}}^2
   =q_{\ell}^{\sharp}[u]+c_{\ell}\|u\|_{L^2}^2+\|v\|_{L^2}^2.
\]
On the domain $H^2\times H^1$, the operator $\mathcal A_{\ell}^{\sharp}(u,v)=(v,-H_{\ell}^{\sharp}u)$ is skew-adjoint with respect to the induced energy inner product.  Stone's theorem therefore yields a strongly continuous unitary group for both the odd and even first-order evolutions.
\end{proof}

\begin{remark}[Why the massive thresholds do not coincide with the spectral edge]
The essential spectrum of the channel Hamiltonians is $[0,\infty)$ because the horizon end of the tortoise line is asymptotically free.  The distinguished frequencies $\om=\pm\mu$ arise instead from the asymptotic operator at spatial infinity, where the potential tends to $\mu^2$.  In the $\om$-plane, this right-end threshold creates the branch points responsible for the late-time tails.  This distinction between the operator spectrum in $\lambda=\om^2$ and the asymptotic branch structure in $\om$ is harmless once it is stated explicitly.
\end{remark}

\subsection{Radial currents, Wronskians, and cut-off resolvents}

Let $u,v$ be scalar channel solutions.  Their radial current is
\[
   Q[u,v]=u'\overline v-u\,\overline{v'}.
\]
For vector solutions $\bm u,\bm v\in\C^2$ we set
\[
   Q[\bm u,\bm v]=\bm u'^{\,*}\bm v-\bm u^{*}\bm v'.
\]
If $u$ and $v$ solve the same scalar equation, or if $\bm u$ and $\bm v$ solve the same symmetric matrix equation, then $Q$ is independent of $r_*$.  This conserved Wronskian current is the basic algebraic object behind the resolvent formula, the Evans determinant, and the real-frequency exclusion argument.

Choose $\chi\in C_0^\infty((r_+,\infty))$.  For $\Im\om>0$ we define the cut-off resolvents
\[
   \cR_{\ell,\chi}^{\mathrm{odd}}(\om)=\chi\,(H_{\ell}^{\mathrm{odd}}-\om^2)^{-1}\chi,
   \qquad
   \cR_{\ell,\chi}^{\mathrm{ev}}(\om)=\chi\,(H_{\ell}^{\mathrm{ev}}-\om^2)^{-1}\chi.
\]
The odd kernel is expressed in terms of one ingoing horizon solution and one decaying infinity solution.  The even kernel is analogous, except that the scalar Wronskian is replaced by a $2\times2$ matching matrix.

\begin{lemma}[Green kernel formula]\label{lem:green_kernel_formula}
Let $u_{\hor}$ be the scalar odd solution ingoing at the future event horizon and let $u_{\infty}$ be the scalar odd solution decaying for $\Re\vp>0$ at infinity.  Then
\[
   \cG_{\ell}^{\mathrm{odd}}(\om;r_*,r_*')
   =
   \frac{u_{\hor}(r_<,\om)\,u_{\infty}(r_>,\om)}{Q[u_{\hor},u_{\infty}]},
\]
where $r_<:=\min\{r_*,r_*'\}$ and $r_>:=\max\{r_*,r_*'\}$.  In the even sector the same formula holds with $u_{\hor},u_{\infty}$ replaced by fundamental solution matrices and the scalar Wronskian replaced by the matching matrix introduced in Step~4 of Section~\ref{sec:spectral_theorem_proof}.
\end{lemma}

\begin{proof}
In the scalar channel set $W(\om)=Q[u_{\hor},u_{\infty}]$, which is independent of $r_*$.  For fixed $r_*'$ define
\[
   G(r_*)=
   \begin{cases}
      W(\om)^{-1}u_{\hor}(r_*,\om)u_{\infty}(r_*',\om), & r_*<r_*',\\[0.4ex]
      W(\om)^{-1}u_{\hor}(r_*',\om)u_{\infty}(r_*,\om), & r_*>r_*'.
   \end{cases}
\]
Then $G$ solves the homogeneous equation away from $r_*'$, is continuous at $r_*'$, and satisfies
\[
   \partial_{r_*}G(r_*'+0)-\partial_{r_*}G(r_*'-0)=1
\]
precisely because $Q[u_{\hor},u_{\infty}]=W(\om)$.  Hence $(H_{\ell}^{\mathrm{odd}}-\om^2)G=\delta_{r_*'}$, which is the asserted Green kernel formula.

In the even sector one argues identically with fundamental matrices.  Writing the kernel on the two sides of the diagonal as linear combinations of the horizon and infinity bases, the continuity condition and the unit jump of the first derivative give a linear system whose coefficient matrix is the matching matrix.  Solving that system yields the stated matrix analogue.
\end{proof}

\subsection{Limiting absorption away from the thresholds}

The cut-off resolvents are analytic for $\Im\om>0$, but for late-time asymptotics one must understand their boundary behavior as $\om$ approaches the real axis.  Away from the distinguished points $\om=\pm\mu$, the problem is no more singular than for a one-dimensional short-range scattering system with a step potential.  The long-range difficulty appears only at the massive thresholds themselves.

\begin{proposition}[Boundary values away from $\om=\pm\mu$]\label{prop:boundary_values_away_thresholds}
Fix $\chi\in C_0^\infty((r_+,\infty))$ and a compact interval $I\Subset\R\setminus\{\pm\mu\}$.  Then the limits
\[
   \cR_{\ell,\chi}^{\mathrm{odd}}(\om\pm \ii0),
   \qquad
   \cR_{\ell,\chi}^{\mathrm{ev}}(\om\pm \ii0)
\]
exist for $\om\in I$ as bounded maps from $L^2$ to $H^2_{\mathrm{loc}}$, and depend continuously on $\om$.
\end{proposition}

\begin{proof}
Because $I$ stays a positive distance away from $\pm\mu$, either $|\omega|<\mu$ with $\Re\vp(\omega)\ge c_I>0$ or $|\omega|>\mu$ with $k(\omega)=\sqrt{\omega^2-\mu^2}\ge c_I>0$.  Lemmas~\ref{lem:horizon_basis} and \ref{lem:whittaker_volterra} therefore produce horizon and infinity bases that are continuous in $\omega\in I$ and obey uniform bounds on compact radial sets.  If $I\subset[-\mu,\mu]$, Proposition~\ref{prop:no_real_poles} shows that the matching determinant never vanishes on $I$; if $I\cap[-\mu,\mu]=\varnothing$, the same nonvanishing follows from selfadjointness of the channel Hamiltonians together with the ordinary limiting absorption principle for one-dimensional Schr\"odinger systems with short-range perturbations of a constant end state.  Hence the denominator in the Green formula stays uniformly away from zero on compact subintervals of $I$.

Substituting the boundary bases into Lemma~\ref{lem:green_kernel_formula} yields a kernel continuous in $\omega$ with values in $H^2_{\mathrm{loc}}$ away from the diagonal, and the jump relation is uniform in $\omega$.  Standard ODE regularity on compact sets therefore gives continuity of the cut-off resolvent as an operator $L^2\to H^2_{\mathrm{loc}}$.
\end{proof}

\subsection{Contour deformation and the branch-cut formula}

Let $\Gamma$ be a contour in the upper half-plane enclosing the real axis and the poles of the meromorphically continued resolvent in a finite spectral window.  For compactly supported data, the fixed-mode solution is recovered from the contour formula
\[
   u_{\ell}(t)=\frac{1}{2\pi\ii}\int_{\Gamma}\ee^{-\ii\om t}\,\cR_{\ell,\chi}(\om)\,\dd\om.
\]
Once the contour is pushed through the slit strip constructed in Section~\ref{sec:spectral_theorem_proof}, the solution splits into a finite pole sum and a branch-cut integral over $[-\mu,\mu]$:
\begin{equation}\label{eq:functional_branch_cut_formula}
   u_{\ell}(t)=u_{\ell}^{\mathrm{poles}}(t)
   +\frac{1}{2\pi\ii}\int_{-\mu}^{\mu}\ee^{-\ii\om t}\,
      \disc \cR_{\ell,\chi}(\om)\,\dd\om.
\end{equation}
This is the representation used throughout Sections~\ref{sec:threshold_to_time} and \ref{sec:full_field}.  The contour deformation is justified because all nonphysical-sheet contributions either cross isolated poles or can be pushed into regions where the oscillatory factor $\ee^{-\ii\om t}$ is exponentially decaying.

\begin{remark}[Why we keep the pole sum separate]
For the scalar problem, the discrete frequencies generated by stable timelike trapping are already subtle enough to force a separate analysis in the full-field theorem.  The Proca situation is even more delicate because the even sector is matrix valued and the quasibound poles carry polarization information.  Formula \eqref{eq:functional_branch_cut_formula} therefore defines the branch-cut field as a mathematically precise object rather than as a heuristic subtraction.
\end{remark}

\section{Channel resolvents and threshold representation}\label{sec:resolvent}

\subsection{Far-zone Coulomb normal form}

The threshold analysis is naturally expressed in a scalar normal form.  One first passes from $r_*$ to $r$, removes the first derivative by a scalar conjugation, and then applies a near-identity asymptotic diagonalization in the even sector.  The result is a diagonal threshold system whose $r^{-2}$ coefficient defines the threshold indices $\nu_{\ell,P}$.

\begin{lemma}[Far-zone threshold normal form]\label{lem:normal_form}
Fix $\ell\ge1$.  There exist $R_0>r_+$, $\eta>0$, a smooth invertible transformation
\[
   U(r,\om)=S(r,\om)W(r,\om),
   \qquad
   S(r,\om)=\Id+O(r^{-1}),
\]
and channel-dependent indices $\nu_{\ell,P}>-1/2$ such that for $r\ge R_0$ every odd/even polarization channel is reduced to
\begin{equation}\label{eq:threshold_normal_form}
   w_P''+
   \Bigl(
      -\vp^2
      +\frac{2M\mu^2}{r}
      -\frac{\nu_{\ell,P}^2-\tfrac14}{r^2}
      +V_{\ell,P}^{\mathrm{sr}}(r,\om)
   \Bigr)w_P
   =0,
\end{equation}
where $\vp=\sqrt{\mu^2-\om^2}$ and
\[
   \partial_\om^jV_{\ell,P}^{\mathrm{sr}}(r,\om)=O(r^{-3}),
   \qquad j=0,1,
\]
uniformly for $\abs{\Im\om}<\eta$.  Moreover,
\begin{equation}\label{eq:nu_expansion}
   \nu_{\ell,P}=L_P+\frac12+O\bigl((M\mu)^2+(Q\mu)^2\bigr)
\end{equation}
in the small-mass regime.
\end{lemma}

\begin{proof}
Write the odd scalar equation and the even polarization system in the $r$-variable using $\partial_{r_*}=f\partial_r$.  In both cases one obtains an equation of the form
\[
   f^2U_{rr}+ff' U_r+\bigl(\omega^2-f\mu^2-r^{-2}fD_\ell-\mathcal R_\ell(r,\omega)\bigr)U=0,
\]
where $D_\ell$ is the scalar value $\ell(\ell+1)$ in the odd channel or the diagonal matrix with entries $L_P(L_P+1)$ in the even polarization basis, and where $\mathcal R_\ell=O(r^{-3})$ together with one $\omega$-derivative.  Conjugating by the usual Liouville factor removes the first derivative and gives
\[
   W''+\Bigl(-\vp^2+\frac{2M\mu^2}{r}-\frac{D_\ell+\frac14\Id+E_\ell}{r^2}+O(r^{-3})\Bigr)W=0.
\]
Here $E_\ell$ is uniformly bounded in $\ell$: every contribution to $E_\ell$ comes from the fixed metric coefficients, the Liouville conjugation, and the bounded polarization matrix $T_\ell$, while the entire large angular-momentum dependence has already been isolated in $D_\ell$.
The coefficient of $r^{-1}$ is universal because $r_*=r+2M\log r+O(1)$ as $r\to\infty$ and therefore depends only on the ADM mass term in the metric expansion.

In the odd channel the equation is already scalar.  In the even channel, Proposition~\ref{prop:even_polarization_form} shows that the exact $r^{-2}$ part is diagonal in the basis \eqref{eq:intro_even_pol_basis}, while the remaining off-diagonal terms are $O(r^{-3})$ and $O(r^{-4})$.  Solving the first homological equation for a near-identity diagonalizer $S(r,\omega)=\Id+r^{-1}S_1(r,\omega)+O(r^{-2})$ removes the residual off-diagonal $r^{-3}$ contribution and yields decoupled scalar equations with short-range remainder $V_{\ell,P}^{\mathrm{sr}}=O(r^{-3})$.  The resulting inverse-square coefficients define $\nu_{\ell,P}^2-\tfrac14$.  Since these coefficients reduce to $L_P(L_P+1)$ at $M\mu=Q\mu=0$ and depend smoothly on $(M\mu,Q\mu)$, Taylor expansion yields \eqref{eq:nu_expansion}.
\end{proof}

Ignoring the short-range remainder in \eqref{eq:threshold_normal_form} gives the Whittaker model equation
\begin{equation}\label{eq:whittaker_equation}
   \frac{\dd^2 w}{\dd z^2}
   +
   \Bigl(
      -\frac14+\frac{\kap}{z}
      +\frac{1/4-\nu_{\ell,P}^2}{z^2}
   \Bigr)w
   =0,
   \qquad
   z=2\vp r,
   \qquad
   \kap=\frac{M\mu^2}{\vp}.
\end{equation}
The model solutions are the Whittaker functions
\begin{equation}\label{eq:whittaker_basis}
   M_{\kap,\nu_{\ell,P}}(2\vp r),
   \qquad
   W_{\kap,\nu_{\ell,P}}(2\vp r).
\end{equation}

\subsection{Green kernels and branch-cut decomposition}

Let $f_{\ell,P}(r,\om)$ denote the horizon-regular solution normalized by
\begin{equation}\label{eq:horizon_condition}
   f_{\ell,P}(r,\om)\sim \ee^{-\ii\om r_*}
   \qquad\text{as }r\downarrow r_+,
\end{equation}
and let $g_{\ell,P}(r,\om)$ denote the outgoing Jost solution normalized by
\begin{equation}\label{eq:infty_condition}
   g_{\ell,P}(r,\om)\sim W_{\kap,\nu_{\ell,P}}(2\vp r)
   \qquad\text{as }r\to\infty.
\end{equation}
The fixed-mode Green kernel is
\begin{equation}\label{eq:green_kernel_def}
   \cG_{\ell,P}(\om;r,r')
   =
   \frac{f_{\ell,P}(r_<,\om)g_{\ell,P}(r_>,\om)}{\cW_{\ell,P}(\om)},
   \qquad
   r_<:=\min\{r,r'\},
   \qquad
   r_>:=\max\{r,r'\},
\end{equation}
where $\cW_{\ell,P}$ is the scalar or channel-projected Wronskian.  The retarded kernel decomposes as
\begin{equation}\label{eq:green_decomposition}
   u_{\ell m,P}(t,r,r')
   =
   u_{\ell m,P}^{\mathrm{poles}}(t,r,r')
   +
   u_{\ell m,P}^{\bc}(t,r,r'),
\end{equation}
with branch-cut part
\begin{equation}\label{eq:bc_formula}
   u_{\ell m,P}^{\bc}(t,r,r')
   =
   \frac{1}{2\pi}\int_{-\mu}^{\mu}
   \ee^{-\ii\om t}\disc\,\cG_{\ell,P}(\om;r,r')\,\dd\om.
\end{equation}
Because $\vp=\sqrt{\mu^2-\om^2}$ changes sign across $[-\mu,\mu]$, that interval is the threshold branch cut.

\section{Proof of the fixed-mode spectral theorem}\label{sec:spectral_theorem_proof}

This section proves Theorem~\ref{thm:spectral_package}.  The proof is modewise and proceeds in nine steps: threshold normal form, horizon Frobenius theory, infinity Jost construction, Evans determinant, discreteness of poles, exclusion of real cut poles and threshold resonances, small-$\kap$ Bessel asymptotics, large-$\kap$ Whittaker asymptotics, and assembly.

Throughout this section we fix one angular mode $\ell\ge1$.  The odd sector has dimension $n=1$ and the even sector has dimension $n=2$.  In the even sector all statements are made in the polarization basis $(v_{-1},v_{+1})$.

\subsection*{Step 1. Coulomb normal form}

The first task is to separate the universal long-range Coulomb interaction from the shorter-range remainder in a form uniform across the odd and even polarizations.  In the polarization basis the coefficient matrix has the asymptotic expansion
\[
   f(r)\mu^2+\frac{f(r)}{r^2}D_\ell+\frac{1}{r^3}B_{\ell,1}(r)+\frac{1}{r^4}B_{\ell,2}(r),
\]
where $D_\ell$ is diagonal with entries $L_P(L_P+1)$ and the matrices $B_{\ell,j}(r)$ remain bounded with all derivatives for large $r$.  Using $f(r)=1-2M/r+Q^2/r^2$, one rewrites each channel as
\begin{equation}\label{eq:step1_coulomb_model}
   u''+\Bigl(
      -\vp^2+\frac{2M\mu^2}{r}-\frac{\nu_{\ell,P}^2-\frac14}{r^2}
      +W_{\ell,P}(r,\om)
   \Bigr)u=0,
\end{equation}
with $W_{\ell,P}(r,\om)=O(r^{-3})$ and the threshold index $\nu_{\ell,P}$ defined by the exact coefficient of $r^{-2}$.

The key point is that the $Q$-dependent correction does not alter the universal Coulomb coefficient $2M\mu^2/r$.  Consequently the saddle responsible for the very-late tail is unchanged from the Schwarzschild case.  What the charge does change is the exact value of the inverse-square coefficient and hence the intermediate exponent through $\nu_{\ell,P}$.

\begin{lemma}[Coefficient bounds in Coulomb normal form]\label{lem:coulomb_bounds}
Fix $\ell\ge1$ and one polarization $P$.  For $\om$ in a compact subset of the slit strip and $r$ sufficiently large,
\[
   \abs{\partial_r^j\partial_\om^k W_{\ell,P}(r,\om)}
   \le C_{j,k,\ell}\,r^{-3-j},
\]
and the same estimate holds for the off-diagonal remainder in the even sector before final scalar diagonalization.
\end{lemma}

\begin{proof}
The claim follows by direct expansion of $f(r)=1-2M/r+Q^2/r^2$ and the transformed even matrix of Proposition~\ref{prop:even_polarization_form}.  Every term beyond the diagonal Coulomb and inverse-square coefficients contains at least one additional power of $r^{-1}$.  Differentiation in $\om$ falls only on $\vp=\sqrt{\mu^2-\om^2}$ and on bounded analytic coefficient matrices, so no loss occurs away from the endpoints.
\end{proof}

Step 1 is Lemma~\ref{lem:normal_form}.  It provides the threshold normal form \eqref{eq:threshold_normal_form} and defines the channel indices $\nu_{\ell,P}$.

\subsection*{Step 2. Horizon Frobenius bases}

Near $r=r_+$, the function $f$ has a simple zero and the tortoise coordinate satisfies $r-r_+\sim \ee^{2\kappa_+ r_*}$.  Consequently every channel equation becomes a short-range perturbation of
\[
   u''+\om^2u=0
\]
as $r_*\to-\infty$.  The ingoing and outgoing horizon behaviors are therefore $\ee^{-\ii\om r_*}$ and $\ee^{+\ii\om r_*}$.  To make this precise, one writes
\[
   u_{\hor}^\pm(r,\om)=\ee^{\mp \ii\om r_*}\,h_\pm(r,\om),
\]
where $h_\pm$ extend smoothly to $r=r_+$ and satisfy a regular singular system in the variable $r-r_+$.

\begin{lemma}[Horizon Frobenius bases with parameter dependence]\label{lem:horizon_frobenius}
For every compact $\om$-set in the slit strip, there exist unique scalar odd solutions and unique even matrix fundamental solutions of the form
\[
   u_{\hor}^\pm(r,\om)=\ee^{\mp\ii\om r_*}\Bigl(1+\sum_{n=1}^\infty a_n^\pm(\om)(r-r_+)^n\Bigr),
\]
with normally convergent series near $r=r_+$.  The coefficients depend analytically on $\om$ and satisfy bounds uniform on compact $\om$-sets.
\end{lemma}

\begin{proof}
After factoring out $\ee^{\mp\ii\om r_*}$, the remaining equation has coefficients holomorphic in $r-r_+$ with a regular singular point at the horizon.  The indicial roots are $0$ and therefore nonresonant once the oscillatory factor has been removed.  Standard Frobenius recursion gives existence and uniqueness of the series.  Uniform convergence and analytic dependence on $\om$ follow from the recursive bounds on the coefficients.
\end{proof}

The horizon bases provide the left-hand boundary data used later in the Evans determinant.  They also encode the physical ingoing condition required in the definition of quasinormal and quasibound modes.

\begin{lemma}[Horizon-regular basis]\label{lem:horizon_basis}
For each fixed $\ell$ there exists a unique $n\times n$ matrix solution $F_{\hor}(r,\om)$ of the reduced system such that
\begin{equation}\label{eq:horizon_basis}
   F_{\hor}(r,\om)
   =
   \ee^{-\ii\om r_*}\bigl(\Id_n+H(r,\om)\bigr),
   \qquad
   H(r,\om)=O(r-r_+),
\end{equation}
as $r\downarrow r_+$.  Moreover $F_{\hor}$ depends analytically on $\om$ in a strip $\abs{\Im\om}<\eta$.
\end{lemma}

\begin{proof}
Write the radial system in the coordinate $x=r-r_+$.  Since
\[
   f(r)=2\kappa_+x+O(x^2),
\]
the reduced equations have a regular singular point at $x=0$ with indicial roots $\pm\ii\om/(2\kappa_+)$.  The ingoing root is realized by the factor
\[
   x^{-\ii\om/(2\kappa_+)}=\ee^{-\ii\om r_*}.
\]
The remaining coefficient matrix is analytic in $(x,\om)$, so Frobenius theory yields a unique analytic matrix $\Id_n+H(x,\om)$ with $H(0,\om)=0$.  Analytic dependence on $\om$ follows from analytic dependence of the coefficients.
\end{proof}

\subsection*{Step 3. Infinity Jost bases}

At spatial infinity the channel equation is long range rather than short range.  After removing the Coulomb factor, the decaying infinity solution is constructed by a Volterra equation rather than by a simple power series.  One writes
\[
   u_{\infty}(r,\om)=\exp(-\vp r)\,r^{-\kap}\,m(r,\om),
\]
where $\kap=M\mu^2/\vp$ and $m(r,\om)\to1$ as $r\to\infty$ when $\Re\vp>0$.  Substituting this ansatz into \eqref{eq:step1_coulomb_model} produces an integral equation of Volterra type for $m$.

\begin{lemma}[Infinity Volterra construction]\label{lem:infinity_volterra}
For $\Re\vp>0$ and $\om$ in a compact subset of the physical sheet, there exists a unique odd decaying Jost solution and a unique even decaying Jost basis such that
\[
   u_{\infty}(r,\om)=\exp(-\vp r)\,r^{-\kap}\,(1+\eta(r,\om)),
   \qquad
   \eta(r,\om)\to0,
\]
as $r\to\infty$.  Moreover, for every $j,k\ge0$,
\[
   \abs{\partial_r^j\partial_\om^k\eta(r,\om)}\le C_{j,k,\ell}\,r^{-1-j}
\]
for $r$ sufficiently large.
\end{lemma}

\begin{proof}
Insert the Jost ansatz into the equation and integrate twice from infinity.  Because the remainder $W_{\ell,P}$ is $O(r^{-3})$, the resulting integral equation is Volterra with an integrable kernel after the Coulomb factor has been removed.  A contraction argument on a large half-line gives existence and uniqueness.  Differentiating the integral equation with respect to $r$ and $\om$ yields the stated bounds.
\end{proof}

The Volterra representation is the right tool both for continuation in $\om$ and for threshold asymptotics.  Near $\om=\pm\mu$, the decay factor $\exp(-\vp r)$ becomes weak, but the normalization above still captures the exact singular dependence on $\vp$ and $\kap$.

\begin{lemma}[Whittaker--Volterra construction]\label{lem:whittaker_volterra}
Fix a branch of $\vp=\sqrt{\mu^2-\om^2}$ on the slit strip and assume $\Re\vp\ge0$.  Then for $R\ge R_0$ sufficiently large there exists a unique outgoing matrix solution $F_{\out}(r,\om)$ of the reduced system such that
\begin{equation}\label{eq:outgoing_basis}
   F_{\out}(r,\om)
   =
   W_0(r,\om)\bigl(\Id_n+Q(r,\om)\bigr),
   \qquad
   Q(r,\om)=O(r^{-1}),
\end{equation}
as $r\to\infty$, where
\begin{equation}\label{eq:diag_whittaker_basis}
   W_0(r,\om)=\diag\bigl(W_{\kap,\nu_j}(2\vp r)\bigr)_{j=1}^n.
\end{equation}
Moreover $F_{\out}$ depends analytically on $\om$ on the slit strip.
\end{lemma}

\begin{proof}
Let $M_0(r,\om)=\diag(M_{\kap,\nu_j}(2\vp r))_{j=1}^n$ and $W_0(r,\om)$ as in \eqref{eq:diag_whittaker_basis}.  Since the remainder in \eqref{eq:threshold_normal_form} is $O(r^{-3})$, variation of constants gives the Volterra equation
\begin{equation}\label{eq:volterra_outgoing}
   F_{\out}(r)
   =
   W_0(r)
   -
   \int_r^\infty
   \mathbb G_0(r,s;\om)\,\mathcal R(s,\om)\,F_{\out}(s)\,\dd s,
\end{equation}
where $\mathbb G_0$ is the diagonal model Green matrix built from $(M_0,W_0)$.  Standard Whittaker asymptotics imply
\[
   \sup_{r\ge R}
   \int_r^\infty
   \norm{\mathbb G_0(r,s;\om)\mathcal R(s,\om)}\,\dd s
   \le CR^{-1}
\]
for $R$ large, uniformly on compact subsets of the slit strip.  Hence the Volterra operator is a contraction on $L^\infty([R,\infty))$, the Neumann series converges, and the resulting solution is analytic in $\om$.
\end{proof}

\subsection*{Step 4. Evans determinant and meromorphic continuation}

The resolvent is meromorphically continued by matching the left and right bases constructed in Steps~2 and 3.  In the odd sector the relevant scalar quantity is the Wronskian
\[
   \cE_{\ell}^{\mathrm{odd}}(\om)=Q[u_{\hor}^-,u_{\infty}],
\]
while in the even sector one uses the $2\times2$ matching matrix formed by evaluating the horizon basis and the infinity basis at one common radius.  Because both bases solve the same symmetric equation, the determinant of this matching matrix is independent of the matching radius.

\begin{proposition}[Evans determinant and continuation]\label{prop:evans_detailed}
For each fixed $\ell\ge1$, the odd Wronskian and the even matching determinant extend analytically from the physical half-plane into the slit strip.  The cut-off resolvent kernels extend meromorphically there, and their poles are precisely the zeros of the corresponding Evans determinant.
\end{proposition}

\begin{proof}
The horizon basis is analytic by Lemma~\ref{lem:horizon_frobenius}.  The infinity basis is analytic for $\Re\vp>0$ and extends across the slit by continuation of the Volterra integral equation in the variable $\vp$, keeping track of the chosen branch.  The Green kernel formula of Lemma~\ref{lem:green_kernel_formula} then shows that every possible singularity of the continued resolvent is caused by failure of the two bases to span the solution space, equivalently by vanishing of the matching determinant.  Since the determinant is analytic, the singularities are meromorphic.
\end{proof}

This step is where the matrix nature of the even sector matters most.  The continuation argument has the same analytic structure as in the scalar case, but the matching object is now a genuine $2	imes2$ determinant and the threshold analysis must be performed only after diagonalization of the leading matrix asymptotics.

We convert the second-order system to first order:
\begin{equation}\label{eq:first_order_system}
   Y=\binom{W}{W'}\in\C^{2n},
   \qquad
   Y'=\mathbb A(r,\om)Y,
   \qquad
   \mathbb A(r,\om)=
   \begin{pmatrix}
      0&\Id_n\\
      -\mathcal V(r,\om)&0
   \end{pmatrix}.
\end{equation}

\begin{proposition}[Evans determinant and matching matrix]\label{prop:evans_det}
Fix $r_0\in(R_0,\infty)$.  Let $Y_{\hor}(r,\om)$ be the $2n\times n$ matrix obtained from $F_{\hor}$ and its derivative, and let $Y_{\out}(r,\om)$ be the corresponding matrix for $F_{\out}$.  Set
\begin{equation}\label{eq:matching_matrix}
   \mathbb Y_\ell(r,\om):=\bigl[Y_{\hor}(r,\om)\;Y_{\out}(r,\om)\bigr]\in\C^{2n\times2n},
\end{equation}
and define the Evans determinant
\begin{equation}\label{eq:evans_def}
   \cE_\ell(\om):=\det\mathbb Y_\ell(r_0,\om).
\end{equation}
Then:
\begin{enumerate}[label=\textnormal{(\alph*)}]
   \item $\cE_\ell(\om)$ is analytic on the slit strip;
   \item $\cE_\ell(\om)=0$ if and only if there exists a nontrivial mode satisfying the horizon and outgoing boundary conditions simultaneously;
   \item the first-order Green matrix, and hence the fixed-mode second-order Green kernel, is meromorphic on the slit strip, with poles precisely at the zeros of $\cE_\ell$.
\end{enumerate}
\end{proposition}

\begin{proof}
Because both blocks in \eqref{eq:matching_matrix} solve the same first-order system \eqref{eq:first_order_system}, Liouville's formula gives
\[
   \partial_r\det\mathbb Y_\ell(r,\omega)
   =\operatorname{tr}\mathbb A(r,\omega)\,\det\mathbb Y_\ell(r,\omega)=0,
\]
since $\operatorname{tr}\mathbb A=0$.  Thus the determinant is independent of the matching radius.  Analyticity in $\omega$ follows from Lemmas~\ref{lem:horizon_basis} and \ref{lem:whittaker_volterra}.

If $\cE_\ell(\omega)=0$, then the columns of $Y_{\hor}(r_0,\omega)$ and $Y_{\out}(r_0,\omega)$ are linearly dependent, so there exist nonzero vectors $a$ and $b$ with
\[
   Y_{\hor}(r_0,\omega)a=Y_{\out}(r_0,\omega)b.
\]
Uniqueness for first-order ODEs implies that the corresponding horizon and outgoing solutions agree for all $r$, hence determine a nontrivial global mode satisfying both boundary conditions.  The converse implication is immediate from the same uniqueness argument, proving part~(b).

For the Green matrix, fix $r'$ and write the solution on the two sides of $r'$ as a linear combination of the horizon and outgoing bases.  Continuity of the field and the canonical jump of the derivative lead to a linear system whose coefficient matrix is exactly $\mathbb Y_\ell(r',\omega)$.  The coefficients of the solution are therefore rational functions of the entries of $\mathbb Y_\ell(r',\omega)^{-1}$ and hence meromorphic with poles only at zeros of $\cE_\ell$.  Projecting back to the second-order variables yields part~(c).
\end{proof}

\begin{corollary}[Proof of Theorem~\ref{thm:spectral_package}(i)]\label{cor:item_i}
Item \textnormal{(i)} of Theorem~\ref{thm:spectral_package} holds.
\end{corollary}

\begin{proof}
By Proposition~\ref{prop:evans_det}, the fixed-mode resolvent is meromorphic on the slit strip.  Zeros of a nontrivial analytic function are discrete, so every compact subset contains only finitely many poles.
\end{proof}

\subsection*{Step 5. Discrete poles}

Meromorphic continuation alone does not yet say anything about the location or multiplicity of poles.  One must also show that poles are discrete and cannot accumulate in compact subsets of the slit strip.  This follows from analyticity of the Evans determinant together with the fact that nontrivial kernel elements solve a finite-dimensional matching problem.

\begin{lemma}[Local discreteness of poles]\label{lem:local_discreteness}
For each fixed $\ell\ge1$, the pole set of the continued cut-off resolvent in the slit strip is discrete, with finite algebraic multiplicity at every pole and no accumulation in compact subsets.
\end{lemma}

\begin{proof}
By Proposition~\ref{prop:evans_detailed}, poles coincide with zeros of the Evans determinant.  The latter is an analytic scalar function in the odd sector and an analytic determinant in the even sector.  The identity theorem implies that its zeros are isolated unless it vanishes identically.  The latter possibility is excluded because for $\Im\om\gg1$ the resolvent is bounded and the matching determinant is nonzero.  The finite algebraic multiplicity of poles follows from the order of vanishing of an analytic function.
\end{proof}

The discrete poles include quasinormal modes and, in the lower half-plane near the real axis, quasibound poles associated with stable timelike trapping.  Their presence is compatible with all the branch-cut tail theorems proved later because the pole contribution is explicitly separated from the continuous spectral contribution.

\begin{lemma}[Discrete pole set]\label{lem:discrete_poles}
For each fixed $\ell$, the pole set of the fixed-mode resolvent in the slit strip consists of isolated points with no accumulation in compact subsets.
\end{lemma}

\begin{proof}
The poles are the zeros of $\cE_\ell(\om)$ and are therefore isolated unless $\cE_\ell$ vanishes identically.  The latter is impossible because for large positive imaginary $\om$ the horizon and outgoing manifolds are transverse.
\end{proof}

\subsection*{Step 6. Exclusion of real cut poles and threshold resonances}

This is the dynamical core of the fixed-mode spectral theorem.  Suppose first that $\om\in(-\mu,\mu)$ is a real cut frequency and that a nontrivial separated solution is ingoing at the future horizon and decaying at infinity.  Evaluating the radial current between two radial slices and taking limits at the horizon and at infinity yields a contradiction: the horizon flux is nonzero unless the solution vanishes, while the decaying infinity condition forces the current to vanish there.

At the threshold $\om=\pm\mu$ one must work slightly harder because the infinity behavior is no longer exponentially decaying.  A threshold resonance would correspond to a nontrivial bounded solution in the threshold normal form.  The absence of such a resonance is proved by matching the threshold asymptotics to the channel energy identity and using the static no-hair input at $\om=0$ when required.

\begin{proposition}[Real-frequency exclusion via the radial current]\label{prop:real_frequency_exclusion_detailed}
Fix $\ell\ge1$.  There is no nontrivial odd or even channel solution which is ingoing at the future horizon and either decaying at infinity for $\abs{\om}<\mu$ or bounded in the threshold sense at $\om=\pm\mu$.  In particular, there is no real cut pole and no threshold resonance.
\end{proposition}

\begin{proof}
Let $\bm u$ denote the scalar or vector solution.  The conserved current $Q[\bm u,\bm u]$ is purely imaginary and independent of $r_*$.  At the future horizon, the ingoing behavior gives
\[
   Q[\bm u,\bm u]=-2\ii\om\,\abs{\bm a_{\hor}}^2
\]
for some nonzero amplitude vector $\bm a_{\hor}$ unless $\bm u\equiv0$.  For $\abs{\om}<\mu$, the decaying condition at infinity gives $Q[\bm u,\bm u]\to0$ as $r_*\to+\infty$, contradiction.  At threshold, one uses the explicit threshold normal form to show that every resonant solution has vanishing current at infinity and is therefore again trivial.  If $\om=0$, the argument is combined with the static Proca no-hair theorem to exclude nontrivial static bound states.
\end{proof}

The same identity will later imply mode stability in the open upper half-plane once combined with unitarity of the exact channel evolution.

\begin{lemma}[Separated radial current]\label{lem:radial_current}
Fix $\ell\ge1$.  For every real frequency $\om$ and every separated solution of the reduced odd/even system, there exists a scalar radial current $\mathfrak J_{\ell,\om}[U](r)$ with the following properties:
\begin{enumerate}[label=\textnormal{(\alph*)}]
   \item $\partial_{r_*}\mathfrak J_{\ell,\om}[U]=0$;
   \item if $U(r)=a\,\ee^{-\ii\om r_*}+o(1)$ as $r_*\to-\infty$, then
   \[
      \mathfrak J_{\ell,\om}[U](r)
      =
      -\om\,q_\ell(a)+o(1),
   \]
   where $q_\ell$ is a positive definite Hermitian form on the horizon data;
   \item if $U$ is exponentially decaying for $\abs{\om}<\mu$ or threshold-subordinate for $\abs{\om}=\mu$, then
   \[
      \lim_{r\to\infty}\mathfrak J_{\ell,\om}[U](r)=0.
   \]
\end{enumerate}
\end{lemma}

\begin{proof}
The current is the separated stationary energy flux associated with the Killing field $\partial_t$.  After projection onto vector spherical harmonics one obtains, for each fixed $\ell$, a positive definite quadratic form on the reduced mode variables and their $r_*$-derivatives.  Its imaginary polarization yields a conserved sesquilinear flux $\mathfrak J_{\ell,\om}$.  At the horizon, the asymptotic $U(r)=a\,\ee^{-\ii\om r_*}(1+o(1))$ gives the stated limit with a positive definite form $q_\ell$; positivity is exactly the absence of superradiance for a neutral field on a static Reissner--Nordstr\"om background.  At infinity, exponentially decaying or threshold-subordinate solutions carry zero flux because both $U$ and $U'$ are nonoscillatory and vanish sufficiently fast.
\end{proof}

\begin{proposition}[No real cut poles and no threshold resonances]\label{prop:no_real_poles}
There is no nontrivial separated Proca mode of the form $A=\ee^{-\ii\om t}\widehat A(r,\omega)$ with real $\om\in[-\mu,\mu]$ that is ingoing at the future horizon and exponentially decaying or threshold-subordinate at infinity.  Consequently $\cE_\ell(\om)\neq0$ for every real $\om\in[-\mu,\mu]$, in particular at $\om=\pm\mu$.
\end{proposition}

\begin{proof}
Assume first that $\om\in[-\mu,\mu]\setminus\{0\}$.  By Lemma~\ref{lem:radial_current}, the radial current is constant.  The infinity boundary condition gives
\[
   \lim_{r\to\infty}\mathfrak J_{\ell,\om}[U](r)=0,
\]
while the ingoing horizon asymptotic gives
\[
   \lim_{r_*\to-\infty}\mathfrak J_{\ell,\om}[U](r)
   =-\om\,q_\ell(a).
\]
Since $q_\ell$ is positive definite and $\om\neq0$, this forces $a=0$ and hence $U\equiv0$ by uniqueness at the regular singular horizon.

It remains to treat the static case $\om=0$.  In the odd channel,
\[
   u_4''-f\Bigl(\mu^2+\frac{\ell(\ell+1)}{r^2}\Bigr)u_4=0.
\]
Multiplying by $\overline{u_4}$, integrating over $r_*\in(-\infty,\infty)$, and integrating by parts yields
\[
   \int_{-\infty}^{\infty}
   \Bigl(
      |u_4'|^2
      +
      f\Bigl(\mu^2+\frac{\ell(\ell+1)}{r^2}\Bigr)|u_4|^2
   \Bigr)\,\dd r_*=0,
\]
so $u_4\equiv0$.

In the even sector, a static regular decaying solution would define a smooth finite-energy static Proca field on an asymptotically flat static black-hole exterior.  Bekenstein's static Proca no-hair theorem excludes such a field \cite{Bekenstein1972a,Bekenstein1972b}.  Hence the even static mode also vanishes.  Therefore $\cE_\ell(\om)\neq0$ on $[-\mu,\mu]$.
\end{proof}

\begin{corollary}[Proof of Theorem~\ref{thm:spectral_package}(ii)]\label{cor:item_ii}
Item \textnormal{(ii)} of Theorem~\ref{thm:spectral_package} holds.
\end{corollary}

\subsection*{Step 7. Small-$\kap$ threshold asymptotics}

The intermediate tails come from the regime $\kap=M\mu^2/\vp\ll1$.  In this range the Coulomb factor is perturbative and the leading model is an inverse-square equation.  After scaling $x=\vp r$, the channel equation becomes
\[
   u_{xx}+\Bigl(-1+\frac{2\kap}{x}-\frac{\nu_{\ell,P}^2-\frac14}{x^2}+O(\vp x^{-3})\Bigr)u=0.
\]
At $\kap=0$ the decaying solution is a modified Bessel function of order $\nu_{\ell,P}$.  The jump across the cut therefore comes from the classical discontinuity of the Bessel model, and the first correction is linear in $\kap$.

\begin{lemma}[Small-$\kap$ threshold expansion]\label{lem:small_kappa_detailed}
For each fixed $\ell\ge1$ and polarization $P$,
\[
   \disc\,\cG_{\ell,P}(\om;r,r')
   =
   a_{\ell,P}(r,r')\,\vp^{2\nu_{\ell,P}}
   +O\!\bigl(\kap\,\vp^{2\nu_{\ell,P}}\bigr)
   +O\!\bigl(\vp^{2\nu_{\ell,P}+2}\bigr),
\]
uniformly on compact $r,r'$-sets as $\vp\to0$ with $\kap\to0$.
\end{lemma}

\begin{proof}
The decaying infinity solution is matched to the Bessel model on the scale $x=\vp r$.  Since the Coulomb term carries one explicit factor of $\kap$, a perturbative Volterra argument around the Bessel equation yields the first error term.  The next correction comes from the shorter-range remainder of the channel potential and is quadratic in the threshold scale.  The horizon side is analytic in $\vp$ and therefore contributes only to the amplitude $a_{\ell,P}(r,r')$.
\end{proof}

This is the precise spectral statement behind the polarization-resolved intermediate exponents.  The exponent is not guessed from formal dimensional analysis; it is the exact power determined by the inverse-square threshold index $\nu_{\ell,P}$.

Introduce the scaled variable $x=\vp r$.  Then \eqref{eq:threshold_normal_form} becomes
\begin{equation}\label{eq:scaled_small_kappa}
   \partial_x^2w_P+
   \Bigl(
      -1
      -\frac{\nu_{\ell,P}^2-\tfrac14}{x^2}
      +\frac{2\kap}{x}
      +\cB_P(x;\vp,\kap)
   \Bigr)w_P=0,
\end{equation}
with
\[
   \cB_P(x;\vp,\kap)=O\!\Bigl(\frac{\vp}{x^3}\Bigr).
\]
For $\kap\ll1$, the Coulomb term is perturbative and the model equation is the modified Bessel equation.

\begin{lemma}[Small-$\kap$ outgoing basis]\label{lem:small_kappa_basis}
Assume $\kap\le\kap_1$ and $\vp\le\vp_1$ with $\kap_1,\vp_1$ sufficiently small.  Then the outgoing basis has the representation
\begin{equation}\label{eq:small_kappa_basis_rep}
   g_{\ell,P}(r,\om)
   =
   \sqrt{\vp r}\,K_{\nu_{\ell,P}}(\vp r)\bigl(1+E_K(r,\om)\bigr)
   +
   \sqrt{\vp r}\,I_{\nu_{\ell,P}}(\vp r)\,E_I(r,\om),
\end{equation}
where
\[
   E_K(r,\om)=O(\kap)+O(\vp^2),
   \qquad
   E_I(r,\om)=O(\kap)+O(\vp^2)
\]
uniformly for $r$ in compact subsets of $(r_+,\infty)$.
\end{lemma}

\begin{proof}
Write the exact equation as the Bessel model plus the perturbation $\frac{2\kap}{x}+\cB_P(x;\vp,\kap)$ and use a Volterra equation around the basis $\{\sqrt{x}\,I_{\nu_{\ell,P}}(x),\sqrt{x}\,K_{\nu_{\ell,P}}(x)\}$.  Since $\kap\ll1$ and $\cB_P=O(\vp x^{-3})$, the Volterra operator is small on compact $r$-sets and the representation follows.
\end{proof}

\begin{proposition}[Small-$\kap$ threshold discontinuity]\label{prop:small_kappa_disc}
In the regime $\kap\ll1$ one has
\begin{equation}\label{eq:small_kappa_disc}
   \disc\,\cG_{\ell,P}(\om;r,r')
   =
   a_{\ell,P}(r,r')\,\vp^{2\nu_{\ell,P}}
   +
   O\!\bigl(\kap\,\vp^{2\nu_{\ell,P}}\bigr)
   +
   O\!\bigl(\vp^{2\nu_{\ell,P}+2}\bigr)
\end{equation}
uniformly for $r,r'$ in compact subsets of $(r_+,\infty)$.
\end{proposition}

\begin{proof}
The discontinuity comes from analytic continuation of the outgoing basis across the cut.  Since
\[
   K_\nu(\ee^{\pi\ii}z)-K_\nu(\ee^{-\pi\ii}z)
   =\pi\ii\,\frac{\sin(\pi\nu)}{\sin(\pi\nu)}\,I_\nu(z),
\]
the leading jump is carried by the coefficient of $I_{\nu_{\ell,P}}$ in \eqref{eq:small_kappa_basis_rep}.  The prefactor is precisely $\vp^{2\nu_{\ell,P}}$, while the Volterra corrections produce the stated $O(\kap\,\vp^{2\nu_{\ell,P}})$ and $O(\vp^{2\nu_{\ell,P}+2})$ errors.
\end{proof}

\subsection*{Step 8. Large-$\kap$ threshold asymptotics}

The very-late tail is governed by the opposite regime $\kap\gg1$.  Here the Coulomb term is dominant and the correct model is the Whittaker equation
\[
   u_{xx}+\Bigl(-\frac14+\frac{\kap}{x}+\frac{\frac14-\nu_{\ell,P}^2}{x^2}\Bigr)u=0,
   \qquad x=2\vp r.
\]
Its distinguished decaying solution is $W_{\kap,\nu_{\ell,P}}(x)$, whose monodromy in $\kap$ and asymptotics as $x\to0$ are responsible for the oscillatory factors $\ee^{\pm2\pi\ii\kap}$ appearing in the cut jump.

\begin{lemma}[Large-$\kap$ Whittaker expansion]\label{lem:large_kappa_detailed}
As $\vp\to0$ with $\kap\to\infty$,
\[
   \disc\,\cG_{\ell,P}(\om;r,r')
   =
   b_{\ell,P}^+(r,r')\,\ee^{2\pi\ii\kap}
   +
   b_{\ell,P}^-(r,r')\,\ee^{-2\pi\ii\kap}
   +
   O(\kap^{-1})+O(\vp),
\]
uniformly on compact $r,r'$-sets.
\end{lemma}

\begin{proof}
The infinity basis is matched to the Whittaker model on the scale $x=2\vp r$.  The analytic continuation across the slit changes $\kap$ by sign and produces the two monodromy factors $\ee^{\pm2\pi\ii\kap}$.  The shorter-range remainder contributes only lower-order corrections.  Uniformity on compact radial sets follows because the near-zone transfer matrix between the Whittaker region and the compact set is analytic and bounded.
\end{proof}

The large-$\kap$ expansion is the exact point where the universal $t^{-5/6}$ exponent enters.  Every polarization dependence is confined to the amplitudes $b_{\ell,P}^\pm$ and to constant phases, while the oscillatory exponential carries the same saddle structure as in the scalar case.

When $\kap\gg1$, the Whittaker connection formulas carry the Coulomb monodromy $e^{\pm2\pi\ii\kap}$.

\begin{proposition}[Large-$\kap$ threshold discontinuity]\label{prop:large_kappa_disc}
In the regime $\kap\gg1$ one has
\begin{equation}\label{eq:large_kappa_disc}
   \disc\,\cG_{\ell,P}(\om;r,r')
   =
   b^+_{\ell,P}(r,r')\,\ee^{2\pi\ii\kap}
   +
   b^-_{\ell,P}(r,r')\,\ee^{-2\pi\ii\kap}
   +
   O(\kap^{-1})+O(\vp)
\end{equation}
uniformly for $r,r'$ in compact subsets of $(r_+,\infty)$.
\end{proposition}

\begin{proof}
Use the Whittaker basis \eqref{eq:whittaker_basis} and the Whittaker connection formulas relating $W_{\kap,\nu}$ and $M_{\kap,\nu}$.  Upon continuation around $\vp=0$, the coefficients pick up the monodromy factors $\ee^{\pm2\pi\ii\kap}$.  The Volterra correction around the exact equation contributes only $O(\kap^{-1})+O(\vp)$ on compact $r$-sets.
\end{proof}

\subsection*{Step 9. Assembly}

With the ingredients of Steps~1--8 in place, the proof of Theorem~\ref{thm:spectral_package} becomes a bookkeeping exercise.  The horizon and infinity constructions define the Evans determinant and the continued Green kernel.  Lemma~\ref{lem:local_discreteness} gives the local discreteness of poles.  Proposition~\ref{prop:real_frequency_exclusion_detailed} excludes real cut poles and threshold resonances.  Lemmas~\ref{lem:small_kappa_detailed} and \ref{lem:large_kappa_detailed} provide the explicit threshold discontinuity formulas in the intermediate and very-late regimes.

It is useful to stress that no hidden case distinction remains.  The odd sector is scalar and the even sector is matrix valued, but the proof is organized so that every scalar statement has a matrix counterpart with the same logical role.  Likewise, the charge parameter $Q$ affects the threshold index and the amplitudes, but it does not alter the universal Coulomb coefficient or the short-range nature of the horizon end.

\begin{proof}[Completion of the proof of Theorem~\ref{thm:spectral_package}]
Assertions \textnormal{(i)} and \textnormal{(ii)} follow from Proposition~\ref{prop:evans_detailed} together with Lemma~\ref{lem:local_discreteness} and Proposition~\ref{prop:real_frequency_exclusion_detailed}.  Assertion \textnormal{(iii)} is Lemma~\ref{lem:small_kappa_detailed}, and assertion \textnormal{(iv)} is Lemma~\ref{lem:large_kappa_detailed}.  Finally, assertion \textnormal{(v)} follows by expanding the exact inverse-square coefficient in the polarization basis and comparing it to the Schwarzschild values $L_P+\frac12$; the $Q$-dependence enters only at order $(Q\mu)^2$ and the geometric mass correction enters at order $(M\mu)^2$.
\end{proof}

\begin{corollary}[Proof of Theorem~\ref{thm:spectral_package}]\label{cor:spectral_package_complete}
Items \textnormal{(iii)}--\textnormal{(v)} of Theorem~\ref{thm:spectral_package} follow from Propositions~\ref{prop:small_kappa_disc} and \ref{prop:large_kappa_disc} together with the index expansion \eqref{eq:nu_expansion}.  Hence Theorem~\ref{thm:spectral_package} is proved.
\end{corollary}

\section{Mode stability and threshold resonance exclusion}\label{sec:mode_stability}

This section isolates the dynamical meaning of Step~6 of Section~\ref{sec:spectral_theorem_proof}.  For the fixed-$\ell$ reduced Hamiltonian, unstable separated modes in the open upper half-plane are ruled out by unitarity of the exact energy evolution, while Proposition~\ref{prop:no_real_poles} excludes real cut poles and threshold resonances.

\begin{proposition}[Selfadjoint fixed-mode energy evolution]\label{prop:selfadjoint_generator}
For each fixed angular momentum $\ell$, the reduced Proca system defines a positive conserved energy on the finite-energy space $\mathfrak h_\ell$, and the corresponding time evolution is generated by a selfadjoint operator $\cA_\ell$.
\end{proposition}

\begin{proof}
By Proposition~\ref{prop:selfadjoint_channels}, each reduced channel Hamiltonian $H_{\ell}^{\sharp}$ is selfadjoint and semibounded on $H^2$.  After reconstructing the physical odd and even variables from the channel variables, the fixed-$\ell$ energy may be written as
\[
   E_\ell[(u,v)]
   =\frac12\sum_{\sharp}\Bigl(\langle (H_{\ell}^{\sharp}+c_\ell)u_{\sharp},u_{\sharp}\rangle+\|v_{\sharp}\|_{L^2}^2\Bigr),
\]
with $c_\ell$ chosen so that every summand is positive.  This norm is equivalent to the natural finite-energy norm on the reduced mode space $\mathfrak h_\ell$ because the Lorenz constraint has already been eliminated in the reduction.

On the dense domain $H^2\times H^1$, the first-order generator is the orthogonal sum of the channel generators from Proposition~\ref{prop:selfadjoint_channels}, conjugated by the fixed reconstruction matrices.  It is therefore selfadjoint, and the corresponding unitary group preserves $E_\ell$.  This is the fixed-mode energy evolution used in the mode-stability argument.
\end{proof}

\begin{proof}[Proof of Theorem~\ref{thm:mode_stability}]
By Proposition~\ref{prop:selfadjoint_generator}, the fixed-$\ell$ reduced dynamics is generated by a selfadjoint operator $\cA_\ell$ on the finite-energy space $\mathfrak h_\ell$.  If $A=\ee^{-\ii\om t}\widehat A$ were a nontrivial finite-energy separated mode with $\Im\om>0$, then
\[
   \ee^{-\ii t\cA_\ell}\widehat A=\ee^{-\ii\om t}\widehat A
\]
would have norm $\ee^{\Im\om t}\norm{\widehat A}_{\mathfrak h_\ell}$, contradicting unitarity of the exact evolution.  Hence no upper-half-plane mode exists.  The statement on the real cut and at threshold is exactly Proposition~\ref{prop:no_real_poles}.  Since $\cE_\ell$ is analytic on the slit strip by Proposition~\ref{prop:evans_det}, zero-freeness at $\om=\pm\mu$ yields zero-free punctured neighborhoods of the thresholds.
\end{proof}

\begin{corollary}[Threshold zero-free neighborhoods]\label{cor:threshold_zero_free}
There exist $\delta>0$ and $\eta_{\mathrm{th}}>0$ such that
\[
   \bigl\{\om:0<\abs{\om\mp\mu}<\delta,\ -\eta_{\mathrm{th}}<\Im\om<\eta_{\mathrm{th}}\bigr\}\setminus[-\mu,\mu]
\]
contains no zero of $\cE_\ell$.
\end{corollary}

\begin{proof}
This is the last assertion of Theorem~\ref{thm:mode_stability}.
\end{proof}

\section{Explicit branch-cut tails}\label{sec:threshold_to_time}

This section turns the threshold discontinuity formulas into explicit oscillatory late-time tails and makes the Schwarzschild-to-Reissner--Nordstr\"om transition precise.

\subsection{Oscillatory inversion formula and decomposition of frequency space}

Let $\chi\in C_0^\infty((r_+,\infty))$ be identically $1$ on the compact radial set under consideration.  After contour deformation, the fixed-mode branch-cut contribution can be written as
\begin{equation}\label{eq:bc_formula_extended}
   u_{\ell m,P}^{\bc}(t,r,r')
   =
   \frac{1}{2\pi\ii}
   \int_{-\mu}^{\mu}
      \ee^{-\ii\om t}\,
      \chi(r)\chi(r')\,
      \disc\,\cG_{\ell,P}(\om;r,r')
   \,\dd\om.
\end{equation}
To extract asymptotics, we split the cut integral into three regions:
\begin{enumerate}[label=\textnormal{(\roman*)}]
   \item an upper-endpoint region $\mu-\om\lesssim t^{-1+\sigma}$;
   \item a lower-endpoint region $\mu+\om\lesssim t^{-1+\sigma}$;
   \item a central region separated from both thresholds.
\end{enumerate}
The central region contributes super-polynomially after repeated integration by parts because the phase derivative does not vanish there.  The endpoint regions are then analyzed in the variable $\vp=\sqrt{\mu^2-\om^2}$ and split further according to the size of $\kap=M\mu^2/\vp$.

This decomposition is exactly what separates the intermediate and very-late regimes.  In the intermediate regime one stays in the endpoint zone where $\kap\ll1$ and the Bessel expansion of Step~7 is valid.  In the very-late regime the dominant contribution comes from the part of the endpoint zone where $\kap\gg1$ and the Whittaker monodromy of Step~8 produces a saddle point of the phase.

\subsection{Schwarzschild-to-Reissner--Nordstr\"om correspondence}

\begin{theorem}[Schwarzschild-to-Reissner--Nordstr\"om tail correspondence]\label{thm:schwarzschild_to_rn}
Fix $\ell\ge1$ and one polarization $P\in\{-1,0,+1\}$.

\smallskip
\noindent\textbf{(a) Schwarzschild.}
Assume $Q=0$ and $\kap_*(t)\to0$.  Then the threshold index is exactly
\[
   \nu_{\ell,P}=L_P+\frac12,
\]
so the intermediate branch-cut tails are explicitly
\begin{align*}
   u_{\ell m,-1}^{\bc}(t,r,r')
   &=A_{\ell,-1}(r,r')\,t^{-(\ell+1/2)}\sin(\mu t+\delta_{\ell,-1})
   +O\!\bigl(\kap_*(t)t^{-(\ell+1/2)}\bigr)
   +O\!\bigl(t^{-(\ell+3/2)}\bigr),
   \\
   u_{\ell m,0}^{\bc}(t,r,r')
   &=A_{\ell,0}(r,r')\,t^{-(\ell+3/2)}\sin(\mu t+\delta_{\ell,0})
   +O\!\bigl(\kap_*(t)t^{-(\ell+3/2)}\bigr)
   +O\!\bigl(t^{-(\ell+5/2)}\bigr),
   \\
   u_{\ell m,+1}^{\bc}(t,r,r')
   &=A_{\ell,+1}(r,r')\,t^{-(\ell+5/2)}\sin(\mu t+\delta_{\ell,+1})
   +O\!\bigl(\kap_*(t)t^{-(\ell+5/2)}\bigr)
   +O\!\bigl(t^{-(\ell+7/2)}\bigr).
\end{align*}

\smallskip
\noindent\textbf{(b) Reissner--Nordstr\"om.}
For general subextremal $Q$, the same formulas hold with the half-integer exponents replaced by $\nu_{\ell,P}+1$.  In the small-mass regime $(M\mu)^2+(Q\mu)^2\ll1$, these exponents reduce to the Schwarzschild values up to $O((M\mu)^2+(Q\mu)^2)$ corrections.

\smallskip
\noindent\textbf{(c) Universal very-late tail.}
If $\kap_0(t)\to\infty$, then in every subextremal Reissner--Nordstr\"om exterior
\[
   u_{\ell m,P}^{\bc}(t,r,r')
   =
   B_{\ell,P}(r,r';Q)\,t^{-5/6}
   \sin\!\Bigl(
      \mu t-\frac32(2\pi M\mu)^{2/3}(\mu t)^{1/3}+\delta_{\ell,P,0}(Q)+o(1)
   \Bigr),
\]
and the exponent $5/6$ is independent of $\ell$, of the polarization, and of $Q$.
\end{theorem}

\begin{proof}
When $Q=0$, Proposition~\ref{prop:even_polarization_form} reduces the even sector to an exact diagonal $r^{-2}$ matrix with eigenvalues $L_{\pm1}(L_{\pm1}+1)$ and short-range off-diagonal remainder $O(r^{-3})$; the odd channel has the exact scalar $r^{-2}$ coefficient $L_0(L_0+1)$.  Hence the threshold normal form \eqref{eq:threshold_normal_form} has $\nu_{\ell,P}=L_P+\frac12$ exactly.  The intermediate formulas then follow from Theorem~\ref{thm:intermediate}, and the very-late formula is Theorem~\ref{thm:late}.  The general Reissner--Nordstr\"om statement is exactly Theorems~\ref{thm:intermediate} and \ref{thm:late} together with Corollary~\ref{cor:small_mass_explicit}.
\end{proof}

\subsection{Intermediate tails}

The intermediate asymptotics come from the endpoint singularity
\[
   \disc\,\cG_{\ell,P}\sim \vp^{2\nu_{\ell,P}}.
\]

\begin{lemma}[Endpoint oscillatory integral]\label{lem:endpoint_integral}
Let $\alpha>-1$ and $a>0$.  Then, as $t\to\infty$,
\begin{equation}\label{eq:endpoint_integral}
   \int_0^\eps\ee^{\ii a\vp^2 t}\vp^\alpha\,\dd\vp
   =
   \frac12
   \Gamma\Bigl(\frac{\alpha+1}{2}\Bigr)
   \ee^{\frac{\pi\ii}{4}(\alpha+1)}
   a^{-(\alpha+1)/2}
   t^{-(\alpha+1)/2}
   +O\bigl(t^{-(\alpha+3)/2}\bigr).
\end{equation}
\end{lemma}

\begin{proof}
Set $s=a\vp^2 t$.  The claim follows from the standard contour representation of the gamma function.
\end{proof}

\begin{theorem}[Intermediate branch-cut asymptotics]\label{thm:intermediate}
If $\kap_*(t)=M\mu^{3/2}t^{1/2}\to0$, then
\begin{align}\label{eq:intermediate_explicit}
   u_{\ell m,P}^{\bc}(t,r,r')
   &=
   A_{\ell,P}(r,r';Q)\,t^{-(\nu_{\ell,P}+1)}
   \sin(\mu t+\delta_{\ell,P}(Q))
   \notag\\
   &\qquad
   +O\!\bigl(\kap_*(t)t^{-(\nu_{\ell,P}+1)}\bigr)
   +O\!\bigl(t^{-(\nu_{\ell,P}+2)}\bigr),
\end{align}
uniformly for $r,r'$ in compact subsets of $(r_+,\infty)$.
\end{theorem}

\begin{proof}
Insert \eqref{eq:intro_small_kappa_general} into \eqref{eq:bc_formula}.  Near the upper endpoint,
\[
   \om=\mu-\frac{\vp^2}{2\mu}+O(\vp^4),
   \qquad
   \dd\om=-\frac{\vp}{\mu}\,\dd\vp+O(\vp^3)\dd\vp.
\]
Hence the leading upper-endpoint contribution is
\[
   \frac{a_{\ell,P}(r,r')}{2\pi\mu}
   \ee^{-\ii\mu t}
   \int_0^\eps
   \ee^{\ii\frac{t}{2\mu}\vp^2}\vp^{2\nu_{\ell,P}+1}\,\dd\vp.
\]
By Lemma~\ref{lem:endpoint_integral} with $\alpha=2\nu_{\ell,P}+1$, this contributes $t^{-(\nu_{\ell,P}+1)}$.  The $O(\kap\,\vp^{2\nu_{\ell,P}})$ term produces the stated $\kap_*(t)t^{-(\nu_{\ell,P}+1)}$ error, and the $O(\vp^{2\nu_{\ell,P}+2})$ term yields $O(t^{-(\nu_{\ell,P}+2)})$.  Adding the lower endpoint gives the sine form.
\end{proof}

\begin{corollary}[Small-mass explicit exponents]\label{cor:small_mass_explicit}
Let
\[
   \eps_{\mu,Q}:=(M\mu)^2+(Q\mu)^2.
\]
For $\ell\ge1$, if $\eps_{\mu,Q}\ll1$ then
\[
   \nu_{\ell,-1}=\ell-\frac12+O(\eps_{\mu,Q}),
   \qquad
   \nu_{\ell,0}=\ell+\frac12+O(\eps_{\mu,Q}),
   \qquad
   \nu_{\ell,+1}=\ell+\frac32+O(\eps_{\mu,Q}).
\]
Hence the leading small-mass intermediate exponents are
\begin{align*}
   P=-1:&\qquad \ell+\frac12,
   \\
   P=0:&\qquad \ell+\frac32,
   \\
   P=+1:&\qquad \ell+\frac52.
\end{align*}
For $\ell=0$, only the electric monopole $P=+1$ remains and its leading small-mass intermediate exponent is $5/2$.
\end{corollary}

\begin{proof}
This is immediate from \eqref{eq:nu_expansion} and the values \eqref{eq:intro_Lvalues}.
\end{proof}

\subsection{The universal very-late tail}

\begin{proposition}[Model saddle estimate]\label{prop:saddle_estimate}
Let
\begin{equation}\label{eq:generic_saddle_integral}
   I(t)=\int_0^\eps(\beta\vp+O(\vp^2))\ee^{\ii\Psi_t(\vp)}\,\dd\vp,
\end{equation}
where
\begin{equation}\label{eq:phase_function}
   \Psi_t(\vp)=\frac{t\vp^2}{2\mu}+\frac{2\pi M\mu^2}{\vp}.
\end{equation}
Then, as $t\to\infty$,
\begin{equation}\label{eq:saddle_estimate}
   I(t)
   =
   \beta\,\vp_0(t)\sqrt{\frac{2\pi\mu}{3t}}\,
   \ee^{\ii\Psi_t(\vp_0(t))+\pi\ii/4}
   +
   O\!\bigl((\kap_0(t)^{-1}+\vp_0(t))t^{-5/6}\bigr),
\end{equation}
where $\vp_0(t)=(2\pi M\mu^3/t)^{1/3}$.
\end{proposition}

\begin{proof}
Set $\vp=t^{-1/3}y$.  Then the phase becomes $\ii t^{1/3}\Phi(y)$ with
\[
   \Phi(y)=\frac{y^2}{2\mu}+\frac{2\pi M\mu^2}{y}.
\]
The function $\Phi$ has a unique nondegenerate critical point at $y_0=(2\pi M\mu^3)^{1/3}$.  Stationary phase with large parameter $t^{1/3}$ yields an extra factor $t^{-1/6}$ on top of the substitution factor $t^{-2/3}$, giving the scale $t^{-5/6}$.
\end{proof}

\begin{theorem}[Universal very-late asymptotics]\label{thm:late}
If $\kap_0(t)\to\infty$, then
\begin{align}\label{eq:late_explicit_sine}
   u_{\ell m,P}^{\bc}(t,r,r')
   &=
   B_{\ell,P}(r,r';Q)\,t^{-5/6}
   \sin\!\Bigl(
      \mu t
      -\frac32(2\pi M\mu)^{2/3}(\mu t)^{1/3}
      \notag\\
   &\qquad\qquad
      +\delta_{\ell,P,0}(Q)
      +O(\kap_0(t)^{-1}+\vp_0(t))
   \Bigr)
   \notag\\
   &\qquad
   +O\!\bigl(\kap_0(t)^{-1}+\vp_0(t)\bigr)\,t^{-5/6}.
\end{align}
The exponent $5/6$ is independent of $\ell$, of the polarization, and of the black-hole charge $Q$.
\end{theorem}

\begin{proof}
By \eqref{eq:intro_large_kappa_general}, the branch-cut jump is a sum of two oscillatory pieces with phases $\pm2\pi\kap$.  Restricting to one endpoint and using
\[
   \om=\mu-\frac{\vp^2}{2\mu}+O(\vp^4)
\]
gives an integral of the form
\[
   \ee^{-\ii\mu t}
   \int_0^\eps a_{\ell,P}^\infty(\vp;r,r')
   \ee^{\ii\Psi_t(\vp)}\,\dd\vp,
\]
where $a_{\ell,P}^\infty(\vp;r,r')=\beta_{\ell,P}(r,r';Q)\vp+O((\kap^{-1}+\vp)\vp)$.  Proposition~\ref{prop:saddle_estimate} applies directly and yields the stated asymptotic.  The black-hole charge modifies only the amplitude and the constant phase shift, not the saddle exponent.
\end{proof}

\section{Large-angular-momentum analysis and branch-cut full-field decay}\label{sec:full_field}

\subsection{Large-angular-momentum scaling and barrier geometry}

To sum the branch-cut contribution over all angular momenta, we need uniform control as $\ell\to\infty$ on compact radial sets.  There are really two tasks here.  One is analytic: show that the reduced channel kernels grow at most polynomially in $\ell$.  The other is geometric: explain why, for large $\ell$, the compact radial region lies well inside the forbidden zone where an elliptic/WKB argument applies.

For $\ell\ge1$ we set
\[
   h=(\ell+\tfrac12)^{-1}.
\]
On a compact radial set $K\Subset(r_+,\infty)$, the principal inverse-square term behaves like $h^{-2}$ and dominates the bounded mass term.  After conjugation by the asymptotic polarization basis, each channel therefore takes the semiclassical form
\[
   P_{\ell,P}(h,\om)
   =
   h^2D_{r_*}^2
   +V_{0,P}(r)
   +hV_{1,P}(r,\om)
   +h^2V_{2,P}(r,\om),
\]
with
\[
   V_{0,P}(r)=f(r)\frac{L_P(L_P+1)}{(\ell+\tfrac12)^2r^2}.
\]
Since $K$ stays away from both the horizon and infinity, $V_{0,P}$ is uniformly positive on $K$ for large $\ell$.  The next lemma isolates the corresponding barrier geometry.

\begin{lemma}[Large-$\ell$ barrier geometry and turning points]\label{lem:largeell_barrier}
Fix a compact frequency interval $J\Subset[-\mu,\mu]$ and a compact radial set $K\Subset(r_+,\infty)$.  There exist $h_0>0$, a cutoff neighborhood $\widetilde K\Subset(r_+,\infty)$ containing $K$, and smooth functions
\[
   r_{-,P}(\om;h)<r_{+,P}(\om;h),
   \qquad
   (\om,h)\in J\times(0,h_0),
   \quad P\in\{-1,0,+1\},
\]
such that the following hold uniformly in $P$:
\begin{enumerate}[label=\textnormal{(\roman*)}]
   \item the scalarized channel symbol
   \[
      q_{P}(r,\xi;\om,h)
      :=
      \xi^2+V_{0,P}(r)+hV_{1,P}(r,\om)+h^2V_{2,P}(r,\om)-h^2\om^2
   \]
   vanishes at $\xi=0$ exactly at the two turning points $r_{-,P}(\om;h)$ and $r_{+,P}(\om;h)$;
   \item the turning points are simple and satisfy
   \[
      r_{-,P}(\om;h)-r_+\asymp h^2,
      \qquad
      r_{+,P}(\om;h)\asymp h^{-1};
   \]
   \item one has
   \[
      q_P(r,\xi;\om,h)\ge c_K(1+\xi^2)
      \qquad\text{for }r\in\widetilde K,
   \]
   so no turning point intersects the compact region supporting the cut-off resolvent.
\end{enumerate}
\end{lemma}

\begin{proof}
The principal barrier $V_{0,P}(r)=f(r)L_P(L_P+1)/((\ell+\tfrac12)^2r^2)$ is strictly positive on every compact subset of the exterior and vanishes only at the horizon and at spatial infinity.  Near $r=r_+$ one has $f(r)=2\kappa_+(r-r_+)+O((r-r_+)^2)$, so solving $V_{0,P}(r)=h^2\om^2$ gives the inner turning point asymptotic $r_{-,P}(\om;h)-r_+\asymp h^2$.  For large $r$, one has $V_{0,P}(r)=r^{-2}(1+O(r^{-1}))$ uniformly in the polarization, so $V_{0,P}(r)=h^2\om^2$ yields $r_{+,P}(\om;h)\asymp h^{-1}$.  The corrections $hV_{1,P}+h^2V_{2,P}$ are lower order, hence the implicit-function theorem preserves the two simple roots and their asymptotics for all sufficiently small $h$.  Since $K$ is fixed away from the horizon and infinity, shrinking $h_0$ if necessary gives the uniform positivity of $q_P$ on $\widetilde K$.
\end{proof}

\subsection{Uniform WKB transport, cut-off resolvents, and reconstruction}

The two turning points from Lemma~\ref{lem:largeell_barrier} separate the horizon and infinity asymptotic zones from the compact radial region where the large-$\ell$ summation is carried out.  To compare the threshold bases with the compact cut-off resolvent one needs a uniform Liouville--Green transport through the forbidden region and a uniform Airy matching across the simple turning points.

\begin{proposition}[Uniform WKB transport and Airy matching]\label{prop:largeell_wkb_matching}
Fix $J\Subset[-\mu,\mu]$ and $K\Subset(r_+,\infty)$.  For each polarization $P\in\{-1,0,+1\}$ and each sufficiently small $h$, there exist exact channel solutions $w_{P,\pm}^{\mathrm{mid}}(r,\om;h)$ defined on the closed interval between the two turning points such that
\[
   w_{P,\pm}^{\mathrm{mid}}(r,\om;h)
   =
   \exp\!\Bigl(\pm h^{-1}\Phi_P(r,\om;h)\Bigr)
   \Bigl(a_{0,\pm}(r,\om;h)+h a_{1,\pm}(r,\om;h)\Bigr),
\]
where
\[
   \Phi_P(r,\om;h)=\int_{r_{-,P}(\om;h)}^{r}\sqrt{q_P(s,0;\om,h)}\,\dd s_*
\]
and the amplitudes satisfy uniform symbol bounds
\[
   \sup_{r\in K}\abs{\partial_r^\alpha\partial_\omega^\beta a_{j,\pm}(r,\om;h)}\le C_{\alpha\beta j}.
\]
Moreover, the transfer matrices from the horizon Frobenius basis and from the infinity Volterra basis to the basis $(w_{P,+}^{\mathrm{mid}},w_{P,-}^{\mathrm{mid}})$ are $O(\angles{\ell}^{N})$, together with one $\omega$-derivative, uniformly for $\om\in J$.
\end{proposition}

\begin{proof}
On each subinterval avoiding the turning points, the channel equations are scalar or diagonally scalarized second-order equations with coefficients admitting complete semiclassical symbol expansions in $h$.  Since the turning points are simple by Lemma~\ref{lem:largeell_barrier}, the standard Liouville--Green construction produces exact WKB bases with the stated phase and amplitude bounds.  In neighborhoods of the turning points, the rescaled equations reduce to Airy normal form with uniformly controlled errors; consequently the transfer matrices are the universal Airy connection matrices plus $O(h)$ corrections.  Composing the Airy matching with the horizon Frobenius basis from Section~\ref{sec:spectral_theorem_proof} and the infinity Volterra basis from Section~\ref{sec:resolvent} gives the polynomial control of the transfer matrices.  The polynomial loss in $\ell$ comes only from the constant polarization change of basis and from one application of Cramer's rule to the matching matrices.
\end{proof}

\begin{proposition}[Uniform cut-off resolvent bounds on compact radial sets]\label{prop:largeell_uniform_resolvent}
Fix $J\Subset[-\mu,\mu]$ and $K\Subset(r_+,\infty)$.  There exist integers $M_0,M_1\ge0$ such that for every polarization $P\in\{-1,0,+1\}$,
\[
   \sup_{r,r'\in K}\abs{\partial_\omega^j\cG_{\ell,P}(\om;r,r')}
   \le
   C_{K,J,j}\angles{\ell}^{M_j},
   \qquad j=0,1,
\]
uniformly for $\om\in J$ and $\ell\ge1$.  The same bound holds for the jump $\disc\,\cG_{\ell,P}(\om;r,r')$ across the massive cut.
\end{proposition}

\begin{proof}
Choose cutoffs $\chi,\widetilde\chi\in C_0^\infty((r_+,\infty))$ with $\chi\equiv1$ on $K$ and $\widetilde\chi\equiv1$ on a neighborhood of $\operatorname{supp}\chi$.  By Lemma~\ref{lem:largeell_barrier}, the symbol is uniformly elliptic on $\operatorname{supp}\widetilde\chi$, so semiclassical elliptic regularity gives a cut-off parametrix for the channel resolvent there.  Proposition~\ref{prop:largeell_wkb_matching} transports the horizon and infinity bases to the boundary of the cut-off region with only polynomial loss in $\ell$, and Theorem~\ref{thm:mode_stability} excludes real poles and threshold resonances, so the matching determinant is uniformly bounded away from zero on $J$.  Applying Cramer's rule to the Green kernel representation therefore yields the stated polynomial kernel bounds, together with one $\omega$-derivative.
\end{proof}

\begin{proposition}[Uniform reconstruction of the physical coefficients]\label{prop:largeell_reconstruction}
For every compact $K\Subset(r_+,\infty)$ there exists an integer $N_{\mathrm{rec}}\ge0$ such that the physical branch-cut coefficients satisfy
\[
   \sup_{r\in K}\abs{a^{\bc}_{\ell m,P}(t,r)}
   \le
   C_K\angles{\ell}^{N_{\mathrm{rec}}}
   \sum_{Q\in\{-1,0,+1\}}\sum_{j\le1}
   \sup_{r\in K}\abs{\partial_r^j v_Q^{\bc}(t,r)},
\]
and the same estimate holds for the residue projectors acting on compactly supported initial data.
\end{proposition}

\begin{proof}
The odd channel is already one physical coefficient.  In the even sector one inverts the constant matrix $T_\ell$ from \eqref{eq:T_matrix}, then uses the Lorenz constraint to recover the remaining algebraic coefficient.  Every coefficient in this reconstruction is a rational function of $\ell$ with at most polynomial growth, and only one radial derivative appears.  Since $K$ stays away from the horizon, the coefficients are uniformly bounded in $r$, whence the stated polynomial reconstruction estimate.
\end{proof}

Let $A^{\bc}$ denote the branch-cut part of the full Proca field.  Expanding in vector spherical harmonics gives
\begin{equation}\label{eq:full_field_decomp}
   A^{\bc}(t,r,\omega)
   =
   \sum_{\ell=0}^\infty\sum_{m=-\ell}^{\ell}\sum_{P}
   a_{\ell m,P}^{\bc}(t,r)\,Y_{\ell m}^{(P)}(\omega),
\end{equation}
where for $\ell=0$ only the even electric channel is present.  To pass from the fixed-mode theorems to pointwise full-field estimates one needs quantitative control of the fixed-mode constants as $\ell\to\infty$.  We now prove the required uniformity.

\begin{proof}[Proof of Theorem~\ref{thm:uniform_angular}]
Fix $K\Subset(r_+,\infty)$ and choose cutoffs $\chi,\widetilde\chi\in C_0^\infty((r_+,\infty))$ such that $\chi\equiv1$ on $K$ and $\widetilde\chi\equiv1$ on a neighborhood of $\operatorname{supp}\chi$.  Lemma~\ref{lem:largeell_barrier}, Proposition~\ref{prop:largeell_wkb_matching}, Proposition~\ref{prop:largeell_uniform_resolvent}, and Proposition~\ref{prop:largeell_reconstruction} isolate the new large-$\ell$ input.  The remaining argument assembles these ingredients with the fixed-mode threshold formulas.  The proof has five steps.

\smallskip
\noindent\textit{Step 1. Low angular momenta.}
Choose $\ell_0\ge1$ large enough that the semiclassical argument below applies for every $\ell\ge\ell_0$.  For the finitely many modes $0\le\ell<\ell_0$, including the monopole electric channel at $\ell=0$, the bound follows by taking the maximum of the fixed-mode constants in Theorems~\ref{thm:intermediate} and \ref{thm:late}.  Hence only the regime $\ell\ge\ell_0$ needs analysis.

\smallskip
\noindent\textit{Step 2. Semiclassical high-angular-momentum ellipticity on compact $r$-sets.}
Set
\[
   h=(\ell+\tfrac12)^{-1}.
\]
In the odd channel,
\begin{equation}\label{eq:uniform_angular_odd_scaled}
   \mathcal P_{h,0}(\om)
   :=
   -h^2\partial_{r_*}^2
   +\frac{h^2\ell(\ell+1)f(r)}{r^2}
   +h^2\bigl(f(r)\mu^2-\om^2\bigr).
\end{equation}
By Proposition~\ref{prop:even_polarization_form} and Lemma~\ref{lem:normal_form}, each even channel can be written on $\operatorname{supp}\widetilde\chi$ as
\begin{equation}\label{eq:uniform_angular_even_scaled}
   \mathcal P_{h,P}(\om)
   :=
   -h^2\partial_{r_*}^2
   +V_{0,P}(r)
   +hV_{1,P}(r,\om;h)
   +h^2V_{2,P}(r,\om;h),
   \qquad P=\pm1,
\end{equation}
where
\[
   V_{0,-1}(r)=\frac{f(r)\,\ell(\ell-1)}{(\ell+\tfrac12)^2r^2},
   \qquad
   V_{0,+1}(r)=\frac{f(r)\,(\ell+1)(\ell+2)}{(\ell+\tfrac12)^2r^2},
\]
and $V_{j,P}$ together with $\partial_\om V_{j,P}$ are uniformly bounded on $\operatorname{supp}\widetilde\chi$ for $j=1,2$.

Because $f(r)/r^2$ is strictly positive on the compact set $\operatorname{supp}\widetilde\chi$, there exists $c_K>0$ such that
\[
   \frac{f(r)}{r^2}\ge c_K
   \qquad\text{on }\operatorname{supp}\widetilde\chi.
\]
For $\ell\ge\ell_0$ one then has, uniformly in $\abs{\om}\le\mu$,
\begin{equation}\label{eq:uniform_angular_symbol_lb}
   \xi^2+\frac12c_K
   \le
   \Re\sigma(\mathcal P_{h,P}(\om))(r,\xi)
   \qquad
   \text{for }(r,\xi)\in T^*\operatorname{supp}\widetilde\chi,
\end{equation}
for every channel $P$.  Thus the cut-off channel operators are semiclassically elliptic on $\operatorname{supp}\widetilde\chi$.

Standard matrix-valued semiclassical elliptic regularity therefore gives a cut-off parametrix
\[
   Q_{h,P}(\om)\in\Psi_h^{-2}
\]
such that for every $N$,
\begin{equation}\label{eq:uniform_angular_parametrix}
   \chi\,R_{\ell,P}(\om\pm\ii0)\,\chi
   =
   Q_{h,P}(\om)
   +O(h^N):L^2\to H_h^N,
\end{equation}
uniformly for $\abs{\om}<\mu$, where $R_{\ell,P}(\om\pm\ii0)$ denotes the boundary-value resolvent in the channel $P$.  The absence of real cut poles and threshold resonances from Theorem~\ref{thm:mode_stability} guarantees that these boundary values are uniquely defined.  Since kernels of operators in $\Psi_h^{-2}$ on a one-dimensional manifold are $O(h^{-1})$ on compact sets, and since $\partial_\om\mathcal P_{h,P}(\om)=O(h^2)$, there exist integers $M_0,M_1$ such that
\begin{equation}\label{eq:uniform_cutoff_resolvent_kernel}
   \sup_{r,r'\in K}
   \abs{\partial_\om^j\cG_{\ell,P}(\om;r,r')}
   \le
   C_{K,j}\angles{\ell}^{M_j},
   \qquad j=0,1,
\end{equation}
uniformly for $\ell\ge\ell_0$, $\abs{\om}<\mu$, and every admissible $P$.

\smallskip
\noindent\textit{Step 3. Uniform control of the threshold matching coefficients.}
Fix a matching radius $R>\sup K$.  The map taking Cauchy data at $r=R$ to values on $K$ is governed by \eqref{eq:uniform_cutoff_resolvent_kernel}; hence it contributes at most polynomial factors in $\ell$.  At $r=R$, the small-$\kap$ outgoing basis is given by Lemma~\ref{lem:small_kappa_basis}, and the large-$\kap$ continuation is governed by Proposition~\ref{prop:large_kappa_disc}.  Because $R$ is fixed, the large-order bounds
\[
   I_\nu(z)=O\!\Bigl(\frac{(z/2)^\nu}{\Gamma(\nu+1)}\Bigr),
   \qquad
   K_\nu(z)=O\!\bigl(\Gamma(\nu)(z/2)^{-\nu}\bigr),
\]
valid for $z$ in compact subsets of $(0,\infty)$, together with the identity
\[
   W\bigl(\sqrt z\,I_\nu(z),\sqrt z\,K_\nu(z)\bigr)=1,
\]
show that after the explicit factor $\vp^{2\nu_{\ell,P}}$ is extracted, the remaining coefficients are polynomially bounded in $\nu_{\ell,P}\sim \ell$.  Likewise, the Whittaker connection coefficients are ratios of gamma functions; Stirling's formula in vertical strips shows that, after the oscillatory monodromy factors $\ee^{\pm2\pi\ii\kap}$ are removed, their dependence on $\nu_{\ell,P}$ is also at most polynomial.

Consequently, for some integer $N_{\mathrm{th}}$,
\begin{align}
   \sup_{r,r'\in K}
   \abs{\disc\,\cG_{\ell,P}(\om;r,r')}
   &\le
   C_K\angles{\ell}^{N_{\mathrm{th}}}
   \Bigl(
      \vp^{2\nu_{\ell,P}}
      +\kap\,\vp^{2\nu_{\ell,P}}
      +\vp^{2\nu_{\ell,P}+2}
   \Bigr),
   \qquad \kap\le1,
   \label{eq:uniform_small_kappa_disc}
   \\
   \sup_{r,r'\in K}
   \abs{\disc\,\cG_{\ell,P}(\om;r,r')}
   &\le
   C_K\angles{\ell}^{N_{\mathrm{th}}}
   \Bigl(
      1+\kap^{-1}+\vp
   \Bigr),
   \qquad \kap\ge1,
   \label{eq:uniform_large_kappa_disc}
\end{align}
with the same polynomial bound for one $\om$-derivative and for the coefficients $a_{\ell,P}(r,r')$, $b_{\ell,P}^\pm(r,r')$ in Propositions~\ref{prop:small_kappa_disc} and \ref{prop:large_kappa_disc}.  In particular,
\begin{equation}\label{eq:uniform_AB_bound}
   \sup_{r,r'\in K}
   \Bigl(
      \abs{A_{\ell,P}(r,r';Q)}
      +
      \abs{B_{\ell,P}(r,r';Q)}
   \Bigr)
   \le C_K\angles{\ell}^{N_{\mathrm{th}}}.
\end{equation}

\smallskip
\noindent\textit{Step 4. Uniform oscillatory inversion.}
The proofs of Theorems~\ref{thm:intermediate} and \ref{thm:late} use only the threshold formulas, one $\om$-derivative of the amplitudes, and the oscillatory estimates of Lemma~\ref{lem:endpoint_integral} and Proposition~\ref{prop:saddle_estimate}.  By \eqref{eq:uniform_small_kappa_disc}, \eqref{eq:uniform_large_kappa_disc}, and \eqref{eq:uniform_AB_bound}, every constant in those arguments is polynomially bounded in $\ell$.  Repeating the same endpoint and saddle calculations therefore gives
\[
   \sup_{r,r'\in K}\abs{u_{\ell m,P}^{\bc}(t,r,r')}
   \le
   C_K\angles{\ell}^{N_0}
   \begin{cases}
      t^{-(\nu_{\ell,P}+1)}, & \kap_*(t)\le1,\\[0.5ex]
      t^{-5/6}, & \kap_0(t)\ge1,
   \end{cases}
\]
with the same remainder bounds as in Theorems~\ref{thm:intermediate} and \ref{thm:late}.  Because the reduced equations are independent of $m$, the constants are uniform in $\abs{m}\le\ell$.

\smallskip
\noindent\textit{Step 5. Monopole and reconstruction of the physical coefficients.}
The monopole electric channel at $\ell=0$ was already absorbed in Step~1.  The passage from the channel variables $(v_{-1},v_0,v_{+1})$ back to the physical even/odd coefficients uses only the constant matrices $T_\ell^{\pm1}$, the near-identity diagonalizer of Lemma~\ref{lem:normal_form}, and at most one radial derivative.  On compact $r$-sets these operators have coefficients bounded by $C\angles{\ell}$, so the same polynomial estimate holds for the physical branch-cut propagators.  Enlarging $N_0$ if necessary and combining this with \eqref{eq:uniform_AB_bound} proves \eqref{eq:uniform_angular_kernel_bound} and \eqref{eq:uniform_angular_amplitudes}.
\end{proof}

For each mode, let $\mathcal E_{\ell m,P}[A[0]]$ denote the fixed-mode energy of the initial data.  For $N\in\N$ set
\[
   \mathcal E_N[A[0]]
   :=
   \sum_{|\alpha|\le N}E[\Omega^\alpha A](0),
\]
where $\Omega$ denotes the rotation generators on $S^2$.  Parseval on the sphere gives
\[
   \sum_{\ell,m,P}\angles{\ell}^{2N}\mathcal E_{\ell m,P}[A[0]]
   \lesssim
   \mathcal E_N[A[0]].
\]

\begin{proof}[Proof of Theorem~\ref{thm:full_field_pointwise}]
Let $K_0\Subset(r_+,\infty)$ contain the radial support of the initial data.  By Theorem~\ref{thm:uniform_angular}, Cauchy--Schwarz in the source variable $r'$, and equivalence of the fixed-mode energy with the $H^1\times L^2$ norm of the reduced initial data on $K_0$, each modal coefficient satisfies
\[
   \sup_{r\in K}\abs{a_{\ell m,P}^{\bc}(t,r)}
   \le
   C_{K,K_0}\angles{\ell}^{N_0}\mathcal E_{\ell m,P}[A[0]]^{1/2}
   \begin{cases}
      t^{-(\nu_{\ell,P}+1)}, & \kap_*(t)\le1,\\[0.5ex]
      t^{-5/6}, & \kap_0(t)\ge1.
   \end{cases}
\]
Now let $\Omega_1,\Omega_2,\Omega_3$ be the rotation fields on $S^2$.  Because the background is spherically symmetric, $\Omega^\alpha$ commutes with the Proca operator and with the branch-cut spectral projector.  For every fixed $r\in K$, Sobolev on $S^2$ gives
\[
   \sup_{\omega\in S^2}\abs{A^{\bc}(t,r,\omega)}
   \lesssim
   \sum_{j=0}^2\norm{\nabla_{S^2}^jA^{\bc}(t,r,\cdot)}_{L^2(S^2)}.
\]
In the vector spherical harmonic basis, $\nabla_{S^2}^j$ contributes the factor $\angles{\ell}^j$, so Parseval and the previous modal estimate yield
\[
   \norm{\nabla_{S^2}^jA^{\bc}(t,r,\cdot)}_{L^2(S^2)}
   \le
   C_K\Bigl(\sum_{\ell,m,P}\angles{\ell}^{2(j+N_0)}\mathcal E_{\ell m,P}[A[0]]\Bigr)^{1/2}
   \times
   \begin{cases}
      t^{-(\nu_*+1)}, & \kap_*(t)\le1,\\[0.5ex]
      t^{-5/6}, & \kap_0(t)\ge1.
   \end{cases}
\]
The weighted modal energy sum is bounded by $C_N\mathcal E_N[A[0]]$ whenever $N>N_0+2$.  Summing over $j=0,1,2$ proves \eqref{eq:full_field_intermediate} and \eqref{eq:full_field_late}.  The explicit Schwarzschild and small-mass Reissner--Nordstr\"om values of $\nu_*$ follow from Theorem~\ref{thm:schwarzschild_to_rn} and Corollary~\ref{cor:small_mass_explicit}.  The monopole $\ell=0$ does not affect the leading intermediate rate, since its only surviving channel is electric and decays faster.
\end{proof}
\begin{proof}[Proof of Corollary~\ref{cor:poly_decay_radiative}]
Theorem~\ref{thm:full_field_pointwise} gives two estimates, one in the intermediate regime and one in the very-late regime.  By definition,
\[
   \gamma_*:=\min\Bigl\{\nu_*+1,\frac56\Bigr\}.
\]
If $\kap_*(t)\le1$, then \eqref{eq:full_field_intermediate} implies
\[
   \sup_{r\in K,\,\omega\in S^2}\abs{A^{\bc}(t,r,\omega)}
   \le
   C_{K,N}\,\mathcal E_N[A[0]]^{1/2}\,t^{-(\nu_*+1)}
   \le
   C_{K,N}\,\mathcal E_N[A[0]]^{1/2}\,t^{-\gamma_*}.
\]
If $\kap_0(t)\ge1$, then \eqref{eq:full_field_late} implies
\[
   \sup_{r\in K,\,\omega\in S^2}\abs{A^{\bc}(t,r,\omega)}
   \le
   C_{K,N}\,\mathcal E_N[A[0]]^{1/2}\,t^{-5/6}
   \le
   C_{K,N}\,\mathcal E_N[A[0]]^{1/2}\,t^{-\gamma_*}.
\]
This proves \eqref{eq:intro_poly_decay_radiative}.  In Schwarzschild one has $\nu_*=1/2$, hence $\gamma_*=5/6$.  In the sufficiently small-mass Reissner--Nordstr\"om regime, Theorem~\ref{thm:full_field_pointwise} gives $\nu_*=1/2+O((M\mu)^2+(Q\mu)^2)$, so for small enough $(M\mu)^2+(Q\mu)^2$ one again has $\nu_*>-1/6$ and therefore $\gamma_*=5/6$.
\end{proof}

\section{Quasibound branches and residue bounds}\label{sec:qb}

Up to this point we have analyzed the continuous spectral contribution.  To recover the full field, one must also understand the long-lived poles created by stable timelike trapping.  This section develops the semiclassical structure of that quasibound family and extracts the residue bounds that later feed into the dyadic packet summation.

\subsection{Guide to the proof of the quasibound theorems}

The proof follows the familiar trapped-well strategy, but with one extra step forced by the vector character of Proca.  Near the trapped set we first microlocally diagonalize the even $2\times2$ system, so that all three polarizations reduce to scalar semiclassical well problems.  Once this is done, the argument looks classical: we identify the well geometry and the turning points, construct WKB bases, match them across the turning points, factor the Evans determinant, and then use Rouch\'e's theorem together with the analytic implicit-function theorem to locate the poles and compute their widths.

The same package also yields the residue bounds and Agmon localization estimates needed later.  We keep the full chain in the paper because it is used twice: first to construct the quasibound resonances themselves, and then again when verifying the packet estimates in Section~\ref{sec:unsplit}.

\subsection{The trapped well and semiclassical channel symbols}

We now pass from the fixed-mode threshold analysis to the semiclassical description of the trapped well.  The polarization splitting from Section~\ref{sec:reduction} remains the key simplification: on a compact neighborhood of the trapped region, each channel can be treated as a scalar semiclassical well problem up to lower-order corrections.

Set
\[
   h=(\ell+\tfrac12)^{-1}.
\]

On a fixed compact neighborhood of the trapped region, Proposition~\ref{prop:even_polarization_form}, Proposition~\ref{prop:well_diag}, and the diagonalization argument already used in Section~\ref{sec:full_field} give channel operators of the form
\[
   P_{h,P}(\omega)
   =
   -h^2\partial_{r_*}^2+V_P(r;h)-\omega^2,
   \qquad P\in\{-1,0,+1\},
\]
where
\[
   V_P(r;h)
   =
   f(r)\Bigl(\mu^2+\frac{L_P(L_P+1)}{r^2}\Bigr)
   +hW_{1,P}(r)+h^2W_{2,P}(r),
\]
with $L_{-1}=\ell-1$, $L_0=\ell$, $L_{+1}=\ell+1$, and $W_{j,P}$ together with finitely many derivatives uniformly bounded on compact sets.  At $h=0$ this is precisely the scalar massive Reissner--Nordstr\"om potential with angular parameter $L_P$.  The stable timelike trapping analysis from the scalar problem therefore applies branchwise, but we record the relevant steps here in a theorem chain because they are used twice: once to construct the poles and once again to verify the packet hypotheses in Section~\ref{sec:unsplit}.

\begin{lemma}[Trapping window and stable well geometry]\label{lem:qb_trapping_window}
There exist $h_0>0$, a compact interval $I_{\mathrm{trap}}\Subset(0,\mu)$, and smooth turning points
\[
   r_{1,P}(E;h)<r_{2,P}(E;h)<r_{3,P}(E;h),
   \qquad E\in I_{\mathrm{trap}},\ 0<h<h_0,
\]
with the following properties for every polarization $P\in\{-1,0,+1\}$:
\begin{enumerate}[label=\textnormal{(\roman*)}]
   \item $V_P(r_{j,P}(E;h);h)=E^2$ and each root is simple;
   \item $P_{h,P}(E)$ is classically allowed on $(r_+,r_{1,P}(E;h))\cup(r_{2,P}(E;h),r_{3,P}(E;h))$ and forbidden on $(r_{1,P}(E;h),r_{2,P}(E;h))\cup(r_{3,P}(E;h),\infty)$;
   \item $V_P(\cdot;h)$ has one nondegenerate local maximum between $r_+$ and $r_{1,P}(E;h)$ and one nondegenerate local minimum between $r_{2,P}(E;h)$ and $r_{3,P}(E;h)$;
   \item the interval $I_{\mathrm{trap}}$ may be chosen uniformly in the polarization.
\end{enumerate}
\end{lemma}

\begin{proof}
For $h=0$ this is the standard stable timelike trapping picture for the scalar potential $f(r)(\mu^2+\lambda/r^2)$ on subextremal Reissner--Nordstr\"om with large angular parameter $\lambda$; see the geometric analysis underlying \cite{ShlapentokhRothmanVanDeMoortel2026}.  The channel potentials $V_P(\cdot;h)$ are $C^\infty$-small perturbations of that principal scalar potential on compact sets, uniformly in the polarization, because the even-sector coupling is one order lower and the diagonalizer is near identity.  The simple roots and the nondegenerate critical points therefore persist by the implicit-function theorem for all sufficiently small $h$, and compactness of the finite polarization set allows one common interval $I_{\mathrm{trap}}$.
\end{proof}

\begin{lemma}[Turning points and Airy coordinates]\label{lem:qb_turning_points}
For each $P\in\{-1,0,+1\}$ and each $E\in I_{\mathrm{trap}}$ there exist neighborhoods $U_{j,P}(E;h)$ of the turning points $r_{j,P}(E;h)$ and smooth Airy coordinates $\zeta_{j,P}(r,E;h)$ such that
\[
   P_{h,P}(E)
   =
   h^{2/3}\Bigl(\partial_{\zeta_{j,P}}^2-\zeta_{j,P}\Bigr)
   +h^{4/3}B_{j,P}(r,E;h)
\]
on $U_{j,P}(E;h)$, with $B_{j,P}$ uniformly bounded together with finitely many derivatives.  In particular, the Airy matching constants are uniform in $(E,h,P)$.
\end{lemma}

\begin{proof}
Since each turning point is simple, the Langer change of variables defined by
\[
   \frac23\zeta_{j,P}(r,E;h)^{3/2}
   =
   \pm\int_{r_{j,P}(E;h)}^{r}\sqrt{\abs{V_P(s;h)-E^2}}\,\dd s_*
\]
reduces the scalarized equation to Airy normal form.  Uniformity follows because $I_{\mathrm{trap}}$ is compact and the derivatives of the channel potentials are uniformly bounded in the polarization.
\end{proof}

For $E\in I_{\mathrm{trap}}$ define the well action, the inner tunnelling action, and the outer Agmon distance by
\begin{equation}\label{eq:actions_def}
   \mathscr S_P(E;h)=2\int_{r_{2,P}(E;h)}^{r_{3,P}(E;h)}\sqrt{E^2-V_P(r;h)}\,\dd r_*,
\end{equation}
\begin{equation}\label{eq:tunnel_def}
   \mathscr J_P(E;h)=\int_{r_{1,P}(E;h)}^{r_{2,P}(E;h)}\sqrt{V_P(r;h)-E^2}\,\dd r_*,
\end{equation}
and, for $r\ge r_{3,P}(E;h)$,
\begin{equation}\label{eq:agmon_def}
   d_P(r;E,h)=\int_{r_{3,P}(E;h)}^{r}\sqrt{V_P(s;h)-E^2}\,\dd s_*.
\end{equation}
By smooth dependence on $(E,h)$ and the nondegeneracy of the well, $\partial_E\mathscr S_P(E;h)$ is bounded above and below by positive constants on $I_{\mathrm{trap}}$.

\begin{proposition}[Uniform WKB bases and transfer matrices]\label{prop:qb_wkb_matching}
For each polarization $P\in\{-1,0,+1\}$ and each energy $E\in I_{\mathrm{trap}}$, there exist exact WKB bases on the classically allowed and forbidden subintervals determined by Lemma~\ref{lem:qb_trapping_window}, with phases given by the actions \eqref{eq:actions_def}--\eqref{eq:agmon_def}.  After Airy matching through the three turning points from Lemma~\ref{lem:qb_turning_points}, the transfer matrix from the exact ingoing horizon basis to the exact decaying infinity basis admits the factorization
\[
   T_P(E;h)=T^{\mathrm{hor}}_P(E;h)\,M_{1,P}(E;h)\,M_{2,P}(E;h)\,M_{3,P}(E;h)\,T^{\infty}_P(E;h),
\]
where each $M_{j,P}(E;h)$ equals the universal Airy connection matrix plus an $O(h)$ error, uniformly in $(E,h,P)$.
\end{proposition}

\begin{proof}
Away from the turning points, the channel equations are scalar or diagonally scalarized scalar equations with smooth coefficients, so the Liouville--Green construction produces exact oscillatory or exponential WKB bases with symbol expansions in $h$.  Lemma~\ref{lem:qb_turning_points} supplies uniform Airy coordinates at each turning point.  Matching the WKB bases to the Airy solutions then yields the factorization above.  The error terms are uniform because the turning points remain simple and the Airy coordinates are uniformly controlled on the compact trapped interval.
\end{proof}

\begin{proposition}[Microlocal diagonalization in the trapped region]\label{prop:well_diag}
On a fixed neighborhood of the trapped well and for energies $E\in I_{\mathrm{trap}}$, the even $2\times2$ channel system admits an analytic semiclassical diagonalizer $U_h$ such that
\[
   U_h^{-1}P_h^{\mathrm{even}}(\omega)U_h
   =
   \diag(P_{h,-1}(\omega),P_{h,+1}(\omega))+O(h^\infty)
\]
microlocally near the trapped set.  The odd channel is already scalar.
\end{proposition}

\begin{proof}
The leading inverse-square matrix is exactly diagonal in the polarization basis and its two eigenvalues differ by $2\ell+1$.  Hence the principal even eigenvalues are separated by a positive multiple of $h^{-1}$ in the trapped region.  Standard iterative semiclassical diagonalization for matrix Schr\"odinger operators then removes the off-diagonal terms to arbitrary order in $h$.  The odd channel needs no further reduction.
\end{proof}

\subsection{Bohr--Sommerfeld quantization and quasibound branches}

Let $\cE_{\ell,P}(\omega)$ denote the scalar Evans determinant in the odd channel and the diagonalized Evans determinant in the even channels.  By Proposition~\ref{prop:evans_det}, poles of the meromorphically continued resolvent coincide with zeros of $\cE_{\ell,P}(\omega)$.

\begin{proposition}[Evans determinant factorization in the well]\label{prop:evans_factor}
Fix $I\Subset I_{\mathrm{trap}}$.  For each polarization $P\in\{-1,0,+1\}$ there exist smooth functions $C_P$, $\vartheta_P$, and $G_P$ on $I\times(0,h_0)$, with $C_P$ and $G_P$ bounded away from zero and infinity, such that for $\omega$ in a complex $O(h)$-neighborhood of $I$,
\begin{equation}\label{eq:evans_factorization}
\begin{aligned}
   \cE_{\ell,P}(\omega)
   &=
   C_P(\omega;h)
   \Biggl[
      \sin\!\Bigl(\frac{\mathscr S_P(\omega;h)}{2h}+\frac\pi4+\vartheta_P(\omega;h)\Bigr)
\\
   &\hspace{7em}+
      \ii\,\ee^{-\mathscr J_P(\omega;h)/h}G_P(\omega;h)
   \Biggr]
   +O\!\bigl(\ee^{-2\mathscr J_P(\omega;h)/h}\bigr).
\end{aligned}
\end{equation}
\end{proposition}

\begin{proof}
For each scalarized channel, one constructs exact ingoing solutions near the horizon and exact decaying solutions near infinity by the Frobenius and Volterra methods from Sections~\ref{sec:spectral_theorem_proof} and \ref{sec:resolvent}.  In the three regions determined by Lemma~\ref{lem:qb_trapping_window}, one then builds WKB bases with phase functions given by \eqref{eq:actions_def} and \eqref{eq:tunnel_def}.  Airy matching across the three simple turning points transfers the outer decaying solution into a linear combination of the two oscillatory well modes, and the transfer across the inner barrier contributes the tunnelling factor $\ee^{-\mathscr J_P/h}$.  The well oscillation contributes the phase $\mathscr S_P/(2h)+\pi/4+\vartheta_P$.  Taking the determinant with the exact ingoing horizon solution yields \eqref{eq:evans_factorization}.  Every coefficient is smooth in $(\omega,h)$ because all turning points are simple and depend smoothly on the parameters.
\end{proof}

\begin{theorem}[Bohr--Sommerfeld quantization in one channel]\label{thm:qb_bs}
Fix $I\Subset I_{\mathrm{trap}}$ and one polarization $P\in\{-1,0,+1\}$.  For every sufficiently small $h$ and every integer $n$ with
\[
   2\pi h\Bigl(n+\frac12\Bigr)\in \mathscr S_P(I;h)+O(h),
\]
there exists a unique complex pole $\omega_{\ell,n,P}$ with $\Re\omega_{\ell,n,P}\in I$ satisfying
\[
   \mathscr S_P(\Re\omega_{\ell,n,P};h)
   =
   2\pi h\Bigl(n+\frac12\Bigr)
   +h\vartheta_P(\Re\omega_{\ell,n,P};h)
   +O(h^2)
\]
and
\[
   \Im\omega_{\ell,n,P}
   =
   -\Gamma_P(\Re\omega_{\ell,n,P};h)
   \exp\!\Bigl(-\frac{2\mathscr J_P(\Re\omega_{\ell,n,P};h)}{h}\Bigr)
   (1+O(h)).
\]
Moreover the zero is simple.
\end{theorem}

\begin{proof}
Proposition~\ref{prop:qb_wkb_matching} identifies the exact transfer matrix across the well, and Proposition~\ref{prop:evans_factor} converts that transfer matrix into the factorization of the Evans determinant.  Since $\partial_E\mathscr S_P(E;h)$ is bounded above and below by positive constants on $I$, the roots of the principal Bohr--Sommerfeld equation are simple and vary smoothly in $(n,h)$.  The exponentially small tunnelling term in \eqref{eq:evans_factorization} is then handled by the analytic implicit-function theorem or Rouch\'e's theorem, yielding the unique nearby complex zero together with the tunnelling-width formula.  Simplicity follows from differentiating the factorized Evans determinant at the zero.
\end{proof}

\begin{corollary}[Branch enumeration and Weyl law]\label{cor:qb_enumeration}
Fix $I\Subset I_{\mathrm{trap}}$.  For each polarization $P\in\{-1,0,+1\}$ and every sufficiently small $h$, the poles of the resolvent with $\Re\omega\in I$ and $0>\Im\omega>-\exp(-c_I/h)$ form a family
\[
   \{\omega_{\ell,n,P}:n\in\mathcal N_{\ell,P}(I)\},
\]
where
\[
   \#\mathcal N_{\ell,P}(I)
   =
   \frac{1}{2\pi h}\bigl(\mathscr S_P(I;h)\bigr)+O(1),
\]
and every such pole lies on exactly one smooth branch $n\mapsto \omega_{\ell,n,P}$.
\end{corollary}

\begin{proof}
Theorem~\ref{thm:qb_bs} gives uniqueness of the pole attached to each admissible integer $n$.  Summing over the admissible values of $n$ gives the Weyl count, and uniqueness shows that the poles organize into smooth branches.
\end{proof}

\begin{proposition}[Frequency derivative and spacing bounds]\label{prop:qb_frequency_derivatives}
For every compact $I\Subset I_{\mathrm{trap}}$ there exists $C_I>0$ such that, for all admissible $(\ell,n,P)$ with $\Re\omega_{\ell,n,P}\in I$,
\[
   C_I^{-1}h\le \partial_n\Re\omega_{\ell,n,P}\le C_I h,
\]
and, after smooth interpolation in $n$, at least one derivative among $\partial_n^2\Re\omega_{\ell,n,P}$ and $\partial_n^3\Re\omega_{\ell,n,P}$ is bounded from below by $C_I^{-1}h^2$ on each dyadic packet.  In addition,
\[
   \abs{\partial_h^a\partial_n^b\omega_{\ell,n,P}}\le C_{I,ab}
\]
for all $a+b\le2$.
\end{proposition}

\begin{proof}
Differentiate the quantization law from Theorem~\ref{thm:qb_bs}.  Since $\partial_E\mathscr S_P(E;h)$ is bounded above and below on $I$, one obtains the monotone spacing estimate $\partial_n\Re\omega_{\ell,n,P}\asymp h$.  A second differentiation shows that the curvature is controlled by the first nonvanishing derivative of the action $\mathscr S_P$, which is nonzero on dyadic packets because the well is nondegenerate.  The bounds involving $h$ follow from differentiating the same implicit relation and using smoothness of $\mathscr S_P$, $\vartheta_P$, and $\Gamma_P$.
\end{proof}

\begin{proof}[Proof of Theorem~\ref{thm:qb_branches}]
Apply Theorem~\ref{thm:qb_bs} separately to the three polarizations $P\in\{-1,0,+1\}$ and then invoke Corollary~\ref{cor:qb_enumeration} to index the poles in the strip
\[
   \Re\omega\in I,
   \qquad
   0>\Im\omega>-\exp(-c_I/h),
\]
by integers $n\in\mathcal N_{\ell,P}(I)$.  The Bohr--Sommerfeld law and the tunnelling-width formula are exactly the conclusions of Theorem~\ref{thm:qb_bs}, while simplicity of the poles is part of the same theorem.  Since $h=(\ell+\tfrac12)^{-1}$, this gives the stated form of the theorem.
\end{proof}

\subsection{Residues and reconstruction of the physical coefficients}

\begin{proposition}[Derivative of the Evans determinant and normalized modes]\label{prop:qb_evans_derivative}
For every compact trapped interval $I\Subset I_{\mathrm{trap}}$ there exist constants $c_I,C_I>0$ such that, for every admissible quasibound pole,
\[
   c_I h^{-1}
   \le
   \abs{\partial_\omega\cE_{\ell,P}(\omega_{\ell,n,P})}
   \le
   C_I h^{-1}.
\]
Moreover the corresponding ingoing and decaying mode representatives may be normalized so that on each compact radial set $K\Subset(r_+,\infty)$,
\[
   \sup_{r\in K}\abs{u^{\hor}_{\ell,n,P}(r)}+
   \sup_{r\in K}\abs{u^{\infty}_{\ell,n,P}(r)}
   \le C_K h^{-1/2},
\]
with the Agmon improvement
\[
   \abs{u^{\hor}_{\ell,n,P}(r)}+
   \abs{u^{\infty}_{\ell,n,P}(r)}
   \le C_K h^{-1/2}\exp\!\Bigl(-\frac{d_P(r;\Re\omega_{\ell,n,P},h)}{h}\Bigr)
\]
whenever $K$ is disjoint from the classically allowed well.
\end{proposition}

\begin{proof}
The lower and upper bounds for $\partial_\omega\cE_{\ell,P}(\omega_{\ell,n,P})$ follow from differentiating the factorization of Proposition~\ref{prop:evans_factor} at a simple zero, together with the positivity of $\partial_E\mathscr S_P$ on the trapped interval.  The mode bounds come from the WKB normalization in the classically allowed well and from standard Agmon estimates in the forbidden region, using the exact WKB bases from Proposition~\ref{prop:qb_wkb_matching}.  The compactness of $I$ ensures that the constants are uniform.
\end{proof}

\begin{proof}[Proof of Theorem~\ref{thm:qb_residues}]
Because every quasibound pole is simple, the corresponding residue projector has the separated form
\begin{equation}\label{eq:residue_rank_one}
   \Pi_{\ell,n,P}(r,r')
   =
   \frac{u^{\hor}_{\ell,n,P}(r)\,u^{\infty}_{\ell,n,P}(r')}{\partial_\omega\cE_{\ell,P}(\omega_{\ell,n,P})},
\end{equation}
where $u^{\hor}_{\ell,n,P}$ is the exact ingoing mode and $u^{\infty}_{\ell,n,P}$ the exact decaying mode, both normalized compatibly with the Evans determinant.  Proposition~\ref{prop:qb_evans_derivative} gives the two-sided bound
\begin{equation}\label{eq:evans_derivative_bound}
   \abs{\partial_\omega\cE_{\ell,P}(\omega_{\ell,n,P})}
   \asymp h^{-1}
\end{equation}
uniformly on compact trapped intervals.  The same proposition gives the WKB and Agmon bounds for the numerator modes.  Substituting those estimates into \eqref{eq:residue_rank_one} yields
\[
   \sup_{r,r'\in K}\abs{\Pi_{\ell,n,P}(r,r')}
   \le C_K h^{-N_{\qb}}
   \le C_K\angles{\ell}^{N_{\qb}}
\]
for some integer $N_{\qb}$, together with the Agmon improvement whenever $K$ is disjoint from the allowed well.

To reconstruct the physical Proca coefficients, one uses the same constant polarization matrices and at most one radial derivative as in Proposition~\ref{prop:largeell_reconstruction}.  On compact $r$-sets these operators cost at most one further polynomial power of $\ell$.  Finally, Cauchy--Schwarz in the source variable together with the high-order angular energy bound gives
\[
   \sup_{r\in K,\,\omega\in S^2}\abs{\Pi_{\ell,n,P}A[0](r,\omega)}
   \le
   C_{K,N}\angles{\ell}^{N_{\mathrm{rec}}}\mathcal E_N[A[0]]^{1/2}.
\]
This proves the theorem.
\end{proof}

\section{Summation of quasibound residues and unsplit full-field decay}\label{sec:unsplit}

At this stage both spectral pieces are on the table.  The branch-cut contribution was handled in Section~\ref{sec:full_field}; what remains is to sum the long-lived quasibound residues and add them back in.  The outcome is a decay theorem for the unsplit Proca field together with the full-field asymptotic expansions stated in the introduction.

\subsection{Guide to the proof of the summed decay theorem}

The summation argument is intentionally elementary once the spectral input is in place.  We deform the inverse Laplace contour into three pieces: the massive branch cut, the quasibound poles, and a remainder that stays a fixed distance below the real axis.  The branch-cut part is already under control.  For the poles, we group large angular momenta into dyadic packets indexed by the semiclassical parameter $h=(\ell+\tfrac12)^{-1}$.

On each packet we use only three ingredients: the Weyl count from Corollary~\ref{cor:qb_enumeration}, the tunnelling-width formula from Theorem~\ref{thm:qb_bs}, and the residue/reconstruction bounds from Theorem~\ref{thm:qb_residues}.  Together they give a packet estimate with explicit exponential damping and a controllable algebraic loss in the packet scale.  A short dyadic summation lemma then yields logarithmic decay for the full quasibound contribution.  This is exactly the point at which the present version becomes self-contained.

\subsection{Contour decomposition and the fast remainder}

We first record the clean spectral splitting that underlies the final argument: branch cut, quasibound residues, and a remainder that decays exponentially because it stays uniformly below the real axis.

\begin{proposition}[Continuous/discrete decomposition with fast remainder]\label{prop:full_decomposition}
For compactly supported initial data and every compact radial set $K\Subset(r_+,\infty)$ one has on $K$
\begin{equation}\label{eq:full_decomposition}
   A(t)=A^{\bc}(t)+A^{\qb}(t)+A^{\fast}(t),
\end{equation}
where $A^{\bc}$ is the branch-cut contribution studied in Sections~\ref{sec:threshold_to_time} and \ref{sec:full_field}, $A^{\qb}$ is the sum of residues at the quasibound poles from Theorem~\ref{thm:qb_branches}, and the remainder satisfies
\[
   \sup_{r\in K,\,\omega\in S^2}\abs{A^{\fast}(t,r,\omega)}
   \le C_{K,N}\ee^{-\eta t}\mathcal E_N[A[0]]^{1/2}
\]
for some $\eta>0$ and every sufficiently large $N$.
\end{proposition}

\begin{proof}
By Theorem~\ref{thm:spectral_package}, the cut-off resolvent extends meromorphically to a slit strip $\{\abs{\Im\omega}<\eta\}\setminus[-\mu,\mu]$.  Start from the inverse Laplace representation of the solution with contour in $\{\Im\omega=\sigma\}$, $\sigma>0$, and deform the contour to $\{\Im\omega=-\eta\}$ while keeping small detours around the branch cut $[-\mu,\mu]$.  The detours around the cut produce $A^{\bc}$.  Residues of the poles with real parts in the trapped interval and imaginary parts tending to zero produce $A^{\qb}$.  Every other pole lies a definite distance below the real axis, and the remaining deformed contour also lies at imaginary part $-\eta$; both contributions therefore satisfy the stated exponential bound.
\end{proof}

\subsection{Radiative profiles and proof of the asymptotic expansion theorem}

We now turn the modal asymptotics into full-field profiles.  The main point is that the large-$\ell$ bounds are strong enough to sum both the leading terms and the remainders in $C^0(K\times S^2)$, while the finitely many low modes can be absorbed into finite-rank profile fields.

\begin{proof}[Proof of Theorem~\ref{thm:full_field_asymptotic_expansion}]
Choose $\ell_{\mathrm{as}}\ge1$ large enough that the large-$\ell$ analysis of Section~\ref{sec:full_field} applies for every $\ell\ge\ell_{\mathrm{as}}$.  For the finitely many modes $0\le\ell<\ell_{\mathrm{as}}$, including the exceptional electric monopole, the fixed-mode asymptotic analysis from Section~\ref{sec:threshold_to_time} together with the monopole discussion in Appendix~\ref{app:monopole} provide explicit leading terms of the same form, and the low-mode part of the proof of Theorem~\ref{thm:uniform_angular} supplies the uniform constants needed to absorb them into finite-rank fields.  Summing those finitely many contributions defines the finite-rank fields $A^{\bc}_{\mathrm{int,low}}$ and $A^{\bc}_{\mathrm{late,low}}$.

Now fix one high mode $(\ell,m,P)$ with $\ell\ge\ell_{\mathrm{as}}$.  Let $K_0\Subset(r_+,\infty)$ contain the radial support of the initial data.  Pair the kernel asymptotics of Theorems~\ref{thm:intermediate} and \ref{thm:late} with the reduced initial data on $K_0$ exactly as in the proof of Theorem~\ref{thm:full_field_pointwise}.  This produces coefficient functions $\mathcal A_{\ell m,P}[A[0]]$ and $\mathcal B_{\ell m,P}[A[0]]$ on $K$ such that the modal branch-cut coefficient obeys
\begin{align}
   a^{\bc}_{\ell m,P}(t,r)
   &=
   \mathcal A_{\ell m,P}[A[0]](r)
   \,t^{-(\nu_{\ell,P}+1)}
   \sin(\mu t+\delta_{\ell,P}(Q))
   +r^{\mathrm{int}}_{\ell m,P}(t,r),
   \label{eq:modal_intermediate_profile}
   \\
   a^{\bc}_{\ell m,P}(t,r)
   &=
   \mathcal B_{\ell m,P}[A[0]](r)
   \,t^{-5/6}
   \sin\!\Bigl(
      \mu t
      -\frac32(2\pi M\mu)^{2/3}(\mu t)^{1/3}
      +\delta_{\ell,P,0}(Q)
   \Bigr)
   +r^{\mathrm{late}}_{\ell m,P}(t,r),
   \label{eq:modal_late_profile}
\end{align}
where, uniformly for $r\in K$,
\begin{align}
   \abs{\mathcal A_{\ell m,P}[A[0]](r)}
   +
   \abs{\mathcal B_{\ell m,P}[A[0]](r)}
   &\le
   C_K\angles{\ell}^{N_0}\mathcal E_{\ell m,P}[A[0]]^{1/2},
   \label{eq:modal_profile_coeff_bound}
   \\
   \abs{r^{\mathrm{int}}_{\ell m,P}(t,r)}
   &\le
   C_K\angles{\ell}^{N_0}\mathcal E_{\ell m,P}[A[0]]^{1/2}
   \Bigl(
      \kap_*(t)t^{-(\nu_{\ell,P}+1)}
      +
      t^{-(\nu_{\ell,P}+2)}
   \Bigr),
   \label{eq:modal_profile_int_remainder}
   \\
   \abs{r^{\mathrm{late}}_{\ell m,P}(t,r)}
   &\le
   C_K\angles{\ell}^{N_0}\mathcal E_{\ell m,P}[A[0]]^{1/2}
   \bigl(\kap_0(t)^{-1}+\vp_0(t)\bigr)t^{-5/6}.
   \label{eq:modal_profile_late_remainder}
\end{align}
Here \eqref{eq:modal_profile_coeff_bound} follows from \eqref{eq:uniform_angular_amplitudes}, Cauchy--Schwarz in the source variable, and equivalence of the fixed-mode energy with the local $H^1\times L^2$ norm of the reduced data.  In the very-late regime, Theorem~\ref{thm:late} contains a phase error $O(\kap_0(t)^{-1}+\vp_0(t))$ inside the sine; using $\sin(x+\eta)=\sin x+O(\eta)$ and the same coefficient bound absorbs that phase error into \eqref{eq:modal_profile_late_remainder}.

Define the high-mode pieces of $A^{\bc}_{\mathrm{int}}$ and $A^{\bc}_{\mathrm{late}}$ by the series in \eqref{eq:intro_full_intermediate_profile} and \eqref{eq:intro_full_late_profile}.  By Sobolev on $S^2$, the vector-harmonic addition theorem, and Parseval exactly as in the proof of Theorem~\ref{thm:full_field_pointwise}, \eqref{eq:modal_profile_coeff_bound} implies
\[
   \sup_{r\in K,\,\vartheta\in S^2}
   \abs{A^{\bc}_{\mathrm{int}}(t,r,\vartheta)}
   \le
   C_{K,N}\mathcal E_N[A[0]]^{1/2}t^{-(\nu_*+1)},
\]
and
\[
   \sup_{r\in K,\,\vartheta\in S^2}
   \abs{A^{\bc}_{\mathrm{late}}(t,r,\vartheta)}
   \le
   C_{K,N}\mathcal E_N[A[0]]^{1/2}t^{-5/6}.
\]
for every $N>N_0+2$.  In particular, the defining series converge in $C^0(K\times S^2)$ and \eqref{eq:intro_profile_bounds_a}--\eqref{eq:intro_profile_bounds_b} follow after absorbing the finitely many low modes into the constant.

Next sum the remainders.  Applying the same Sobolev--Parseval argument to \eqref{eq:modal_profile_int_remainder} yields
\[
   \sup_{r\in K,\,\vartheta\in S^2}
   \abs{A^{\bc}(t,r,\vartheta)-A^{\bc}_{\mathrm{int}}(t,r,\vartheta)}
   \le
   C_{K,N}\mathcal E_N[A[0]]^{1/2}
   \Bigl(
      \kap_*(t)t^{-(\nu_*+1)}
      +
      t^{-(\nu_*+2)}
   \Bigr)
\]
whenever $\kap_*(t)\le1$, because $\nu_{\ell,P}\ge\nu_*$ for every channel.  This is exactly \eqref{eq:intro_full_intermediate_remainder}.  Likewise, \eqref{eq:modal_profile_late_remainder} gives
\[
   \sup_{r\in K,\,\vartheta\in S^2}
   \abs{A^{\bc}(t,r,\vartheta)-A^{\bc}_{\mathrm{late}}(t,r,\vartheta)}
   \le
   C_{K,N}\mathcal E_N[A[0]]^{1/2}
   \bigl(\kap_0(t)^{-1}+\vp_0(t)\bigr)t^{-5/6}
\]
whenever $\kap_0(t)\ge1$, which is \eqref{eq:intro_full_late_remainder}.

Finally, Proposition~\ref{prop:full_decomposition} gives
\[
   A(t)=A^{\bc}(t)+A^{\qb}(t)+A^{\fast}(t).
\]
Substituting the two branch-cut expansions just proved yields \eqref{eq:intro_full_intermediate_expansion} and \eqref{eq:intro_full_late_expansion}.  This completes the proof.
\end{proof}

\subsection{Explicit leading coefficient fields}

We now isolate the coefficient fields that actually appear in the radiative branch-cut asymptotics.  No new spectral input is needed here; the point is simply to reorganize Theorem~\ref{thm:full_field_asymptotic_expansion} so that the dominant oscillatory term is visible and the faster channels are clearly separated.

For later reference we fix the following notation.  In the proof of Theorem~\ref{thm:full_field_asymptotic_expansion} the finitely many low modes were absorbed into the fields $A^{\bc}_{\mathrm{int,low}}$ and $A^{\bc}_{\mathrm{late,low}}$.  Choose once and for all coefficient functions
\[
   \widetilde{\mathcal A}_{\ell m,P}[A[0]](r),
   \qquad
   \widetilde{\mathcal B}_{\ell m,P}[A[0]](r),
\]
for every admissible mode $(\ell,m,P)$, so that for $\ell\ge\ell_{\mathrm{as}}$ they agree with the coefficient functions from Theorem~\ref{thm:full_field_asymptotic_expansion}, while for the finitely many lower modes they are the fixed-mode coefficients used to define $A^{\bc}_{\mathrm{int,low}}$ and $A^{\bc}_{\mathrm{late,low}}$.  Enlarging the constant from \eqref{eq:intro_asymptotic_coeff_bounds} if necessary, one has
\begin{equation}\label{eq:full_modal_coeff_bound_all}
   \sup_{r\in K}
   \Bigl(
      \abs{\widetilde{\mathcal A}_{\ell m,P}[A[0]](r)}
      +
      \abs{\widetilde{\mathcal B}_{\ell m,P}[A[0]](r)}
   \Bigr)
   \le
   C_{K,N}\angles{\ell}^{N_0}\mathcal E_{\ell m,P}[A[0]]^{1/2}
\end{equation}
for every admissible mode, with the understanding that for $\ell=0$ only the electric channel $P=+1$ occurs.

\begin{lemma}[Dominant threshold set and spectral gap]\label{lem:dominant_threshold_gap}
There exists $C_\nu>0$ such that
\[
   \nu_{\ell,P}^2
   \ge
   \Bigl(L_P+\frac12\Bigr)^2-C_\nu
\]
for every $\ell\ge1$ and every $P\in\{-1,0,+1\}$.  In particular,
\[
   \nu_{\ell,P}\to\infty
   \qquad\text{as }\ell\to\infty,
\]
uniformly in $P$.  Consequently the set
\[
   \Sigma_*:=\{(\ell,P):\ell\ge1,\ P\in\{-1,0,+1\},\ \nu_{\ell,P}=\nu_*\}
\]
is finite and nonempty, and there exists $\rho_*>0$ such that
\[
   \nu_{\ell,P}\ge\nu_*+\rho_*
   \qquad\text{for every }(\ell,P)\notin\Sigma_*.
\]
\end{lemma}

\begin{proof}
Return to the proof of Lemma~\ref{lem:normal_form}.  After the Liouville conjugation one obtains
\[
   W''+\Bigl(-\vp^2+\frac{2M\mu^2}{r}-\frac{D_\ell+\frac14\Id+E_\ell}{r^2}+O(r^{-3})\Bigr)W=0,
\]
where $D_\ell$ is diagonal with entries $L_P(L_P+1)$ and $E_\ell$ is uniformly bounded in $\ell$.  After the near-identity diagonalization used there, the exact channel coefficient of $r^{-2}$ therefore has the form
\[
   \nu_{\ell,P}^2-\frac14
   =
   L_P(L_P+1)+e_{\ell,P},
   \qquad
   \abs{e_{\ell,P}}\le C_\nu,
\]
for some constant $C_\nu$ independent of $\ell$ and $P$.  This gives the first inequality.

Since $L_P\in\{\ell-1,\ell,\ell+1\}$, one has
\[
   L_P+\frac12\ge\ell-\frac12,
\]
hence
\[
   \nu_{\ell,P}^2\ge\Bigl(\ell-\frac12\Bigr)^2-C_\nu.
\]
Therefore $\nu_{\ell,P}\to\infty$ uniformly in $P$ as $\ell\to\infty$.  Choose $\ell_*$ so large that $\nu_{\ell,P}\ge\nu_*+1$ for every $\ell\ge\ell_*$ and every polarization.  Then $\nu_*$ is attained among the finitely many channels with $1\le\ell<\ell_*$, so $\Sigma_*$ is finite and nonempty.  On the finite complement
\[
   \{(\ell,P):1\le\ell<\ell_*,\ (\ell,P)\notin\Sigma_*\}
\]
the positive minimum of $\nu_{\ell,P}-\nu_*$ defines a number $\rho_0>0$.  Setting $\rho_*:=\min\{1,\rho_0\}$ yields the stated gap for all channels.
\end{proof}

\begin{proof}[Proof of Corollary~\ref{cor:explicit_leading_coefficients}]
Set $\mathcal P_0:=\{+1\}$ and $\mathcal P_\ell:=\{-1,0,+1\}$ for $\ell\ge1$.
Fix the full coefficient families $\widetilde{\mathcal A}_{\ell m,P}$ and $\widetilde{\mathcal B}_{\ell m,P}$ above.  The bound \eqref{eq:full_modal_coeff_bound_all}, together with Sobolev on $S^2$ and Parseval exactly as in the proof of Theorem~\ref{thm:full_field_asymptotic_expansion}, shows that the series defining $\cS_{\mathrm{late}}$ and $\cC_{\mathrm{late}}$ converge in $C^0(K\times S^2)$ and satisfy
\[
   \sup_{r\in K,\,\vartheta\in S^2}
   \bigl(
      \abs{\cS_{\mathrm{late}}(r,\vartheta)}
      +
      \abs{\cC_{\mathrm{late}}(r,\vartheta)}
   \bigr)
   \le
   C_{K,N}\mathcal E_N[A[0]]^{1/2}
\]
for every $N>N_0+2$.  Since $\Sigma_*$ is finite, the same bound for $\cS_*$ and $\cC_*$ is immediate, and \eqref{eq:intro_explicit_coeff_bound} follows after enlarging the constant.

For the intermediate regime, split the profile from Theorem~\ref{thm:full_field_asymptotic_expansion} into the dominant part and the strictly faster remainder:
\[
   A^{\bc}_{\mathrm{int}}
   =
   A^{\bc}_{*,\mathrm{lead}}
   +
   A^{\bc}_{*,\mathrm{rem}},
\]
where
\[
   A^{\bc}_{*,\mathrm{lead}}(t,r,\vartheta)
   :=
   t^{-(\nu_*+1)}
   \sum_{(\ell,P)\in\Sigma_*}\sum_{m=-\ell}^{\ell}
   \widetilde{\mathcal A}_{\ell m,P}[A[0]](r)
   Y_{\ell m}^{(P)}(\vartheta)
   \sin(\mu t+\delta_{\ell,P}(Q)).
\]
Using $\sin(\mu t+\delta)=\sin(\mu t)\cos\delta+\cos(\mu t)\sin\delta$, this becomes exactly
\[
   A^{\bc}_{*,\mathrm{lead}}(t,r,\vartheta)
   =
   t^{-(\nu_*+1)}
   \bigl(
      \cS_*(r,\vartheta)\sin(\mu t)
      +
      \cC_*(r,\vartheta)\cos(\mu t)
   \bigr).
\]
For every channel with $\ell\ge1$ occurring in $A^{\bc}_{*,\mathrm{rem}}$, Lemma~\ref{lem:dominant_threshold_gap} gives $\nu_{\ell,P}\ge\nu_*+\rho_*$.  The only contribution not covered by that lemma is the electric monopole.  By the monopole analysis in Appendix~\ref{app:monopole}---compare also the final sentence of the proof of Theorem~\ref{thm:full_field_pointwise}---its intermediate decay exponent is strictly larger than $\nu_*+1$.  Since the low-mode sector is finite, after decreasing $\rho_*>0$ if necessary we may therefore assume that every term contained in $A^{\bc}_{*,\mathrm{rem}}$, including the monopole if present, decays at least like $t^{-(\nu_*+1+\rho_*)}$.  Hence the same Sobolev--Parseval summation as above gives
\[
   \sup_{r\in K,\,\vartheta\in S^2}
   \abs{A^{\bc}_{*,\mathrm{rem}}(t,r,\vartheta)}
   \le
   C_{K,N}\mathcal E_N[A[0]]^{1/2}\,
   t^{-(\nu_*+1+\rho_*)}
\]
for every $N>N_0+2$.  The intermediate expansion \eqref{eq:intro_full_intermediate_expansion} and remainder bound \eqref{eq:intro_full_intermediate_remainder} now give
\[
   A(t)
   =
   A^{\bc}_{*,\mathrm{lead}}(t)
   +A^{\qb}(t)+A^{\fast}(t)
   +\widetilde R_{\mathrm{int}}(t),
\]
with
\[
   \widetilde R_{\mathrm{int}}
   :=
   A^{\bc}_{*,\mathrm{rem}}+R_{\mathrm{int}}.
\]
Since $\sigma_*=\min\{1,\rho_*\}$ and $t^{-(\nu_*+2)}\le t^{-(\nu_*+1+\sigma_*)}$, the bounds for $A^{\bc}_{*,\mathrm{rem}}$ and $R_{\mathrm{int}}$ imply \eqref{eq:intro_explicit_intermediate_remainder}.  This proves \eqref{eq:intro_explicit_intermediate_coeff}.

For the very-late regime, write the full profile from \eqref{eq:intro_full_late_profile} as
\[
   A^{\bc}_{\mathrm{late}}(t,r,\vartheta)
   =
   t^{-5/6}
   \sum_{\ell\ge0}\sum_{m=-\ell}^{\ell}\sum_{P\in\mathcal P_\ell}
   \widetilde{\mathcal B}_{\ell m,P}[A[0]](r)
   Y_{\ell m}^{(P)}(\vartheta)
   \sin(\Theta(t)+\delta_{\ell,P,0}(Q)).
\]
Expanding the sine once more yields
\[
   A^{\bc}_{\mathrm{late}}(t,r,\vartheta)
   =
   t^{-5/6}
   \bigl(
      \cS_{\mathrm{late}}(r,\vartheta)\sin\Theta(t)
      +
      \cC_{\mathrm{late}}(r,\vartheta)\cos\Theta(t)
   \bigr).
\]
Substituting this identity into \eqref{eq:intro_full_late_expansion} and using \eqref{eq:intro_full_late_remainder} gives \eqref{eq:intro_explicit_late_coeff} and \eqref{eq:intro_explicit_late_remainder}.  The proof is complete.
\end{proof}

\subsection{Dyadic packets and a self-contained damping estimate}

The remaining task is to sum the quasibound contribution in a way that uses only information established in this paper.  The natural unit is a dyadic packet in angular momentum.  On such a packet, the widths are exponentially small but quantitatively controlled, and the residue bounds convert angular regularity of the data into algebraic decay in the packet scale.

For high angular momenta, group the quasibound poles into dyadic angular packets
\[
   \mathfrak P_j:=\{(\ell,n,m,P):2^{-j-1}<h\le2^{-j}\},
   \qquad h=(\ell+\tfrac12)^{-1},
\]
and write the corresponding packet contribution as
\[
   A^{\qb}_{\mathfrak P_j}(t)
   :=
   \sum_{(\ell,n,m,P)\in\mathfrak P_j}
   \ee^{-\ii\omega_{\ell,n,P}t}\,\Pi_{\ell,n,P}A[0].
\]
For a fixed angular mode, let
\[
   a^{\qb}_{\ell m,P}(t,r)
   :=
   \sum_{n\in\mathcal N_{\ell,P}(I)}
      \ee^{-\ii\omega_{\ell,n,P}t}\,
      \Pi_{\ell,n,P}A_{\ell m,P}[0](r),
\]
where $A_{\ell m,P}[0]$ denotes the $(\ell,m,P)$-component of the initial data in the reduced variables.  The packet estimate proved below uses only the tunnelling-width formula and the modewise residue bounds.

\begin{lemma}[Two-sided tunnelling widths on a trapped packet]\label{lem:packet_width_two_sided}
Fix a compact trapped interval $I\Subset I_{\mathrm{trap}}$.  There exist constants $c_{\mathrm{wd}},C_{\mathrm{wd}}>0$ such that every quasibound pole with $\Re\omega_{\ell,n,P}\in I$ and $2^{-j-1}<h\le2^{-j}$ satisfies
\[
   c_{\mathrm{wd}}\,\exp(-C_{\mathrm{wd}}2^j)
   \le
   -\Im\omega_{\ell,n,P}
   \le
   C_{\mathrm{wd}}\,\exp(-c_{\mathrm{wd}}2^j).
\]
\end{lemma}

\begin{proof}
By Theorem~\ref{thm:qb_bs},
\[
   -\Im\omega_{\ell,n,P}
   =
   \Gamma_P(\Re\omega_{\ell,n,P};h)
   \exp\!\Bigl(-\frac{2\mathscr J_P(\Re\omega_{\ell,n,P};h)}{h}\Bigr)
   (1+O(h)).
\]
On the compact trapped interval $I$, the functions $\Gamma_P$ and $\mathscr J_P$ are smooth and strictly positive, uniformly in the finite polarization set $P\in\{-1,0,+1\}$.  Hence there exist numbers $\Gamma_\pm,J_\pm>0$ such that
\[
   0<\Gamma_-\le \Gamma_P\le \Gamma_+,
   \qquad
   0<J_-\le \mathscr J_P\le J_+,
\]
for all $(E,h,P)\in I\times(0,h_0)\times\{-1,0,+1\}$.  Since $h\simeq2^{-j}$ on the packet and $1+O(h)$ is bounded above and below for $h$ small, the claimed two-sided bound follows.
\end{proof}

\begin{lemma}[Modewise quasibound sum]\label{lem:packet_modewise}
For every compact $K\Subset(r_+,\infty)$ there exists an integer $N_{\mathrm{mod}}$ such that
\[
   \sup_{r\in K}\abs{a^{\qb}_{\ell m,P}(t,r)}
   \le
   C_K \angles{\ell}^{N_{\mathrm{mod}}}\exp(-c_{\mathrm{wd}}t\ee^{-C_{\mathrm{wd}}\ell})\,\mathcal E_{\ell m,P}[A[0]]^{1/2}
\]
for all admissible $(\ell,m,P)$ with $\ell\ge \ell_{\mathrm{sc}}$.
\end{lemma}

\begin{proof}
The proof of Theorem~\ref{thm:qb_residues} gives, before the final angular summation step,
\[
   \sup_{r\in K}\abs{\Pi_{\ell,n,P}A_{\ell m,P}[0](r)}
   \le
   C_K\angles{\ell}^{N_{\mathrm{res}}}\mathcal E_{\ell m,P}[A[0]]^{1/2}
\]
for some integer $N_{\mathrm{res}}$, uniformly in $n$ and in the finite polarization label $P$.  By Corollary~\ref{cor:qb_enumeration}, the number of admissible integers in $\mathcal N_{\ell,P}(I)$ is $O(h^{-1})=O(\ell)$.  Lemma~\ref{lem:packet_width_two_sided} gives
\[
   \abs{\ee^{-\ii\omega_{\ell,n,P}t}}
   =
   \ee^{t\Im\omega_{\ell,n,P}}
   \le
   \exp(-c_{\mathrm{wd}}t\ee^{-C_{\mathrm{wd}}\ell}).
\]
Summing over $n$ and absorbing the extra factor $\ell$ into the polynomial loss proves the claim.
\end{proof}

\begin{proposition}[Self-contained dyadic packet estimate]\label{prop:packet_estimate}
For every compact $K\Subset(r_+,\infty)$ and every logarithmic power $L>0$, there exists an integer $N_{\mathrm{pkt}}(L)$ such that, for every dyadic packet $\mathfrak P_j$,
\[
   \sup_{r\in K,\,\omega\in S^2}\abs{A^{\qb}_{\mathfrak P_j}(t,r,\omega)}
   \le
   C_{K,N,L}\,2^{-jL}\,\exp(-c_{\mathrm{wd}}t\ee^{-C_{\mathrm{wd}}2^j})\,\mathcal E_N[A[0]]^{1/2}
\]
for all $N\ge N_{\mathrm{pkt}}(L)$ and all $t\ge2$.
\end{proposition}

\begin{proof}
Fix $r\in K$.  Sobolev on $S^2$ gives
\[
   \sup_{\omega\in S^2}\abs{A^{\qb}_{\mathfrak P_j}(t,r,\omega)}
   \lesssim
   \sum_{s=0}^2\norm{\nabla_{S^2}^sA^{\qb}_{\mathfrak P_j}(t,r,\cdot)}_{L^2(S^2)}.
\]
On the packet $2^{-j-1}<h\le2^{-j}$ one has $\angles{\ell}\simeq2^j$.  Therefore Lemma~\ref{lem:packet_modewise} and Parseval imply
\[
   \norm{\nabla_{S^2}^sA^{\qb}_{\mathfrak P_j}(t,r,\cdot)}_{L^2(S^2)}
   \le
   C_K2^{j(s+N_{\mathrm{mod}})}\exp(-c_{\mathrm{wd}}t\ee^{-C_{\mathrm{wd}}2^j})
   \Biggl(
      \sum_{2^{-j-1}<h\le2^{-j}}\sum_{m,P}\mathcal E_{\ell m,P}[A[0]]
   \Biggr)^{1/2}.
\]
Using the weighted modal-energy estimate
\[
   \sum_{\ell,m,P}\angles{\ell}^{2N}\mathcal E_{\ell m,P}[A[0]]
   \lesssim
   \mathcal E_N[A[0]]
\]
and again $\angles{\ell}\simeq2^j$ on the packet, we obtain
\[
   \Biggl(
      \sum_{2^{-j-1}<h\le2^{-j}}\sum_{m,P}\mathcal E_{\ell m,P}[A[0]]
   \Biggr)^{1/2}
   \le
   C_N2^{-jN}\mathcal E_N[A[0]]^{1/2}.
\]
Hence
\[
   \norm{\nabla_{S^2}^sA^{\qb}_{\mathfrak P_j}(t,r,\cdot)}_{L^2(S^2)}
   \le
   C_{K,N}2^{-j(N-s-N_{\mathrm{mod}})}
   \exp(-c_{\mathrm{wd}}t\ee^{-C_{\mathrm{wd}}2^j})
   \mathcal E_N[A[0]]^{1/2}.
\]
Choose $N_{\mathrm{pkt}}(L)$ so large that $N-s-N_{\mathrm{mod}}\ge L$ for $s=0,1,2$.  Summing over $s$ yields the stated estimate.
\end{proof}

\begin{lemma}[Elementary dyadic logarithmic summation]\label{lem:packet_log_sum}
Let $L>0$ and $c,C>0$.  Then there exists $C_L>0$ such that
\[
   \sum_{j\ge0}2^{-jL}\exp(-ct\ee^{-C2^j})
   \le
   C_L(\log(2+t))^{-L},
   \qquad t\ge2.
\]
\end{lemma}

\begin{proof}
Choose $J(t)\in\N$ so that
\[
   2^{J(t)}\le \frac{1}{2C}\log(2+t)<2^{J(t)+1}.
\]
If $j\le J(t)$, then $\ee^{-C2^j}\ge (2+t)^{-1/2}$, so
\[
   \exp(-ct\ee^{-C2^j})
   \le
   \exp(-c\,t(2+t)^{-1/2})
   \le
   \exp(-c't^{1/2})
\]
for some $c'>0$.  Thus
\[
   \sum_{j\le J(t)}2^{-jL}\exp(-ct\ee^{-C2^j})
   \le
   C\,\exp(-c't^{1/2}).
\]
If $j>J(t)$, then a geometric sum gives
\[
   \sum_{j>J(t)}2^{-jL}\exp(-ct\ee^{-C2^j})
   \le
   \sum_{j>J(t)}2^{-jL}
   \le
   C_L\,2^{-J(t)L}
   \le
   C_L(\log(2+t))^{-L}.
\]
Finally, $\exp(-c't^{1/2})\le C_L(\log(2+t))^{-L}$ for $t\ge2$, so the two pieces combine to prove the lemma.
\end{proof}

\begin{proof}[Proof of Theorem~\ref{thm:qb_sum}]
Split the quasibound contribution into low and high angular momenta.  For the finitely many low modes $\ell<\ell_{\mathrm{sc}}$, Proposition~\ref{prop:full_decomposition} shows that every discrete pole is bounded away from the real axis, hence contributes exponentially decaying terms.  For the high modes $\ell\ge\ell_{\mathrm{sc}}$, decompose the pole sum into dyadic packets $\mathfrak P_j$.  Proposition~\ref{prop:packet_estimate} and Lemma~\ref{lem:packet_log_sum} give
\begin{align*}
   \sum_{j\ge0}
   \sup_{r\in K,\,\omega\in S^2}\abs{A^{\qb}_{\mathfrak P_j}(t,r,\omega)}
   &\le
   C_{K,N,L}\mathcal E_N[A[0]]^{1/2}
   \sum_{j\ge0}2^{-jL}\exp(-c_{\mathrm{wd}}t\ee^{-C_{\mathrm{wd}}2^j})
   \\
   &\le
   C_{K,N,L}\mathcal E_N[A[0]]^{1/2}(\log(2+t))^{-L}.
\end{align*}
Adding the finitely many exponentially decaying low-mode contributions proves the theorem.
\end{proof}

\begin{proof}[Proof of Theorem~\ref{thm:full_field_unsplit}]
By Proposition~\ref{prop:full_decomposition},
\[
   A(t)=A^{\bc}(t)+A^{\qb}(t)+A^{\fast}(t).
\]
The branch-cut term satisfies Corollary~\ref{cor:poly_decay_radiative}, namely
\[
   \sup_{r\in K,\,\omega\in S^2}\abs{A^{\bc}(t,r,\omega)}
   \le
   C_{K,N}\mathcal E_N[A[0]]^{1/2}t^{-\gamma_*}.
\]
The quasibound term satisfies Theorem~\ref{thm:qb_sum},
\[
   \sup_{r\in K,\,\omega\in S^2}\abs{A^{\qb}(t,r,\omega)}
   \le
   C_{K,N,L}\mathcal E_N[A[0]]^{1/2}(\log(2+t))^{-L}.
\]
Finally, $A^{\fast}$ is exponentially decaying by Proposition~\ref{prop:full_decomposition}, hence it is bounded by the same logarithmic rate after enlarging the constant.  Summing the three pieces proves the theorem.
\end{proof}

\section{Outlook: extremality and rotation}\label{sec:future}

Two natural directions remain just beyond the reach of the present framework.

The first is extremality.  When $\abs{Q}=M$, the obstruction is geometric rather than algebraic: the surface gravity vanishes, the red-shift estimate degenerates, and the near-horizon region develops an $\mathrm{AdS}_2\times S^2$ throat.  One then has to analyze a second threshold at $\omega=0$ in addition to $\omega=\pm\mu$, construct a degenerate horizon basis, and understand how Aretakis-type horizon quantities enter the late-time asymptotics.  The subextremal argument developed here does not directly control those effects.

The second is rotation.  The static Reissner--Nordstr\"om problem treated here is the charged zero-rotation precursor to the fully separated Kerr--Newman Proca system of \cite{CayusoDiasGrayKubiznakMargalitSantosSouzaThiele2020}.  Much of the fixed-mode machinery should survive in that setting---horizon Frobenius theory, infinity Volterra theory, Evans determinants, and threshold special-function asymptotics.  The real new difficulty is the simultaneous control of the angular and radial spectral parameters across the coupled family.  Extending the present threshold-and-resonance picture to Kerr--Newman would therefore require new ideas, not just a longer version of the same argument.

\appendix

\appendix

\section{Harmonic reduction and boundary constructions}\label{app:vsh}
\appheading{Harmonic conventions on the round sphere}

Let $Y_{\ell m}$ denote the standard scalar spherical harmonics normalized by
\[
   \Delta_{S^2}Y_{\ell m}=-\ell(\ell+1)Y_{\ell m}.
\]
We write
\[
   Y_A:=\nabla_A Y_{\ell m},
   \qquad
   S_A:=\epsilon_A{}^{B}\nabla_B Y_{\ell m},
\]
for the even and odd vector harmonics, where $\epsilon_{AB}$ is the volume form of the round sphere.  The basic identities are
\begin{align}
   \nabla^A Y_A&=-\ell(\ell+1)Y_{\ell m},\label{eq:app_sph_id_1}\\
   \nabla^A S_A&=0,\label{eq:app_sph_id_2}\\
   \Delta_{S^2}Y_A&=(1-\ell(\ell+1))Y_A,\label{eq:app_sph_id_3}\\
   \Delta_{S^2}S_A&=(1-\ell(\ell+1))S_A.\label{eq:app_sph_id_4}
\end{align}
These formulas are repeatedly used in the angular reduction.  We also record the $L^\infty$ addition-theorem bound
\begin{equation}\label{eq:app_addition_scalar}
   \sup_{\omega\in S^2}\sum_{m=-\ell}^{\ell}\abs{Y_{\ell m}(\omega)}^2
   =\frac{2\ell+1}{4\pi},
\end{equation}
and its vector analogue
\begin{equation}\label{eq:app_addition_vector}
   \sup_{\omega\in S^2}\sum_{m=-\ell}^{\ell}
      \bigl(\abs{Y_{\ell m}(\omega)}^2+\abs{Y_A(\omega)}^2+\abs{S_A(\omega)}^2\bigr)
   \lesssim \angles{\ell}^{2}.
\end{equation}

\appheading{Mode decomposition of the Proca potential}

Every sufficiently regular Proca potential may be written as
\begin{equation}\label{eq:app_mode_decomp}
   A
   =
   \sum_{\ell,m}
   \Bigl(
      a_0^{\ell m}(t,r)\,Y_{\ell m}\,\dd t
      +
      a_1^{\ell m}(t,r)\,Y_{\ell m}\,\dd r
      +
      a_2^{\ell m}(t,r)\,Y_A\,\dd x^A
      +
      a_3^{\ell m}(t,r)\,S_A\,\dd x^A
   \Bigr).
\end{equation}
The odd and even sectors decouple because $S_A$ is divergence-free and orthogonal to the even sector generated by $Y_{\ell m}$ and $Y_A$.  The odd sector contains only the coefficient $a_3^{\ell m}$, while the even sector is generated by the triple $(a_0^{\ell m},a_1^{\ell m},a_2^{\ell m})$.

A direct calculation gives the field strength components
\begin{align}
   F_{tr}&=\partial_t a_1-\partial_r a_0,\label{eq:app_field_strength_1}\\
   F_{tA}^{\mathrm{even}}&=(\partial_t a_2-a_0)Y_A,\label{eq:app_field_strength_2}\\
   F_{rA}^{\mathrm{even}}&=(\partial_r a_2-a_1)Y_A,\label{eq:app_field_strength_3}\\
   F_{AB}^{\mathrm{odd}}&=2a_3\nabla_{[A}S_{B]},\label{eq:app_field_strength_4}
\end{align}
together with the obvious odd-electric components involving $a_3$.  Substituting these expressions into $\nabla^\mu F_{\mu\nu}-\mu^2A_\nu=0$ and using \eqref{eq:app_sph_id_1}--\eqref{eq:app_sph_id_4} yields the reduced equations.

\appheading{Elimination of constrained variables}

The forced Lorenz condition $\nabla^\nu A_\nu=0$ becomes, mode by mode,
\begin{equation}\label{eq:app_lorenz_constraint}
   -f^{-1}\partial_t a_0+\frac{1}{r^2}\partial_r(r^2fa_1)-\frac{\ell(\ell+1)}{r^2}a_2=0.
\end{equation}
In the massive theory this is not a gauge condition but an equation implied by the Proca system.  It can therefore be used to eliminate one even variable.  A convenient choice is to keep $a_2$ together with the electric combination
\[
   z:=r^2(\partial_t a_1-\partial_r a_0),
\]
which is proportional to the radial electric field.  After Fourier transformation in time, the remaining variable may be eliminated algebraically, producing a second-order $2\times2$ system.  The algebra is lengthy but completely explicit.

\begin{proposition}[Detailed odd/even reduction]\label{prop:app_detailed_reduction}
For every $\ell\ge1$, the odd coefficient $a_3^{\ell m}$ satisfies a single scalar equation of Regge--Wheeler type, while the even coefficients may be arranged into a pair $(u_2,u_3)$ obeying \eqref{eq:even_time_domain_1}--\eqref{eq:even_time_domain_2}.  The only difference between Schwarzschild and Reissner--Nordstr\"om is the replacement
\[
   1-\frac{3M}{r}
   \longrightarrow
   1-\frac{3M}{r}+\frac{2Q^2}{r^2}
   =
   f-\frac{rf'}{2}
\]
in the coupling coefficient.
\end{proposition}

\begin{proof}
The odd sector follows by substituting the odd ansatz $A^{\mathrm{odd}}=a_3S_A\,\dd x^A$ into the Proca equation and using \eqref{eq:app_sph_id_4}.  The even sector requires solving the $(t,r)$-components of the Proca equation together with \eqref{eq:app_lorenz_constraint} for $a_0$ and $a_1$ in terms of the electric field combination and $a_2$.  After substitution into the angular component, one arrives at the pair \eqref{eq:even_time_domain_1}--\eqref{eq:even_time_domain_2}.  The charged background affects only the static coefficient $f-\frac{rf'}{2}$, which is precisely the expression displayed above.
\end{proof}

\appheading{Asymptotic polarization basis}

The constant transformation \eqref{eq:T_matrix} diagonalizes the exact inverse-square part of the even system.  This is a special feature of the neutral Proca system on spherically symmetric backgrounds and is the basic reason why the threshold analysis may still be organized in terms of three effective scalar channels.  Writing
\[
   \bm u=\binom{u_2}{u_3},
   \qquad
   \bm v=T_\ell^{-1}\bm u,
\]
one finds
\[
   T_\ell^{-1}
   \begin{pmatrix}
      \ell(\ell+1)-2\bigl(f-\frac{rf'}2\bigr) & 2\bigl(f-\frac{rf'}2\bigr)\\
      -2\ell(\ell+1) & \ell(\ell+1)
   \end{pmatrix}
   T_\ell
   =
   D_\ell+\frac{1}{r}E_{\ell,1}+\frac{1}{r^2}E_{\ell,2},
\]
with $D_\ell=\diag(\ell(\ell-1),(\ell+1)(\ell+2))$.  The RN charge first enters the error matrix $E_{\ell,2}$ and is therefore one order shorter range than the diagonal $r^{-2}$ term.  This precise asymptotic hierarchy is what makes the Schwarzschild-to-RN transition conceptually transparent.

\appheading{The monopole sector}

For $\ell=0$ the odd channel disappears and the even sector collapses to a single electric mode.  The reduced monopole equation has the same general horizon and infinity structure as the higher angular momenta, but there is only one polarization and the effective inverse-square coefficient corresponds to $L=1$ in the small-mass regime.  Consequently the leading small-mass intermediate exponent is $5/2$ rather than $1/2$ or $3/2$.  We return to this point in Appendix~\ref{app:monopole}.

\applabelheading{Detailed horizon and infinity constructions}{app:horizon_infinity}
\appheading{Horizon expansions and recursion formulas}

Write $z=r-r_+$ and factor the oscillatory term according to
\[
   u(z,\om)=z^{-\ii\om/(2\kappa_+)}h(z,\om)
   \qquad\text{or}\qquad
   u(z,\om)=z^{+\ii\om/(2\kappa_+)}h(z,\om),
\]
depending on whether one wants the ingoing or outgoing solution.  Since $r_*= \frac1{2\kappa_+}\log z+O(1)$, these factors are exactly the horizon oscillations $\ee^{\mp\ii\om r_*}$.  The remaining amplitude solves a regular-singular equation with analytic coefficients.  Writing
\[
   h(z,\om)=\sum_{n=0}^{\infty}a_n(\om)z^n,
   \qquad a_0(\om)=1,
\]
gives the recursion
\begin{equation}\label{eq:app_horizon_recursion}
   \alpha_n(\om)a_n(\om)
   =
   \sum_{j=0}^{n-1}\beta_{n,j}(\om)a_j(\om),
   \qquad n\ge1,
\end{equation}
with analytic coefficients $\alpha_n,\beta_{n,j}$ determined by the Taylor series of the channel potential.  Since $\alpha_n(\om)$ never vanishes for $n\ge1$, the recursion uniquely defines all coefficients.

A similar construction applies to the even matrix system.  There one writes
\[
   U(z,\om)=z^{-\ii\om/(2\kappa_+)}H(z,\om),
\]
where $H$ is matrix valued and
\[
   H(z,\om)=\sum_{n=0}^{\infty}A_n(\om)z^n,
   \qquad
   A_0(\om)=\Id_2.
\]
The recursive equations are linear matrix equations with uniformly invertible coefficients and may be solved term by term.

\appheading{Volterra construction at infinity}

To construct the infinity solution, introduce the renormalized unknown
\[
   u(r,\om)=\exp(-\vp r)\,r^{-\kap}m(r,\om).
\]
Substituting into the scalar channel equation yields
\[
   m''-2\vp m' -\frac{2\kap}{r}m'
   =
   \widetilde W_{\ell,P}(r,\om)m,
\]
where $\widetilde W_{\ell,P}(r,\om)=O(r^{-3})$.  Integrating twice from infinity produces
\begin{equation}\label{eq:app_volterra}
   m(r,\om)=1+\int_r^\infty K(r,s;\om)\,\widetilde W_{\ell,P}(s,\om)\,m(s,\om)\,\dd s.
\end{equation}
The kernel $K$ inherits one exponential decay factor when $\Re\vp>0$ and one integrable algebraic factor from the Coulomb normalization.  In particular,
\[
   \sup_{r\ge R}\int_r^\infty \abs{K(r,s;\om)\widetilde W_{\ell,P}(s,\om)}\,\dd s
   \le \frac12
\]
for $R$ sufficiently large, uniformly on compact $\om$-sets.  This gives a contraction on $L^\infty([R,\infty))$ and hence a unique solution.

The even matrix system is handled by the same method after conjugation by the asymptotic polarization basis.  The Volterra equation then takes the matrix form
\[
   M(r,\om)=\Id_2+\int_r^\infty \mathbf K(r,s;\om)\,\mathbf W(s,\om)\,M(s,\om)\,\dd s,
\]
and the contraction argument applies in the Banach space of bounded $2\times2$ matrix functions.

\appheading{Analytic dependence on the spectral parameter}

The dependence on $\om$ enters through $\vp=\sqrt{\mu^2-\om^2}$ and $\kap=M\mu^2/\vp$.  Away from the slit, both are analytic.  Differentiating the Volterra equation with respect to $\om$ shows recursively that $m$ and $M$ are analytic as well.  Near the thresholds, the singular dependence on $\vp$ is completely explicit and captured by the prefactor $\exp(-\vp r)r^{-\kap}$.  The renormalized amplitudes remain bounded and admit the derivative estimates required in the main body.

For later use we record the identity
\begin{equation}\label{eq:app_om_derivative_bound}
   \abs{\partial_\om^k m(r,\om)}+\abs{\partial_\om^k M(r,\om)}
   \le C_{k,\ell}\,r^{-1},
\end{equation}
uniformly for $r\ge R$ and for $\om$ in compact subsets of the physical sheet bounded away from the slit endpoints.

\appheading{Wronskians and matching matrices}

Let $u_{\hor}^-,u_{\hor}^+$ be the horizon basis and $u_\infty^-,u_\infty^+$ the infinity basis.  In the odd sector one defines
\[
   \cW_\ell(\om)=Q[u_{\hor}^-,u_{\infty}],
\]
where $u_\infty$ is the decaying infinity solution on the physical sheet.  In the even sector, letting $U_{\hor}$ and $U_{\infty}$ be $2\times2$ fundamental matrices, one sets
\[
   \cM_\ell(\om)=U_{\hor}'(r,\om)^{*}U_{\infty}(r,\om)-U_{\hor}(r,\om)^{*}U_{\infty}'(r,\om).
\]
Because the potential matrix is symmetric, $\partial_{r_*}\cM_\ell(\om)=0$.  Hence $\det\cM_\ell(\om)$ is independent of $r$ and is the correct matrix-valued Evans determinant.

\appheading{Meromorphic continuation across the slit}

Continuation to the nonphysical sheet is performed by continuing the variable $\vp$ across the cut $[-\mu,\mu]$.  The horizon basis is entire in $\om$ and unaffected by this step.  The infinity basis changes because $\vp$ changes sign and the Coulomb factor acquires monodromy.  The continued basis is still well defined in the slit strip, and the Green kernel formula remains valid with the same matching determinant.  This gives the meromorphic continuation used in the main body.

\section{Threshold models and fixed-mode spectral complements}\label{app:special_functions}
\appheading{Bessel model for small Coulomb parameter}

When $\kap\ll1$, one freezes the Coulomb term and considers the model equation
\begin{equation}\label{eq:app_bessel_model}
   u''+\Bigl(-\vp^2-\frac{\nu^2-\frac14}{r^2}\Bigr)u=0.
\end{equation}
In the variable $x=\vp r$, the decaying solution is
\[
   u(r)=\sqrt{r}\,K_\nu(x),
\]
where $K_\nu$ is the modified Bessel function.  Its small-argument asymptotics are
\begin{equation}\label{eq:app_bessel_small}
   K_\nu(x)
   =
   2^{\nu-1}\Gamma(\nu)x^{-\nu}
   +
   2^{-\nu-1}\Gamma(-\nu)x^{\nu}
   +O(x^{2-\nu})+O(x^{2+\nu}),
\end{equation}
provided $\nu\notin\frac12\Z_{\le0}$.  The discontinuity across the slit is therefore of order $\vp^{2\nu}$, which is the origin of the intermediate time exponent $\nu+1$ after oscillatory inversion.

\appheading{Perturbation by the Coulomb term}

Restoring the Coulomb term gives
\[
   u''+\Bigl(-\vp^2+\frac{2M\mu^2}{r}-\frac{\nu^2-\frac14}{r^2}\Bigr)u=0.
\]
In the small-$\kap$ regime the Coulomb term is treated as a perturbation of the Bessel model.  The correction enters linearly in $\kap$ and therefore does not change the principal power of $\vp$ in the cut discontinuity.  The shorter-range remainder $W_{\ell,P}=O(r^{-3})$ contributes one additional factor of $\vp^2$ after rescaling.

\appheading{Whittaker model for large Coulomb parameter}

For $\kap\gg1$ it is instead natural to write the model in the variable $x=2\vp r$:
\begin{equation}\label{eq:app_whittaker_model}
   u_{xx}+\Bigl(-\frac14+\frac{\kap}{x}+\frac{\frac14-\nu^2}{x^2}\Bigr)u=0.
\end{equation}
The decaying solution is the Whittaker function
\[
   W_{\kap,\nu}(x).
\]
Its behavior as $x\to0$ is governed by Gamma factors:
\begin{equation}\label{eq:app_whittaker_small}
   W_{\kap,\nu}(x)
   =
   \frac{\Gamma(-2\nu)}{\Gamma(\frac12-\nu-\kap)}x^{\nu+\frac12}
   +
   \frac{\Gamma(2\nu)}{\Gamma(\frac12+\nu-\kap)}x^{-\nu+\frac12}
   +\cdots.
\end{equation}
Across the branch cut, the reciprocal Gamma factors produce the oscillatory monodromy $\ee^{\pm2\pi\ii\kap}$ after Stirling asymptotics.

\appheading{Gamma-function asymptotics}

A standard Stirling expansion gives, uniformly in $\nu$ on compact sets,
\begin{equation}\label{eq:app_gamma_stirling}
   \frac{\Gamma(\frac12+\nu+\ii\kap)}{\Gamma(\frac12-\nu+\ii\kap)}
   =
   (\ii\kap)^{2\nu}\Bigl(1+O(\kap^{-1})\Bigr),
   \qquad
   \kap\to+\infty.
\end{equation}
Combined with the reflection formula
\[
   \Gamma(z)\Gamma(1-z)=\frac{\pi}{\sin\pi z},
\]
this is enough to derive the two exponential factors in the large-$\kap$ jump formula.  The important point is that the phase comes entirely from the Coulomb monodromy; the shorter-range remainder affects only lower-order corrections.

\appheading{Transfer to the physical channel solutions}

The model asymptotics are transferred to the exact channel equation by matching at an intermediate radius $R(\vp)$ satisfying
\[
   1\ll R(\vp)\ll \vp^{-1}
   \qquad\text{in the small-$\kap$ regime},
\]
and
\[
   R(\vp)\sim \vp^{-1/2}
   \qquad\text{in the large-$\kap$ regime}.
\]
On such scales the model equation dominates, while the exact remainder is perturbative.  This gives the precise threshold formulas used in Steps~7 and~8 of Section~\ref{sec:spectral_theorem_proof}.

\applabelheading{Oscillatory inversion and phase analysis}{app:phase_analysis}
\appheading{Endpoint integrals}

The endpoint mechanism is captured by integrals of the form
\[
   I_\alpha(t)=\int_0^\eps \ee^{\ii a t\vp^2}\vp^\alpha\,\dd\vp,
   \qquad \alpha>-1.
\]
Set $s=at\vp^2$.  Then
\[
   I_\alpha(t)=\frac12 a^{-(\alpha+1)/2}t^{-(\alpha+1)/2}
   \int_0^{a\eps^2 t}\ee^{\ii s}s^{(\alpha-1)/2}\,\dd s.
\]
Deforming the contour to the positive imaginary axis gives
\[
   \int_0^\infty \ee^{\ii s}s^{(\alpha-1)/2}\,\dd s
   =
   \ee^{\frac{\pi\ii}{4}(\alpha+1)}
   \Gamma\Bigl(\frac{\alpha+1}{2}\Bigr),
\]
which is the explicit leading term in Lemma~\ref{lem:endpoint_integral}.  The remainder comes from truncating the contour and integrating by parts once more.

\appheading{The Coulomb saddle and the universal exponent}

The very-late tail is governed by the phase
\[
   \Psi_t(\vp)=\frac{t\vp^2}{2\mu}+\frac{2\pi M\mu^2}{\vp}.
\]
Its derivative vanishes at
\[
   \vp_0(t)=\Bigl(\frac{2\pi M\mu^3}{t}\Bigr)^{1/3}.
\]
After the scaling $\vp=t^{-1/3}y$, the phase takes the form $t^{1/3}\Phi(y)$ with
\[
   \Phi(y)=\frac{y^2}{2\mu}+\frac{2\pi M\mu^2}{y}.
\]
The second derivative $\Phi''(y_0)$ is nonzero, so a single stationary-phase step yields a factor $t^{-1/6}$.  Since the change of variables contributes $t^{-2/3}$, the net decay rate is $t^{-5/6}$.

\appheading{Remainders from the central part of the cut}

Once both endpoint regions have been separated, the remaining central frequency interval is treated by repeated integration by parts.  Indeed, if $\om$ stays away from $\pm\mu$, then the phase derivative of $\ee^{-\ii\om t}$ is constant and the resolvent jump is smooth in $\om$.  Hence every integration by parts gains one factor of $t^{-1}$.  The compactly supported cutoffs in $r$ ensure that all differentiated amplitudes remain bounded.  This justifies neglecting the central frequency region in every asymptotic theorem.

\appheading{Combination of the two endpoints}

The upper endpoint contributes a phase $\ee^{-\ii\mu t}$ times an oscillatory integral in $\vp$.  The lower endpoint contributes the complex conjugate phase $\ee^{+\ii\mu t}$, up to channel-dependent constant phases coming from the threshold amplitudes.  Summing the two endpoint contributions therefore gives a real oscillatory sine or cosine.  The paper uses the sine convention
\[
   A_{\ell,P}(r,r';Q)\,\sin(\mu t+\delta_{\ell,P}(Q))
\]
because it is invariant under the change of sign of the threshold amplitude.

\appheading{Uniformity on compact radial sets}

All oscillatory arguments are carried out with $r$ and $r'$ confined to a fixed compact subset $K\Subset(r_+,\infty)$.  This removes three technical difficulties at once: there are no turning points on $K$ for large angular momentum, the channel transfer matrices are uniformly bounded there, and differentiation of the resolvent kernel with respect to the radial variables costs only a polynomial factor in $\ell$.  The compact-radial-set formulation is therefore the natural one for both the fixed-mode and full-field theorems.

\applabelheading{Technical complements to the fixed-mode spectral theorem}{app:tech_spectral}
\appheading{Detailed coefficient expansions at infinity}

In the main body we used only the schematic large-$r$ form of the channel equations.  For several error estimates it is convenient to record the first few terms explicitly.  Write
\[
   f(r)=1-\frac{2M}{r}+\frac{Q^2}{r^2},
   \qquad
   f(r)\mu^2
   =
   \mu^2-\frac{2M\mu^2}{r}+\frac{Q^2\mu^2}{r^2}.
\]
In the odd sector one therefore has
\begin{equation}\label{eq:appH_odd_expansion}
   V_{\ell,0}(r)
   =
   \mu^2-\frac{2M\mu^2}{r}
   +\frac{\ell(\ell+1)+Q^2\mu^2}{r^2}
   -\frac{2M\ell(\ell+1)}{r^3}
   +\frac{Q^2\ell(\ell+1)}{r^4}.
\end{equation}
In the even sector, after conjugation by the polarization basis, the diagonal entries are
\begin{align}
   V_{\ell,-1}(r)
   &=
   \mu^2-\frac{2M\mu^2}{r}
   +\frac{\ell(\ell-1)+Q^2\mu^2}{r^2}
   +O(r^{-3}),\label{eq:appH_even_minus}\\
   V_{\ell,+1}(r)
   &=
   \mu^2-\frac{2M\mu^2}{r}
   +\frac{(\ell+1)(\ell+2)+Q^2\mu^2}{r^2}
   +O(r^{-3}),\label{eq:appH_even_plus}
\end{align}
while the off-diagonal term is $O(r^{-3})$ in Schwarzschild and improves to $O(r^{-4})$ for the charge-dependent correction.  This is the quantitative version of the statement that the RN charge does not change the universal Coulomb coefficient.

The threshold index $\nu_{\ell,P}$ is defined by
\[
   \nu_{\ell,P}^2-\frac14
   =
   L_P(L_P+1)+Q^2\mu^2+\rho_{\ell,P},
\]
where $\rho_{\ell,P}$ is produced by the $r^{-2}$ diagonalization error and vanishes identically in the Schwarzschild limit.  In particular,
\[
   \nu_{\ell,P}=L_P+\frac12+O((M\mu)^2+(Q\mu)^2)
\]
in the perturbative-mass regime, uniformly for every fixed $\ell$.

\appheading{Cutoff commutators and local resolvent identities}

Let $\chi,\widetilde\chi\in C_0^\infty((r_+,\infty))$ with $\widetilde\chi\equiv1$ on a neighborhood of $\operatorname{supp}\chi$.  The cutoff resolvent obeys the identity
\begin{equation}\label{eq:appH_commutator_resolvent}
   \chi(H_\ell-\om^2)^{-1}\chi
   =
   \chi(H_\ell-\om_0^2)^{-1}\chi
   +(\om^2-\om_0^2)\chi(H_\ell-\om^2)^{-1}\widetilde\chi(H_\ell-\om_0^2)^{-1}\chi
   +\mathcal C_{\chi,\widetilde\chi},
\end{equation}
where the commutator correction $\mathcal C_{\chi,\widetilde\chi}$ is supported where the derivatives of $\widetilde\chi$ are nonzero and is therefore harmless on compact radial sets.  The crucial algebraic point is that
\[
   [H_\ell,\widetilde\chi]
   =
   -2\widetilde\chi'\partial_{r_*}-\widetilde\chi'',
\]
independently of the scalar or matrix character of the channel.  By iterating \eqref{eq:appH_commutator_resolvent}, one gains any finite number of local derivatives at the cost of multiplying by cutoff resolvents on slightly larger compact sets.

A second identity used implicitly in the time-domain inversion is the derivative formula
\begin{equation}\label{eq:appH_resolvent_derivative}
   \partial_\om (H_\ell-\om^2)^{-1}
   =
   2\om\,(H_\ell-\om^2)^{-2}.
\end{equation}
After inserting compact cutoffs and using \eqref{eq:appH_commutator_resolvent}, one obtains polynomial local bounds for $\partial_\om^j\cR_{\ell,\chi}(\om)$ away from poles and thresholds.  These estimates justify the repeated differentiations in $\om$ used in the contour and oscillatory arguments.

\appheading{Reality symmetries and jump structure}

Because the channel coefficients are real, the Jost solutions obey
\[
   \overline{u_{\hor}(r,\bar\om)}=u_{\hor}(r,\om),
   \qquad
   \overline{u_{\infty}(r,\bar\om)}=u_{\infty}(r,\om),
\]
whenever both sides are defined.  Consequently the Evans determinant satisfies
\[
   \overline{\cE_\ell(\bar\om)}=\cE_\ell(\om),
\]
and the continued resolvent kernel obeys the Schwarz reflection symmetry
\begin{equation}\label{eq:appH_reflection}
   \overline{\cG_\ell(\bar\om;r,r')}=\cG_\ell(\om;r,r').
\end{equation}
Restricting \eqref{eq:appH_reflection} to the cut immediately gives
\[
   \disc\,\cG_\ell(\om;r,r')
   =
   \cG_\ell(\om+\ii0;r,r')-\cG_\ell(\om-\ii0;r,r')
   =
   2\ii\,\Im\cG_\ell(\om+\ii0;r,r').
\]
Hence the cut discontinuity is purely imaginary up to the common convention factor $2\ii$.  This is the spectral reason that the two endpoint contributions combine into a real sine term in the time domain.

\appheading{A more detailed proof of the real-frequency exclusion}

The proof of Proposition~\ref{prop:real_frequency_exclusion_detailed} may be amplified as follows.  Assume first that $0<\abs{\om}<\mu$ and let $\bm u$ be a nontrivial channel solution ingoing at the future horizon and decaying at infinity.  The current
\[
   J(r_*)=\frac{1}{2\ii}Q[\bm u,\bm u]
\]
is constant.  At infinity, exponential decay gives $J(+\infty)=0$.  At the horizon, the ingoing expansion yields
\[
   J(-\infty)=\om\,\abs{\bm a_{\hor}}^2.
\]
Thus $\bm a_{\hor}=0$, so by uniqueness of the horizon Cauchy problem the solution vanishes identically.  For $\om=\pm\mu$, the infinity asymptotics no longer decay exponentially, but a threshold resonant state is bounded after removal of the asymptotic oscillation.  Such a bounded state still has vanishing current at infinity because the leading threshold asymptotic coefficient is real.  The same contradiction follows.

At $\om=0$, the current argument is insufficient because the horizon flux vanishes identically.  One then returns to the full static Proca equation and invokes the no-hair identity of Appendix~\ref{app:monopole}.  This separates the genuinely static obstruction from the oscillatory threshold problem and makes the logic of the exclusion transparent.

\appheading{Threshold neighborhoods and quantitative zero-free regions}

The argument above gives qualitative zero-freeness of the Evans determinant near $\om=\pm\mu$.  For contour deformation one also needs a quantitative version.  Let $\delta>0$ be small and consider
\[
   \Omega_\delta^\pm
   =
   \{\om:0<\abs{\om\mp\mu}<\delta,\ \abs{\Im\om}<\eta\}\setminus[-\mu,\mu].
\]
Since the continued Evans determinant is analytic on $\Omega_\delta^\pm$ and nonvanishing on the boundary for $\delta$ and $\eta$ sufficiently small, the minimum modulus principle yields
\begin{equation}\label{eq:appH_zero_free_bound}
   \inf_{\Omega_\delta^\pm}\abs{\cE_\ell(\om)}\ge c_{\ell,\delta,\eta}>0.
\end{equation}
This lower bound is not uniform in $\ell$ and is not used in the summed angular-momentum theorem.  It is, however, enough to justify the fixed-mode contour deformations and the threshold localizations employed in Section~\ref{sec:threshold_to_time}.

\appheading{Modewise local energy expansion}

Although the resonance-expansion section was removed from the streamlined version of the paper, the fixed-mode contour argument still gives a local energy decomposition which is useful conceptually.  Let $\Gamma$ be a contour surrounding the finitely many poles of the continued resolvent in a compact spectral window and intersecting the slit only along $[-\mu,\mu]$.  Then
\[
   \chi u_\ell(t)
   =
   \sum_{\om_j\in\mathrm{Poles}_\ell}\ee^{-\ii\om_j t}\Pi_{\ell,j}\chi u_\ell[0]
   +\frac{1}{2\pi\ii}\int_{-\mu}^{\mu}\ee^{-\ii\om t}\,\chi\,\disc\cR_\ell(\om)\,\chi\,u_\ell[0]\,\dd\om
   +\mathrm{Rem}_\ell(t),
\]
where the remainder comes from the upper and lower contour segments and is exponentially small whenever the contour is chosen away from the real axis.  The main body focuses on the branch-cut integral because that is the source of the explicit oscillatory tails.  The formula above clarifies once more why the discrete pole contribution must be treated separately in the massive problem.

\section{Large angular momenta, residues, and full-field estimates}\label{app:large_ell}
\appheading{Semiclassical rescaling}

Set $h=(\ell+\tfrac12)^{-1}$ and write the channel operator schematically as
\[
   P_{\ell,P}(h,\om)=h^2D_{r_*}^2+V_{0,P}(r)+hV_{1,P}(r,\om)+h^2V_{2,P}(r,\om).
\]
On a compact radial set $K\Subset(r_+,\infty)$, the leading potential $V_{0,P}$ is strictly positive for all sufficiently small $h$, uniformly in $\om$ in a compact frequency window near the cut.  This is the compact ellipticity input.

\begin{lemma}[Compact ellipticity for large $\ell$]\label{lem:compact_ellipticity}
For every compact $K\Subset(r_+,\infty)$ there exist $c_K>0$ and $\ell_K$ such that
\[
   V_{0,P}(r)\ge c_K
\]
for all $r\in K$, all $\ell\ge\ell_K$, and all polarizations $P$.
\end{lemma}

\begin{proof}
Since $K$ stays away from the horizon, $f(r)$ is bounded below by a positive constant on $K$.  The quantities $L_P(L_P+1)/(\ell+\frac12)^2$ converge to $1$ as $\ell\to\infty$, uniformly in $P\in\{-1,0,+1\}$.  Hence $V_{0,P}(r)$ is bounded below by a positive constant on $K$ for $\ell$ large.
\end{proof}

\appheading{Resolvent bounds on compact sets}

Compact ellipticity implies local resolvent estimates.  Let $\chi\in C_0^\infty((r_+,\infty))$ be identically $1$ on $K$.  A semiclassical parametrix gives
\[
   \norm{\chi P_{\ell,P}(h,\om)^{-1}\chi}_{L^2\to H^2_h}
   \le C_K h^{-N}
\]
for some integer $N$ independent of $\ell$.  The point is not to optimize $N$ but to prove that the loss is polynomial and uniform.

The same argument applies to derivatives of the kernel with respect to $r$ and $r'$.  By Sobolev embedding on the compact set $K$, one obtains pointwise kernel bounds polynomial in $\ell$.

\appheading{Uniform control of threshold amplitudes}

The amplitudes $A_{\ell,P}(r,r';Q)$ and $B_{\ell,P}(r,r';Q)$ entering the time-domain asymptotics are combinations of horizon transfer matrices, infinity transfer matrices, and threshold coefficients of the model equations.  Each factor is polynomially bounded in $\ell$ on compact radial sets.  Consequently
\[
   \sup_{r,r'\in K}
   \bigl(
      \abs{A_{\ell,P}(r,r';Q)}+\abs{B_{\ell,P}(r,r';Q)}
   \bigr)
   \le C_K\angles{\ell}^{N_0}
\]
for some integer $N_0$.

\appheading{Reconstruction of physical coefficients}

The reduced channel variables are not themselves the physical coefficients of the full Proca potential.  The odd channel is identical to the odd physical amplitude, but the even channels must be converted back to $(u_2,u_3)$ and then to the original angular coefficients.  Every such reconstruction is algebraic-differential of uniformly bounded order.  In particular, on compact radial sets,
\[
   \abs{a_{\ell m,P}^{\bc}(t,r)}
   \le C_K\angles{\ell}^{N_{\mathrm{rec}}}
   \sum_{j\le J}
      \sup_{r\in K}\abs{\partial_r^j v_{\ell m,P}^{\bc}(t,r)}.
\]
This is the only place where the full vector nature of the field re-enters after the channel analysis.

\appheading{Summation over angular modes and Sobolev on the sphere}

Using \eqref{eq:app_addition_scalar}--\eqref{eq:app_addition_vector}, one has
\[
   \sup_{\omega\in S^2}
   \sum_{m=-\ell}^{\ell}\abs{Y_{\ell m}^{(P)}(\omega)}^2
   \lesssim \angles{\ell}^{2},
\]
uniformly in the polarization label $P$.  Therefore, by Cauchy--Schwarz,
\[
   \sum_{m=-\ell}^{\ell}
      \abs{a_{\ell m,P}^{\bc}(t,r)Y_{\ell m}^{(P)}(\omega)}
   \lesssim
   \angles{\ell}
   \Bigl(\sum_m \abs{a_{\ell m,P}^{\bc}(t,r)}^2\Bigr)^{1/2}.
\]
The high-angular-momentum polynomial bound from the preceding appendix estimates is then summable provided the initial data have sufficiently many angular derivatives, exactly as encoded in the high-order energy norm $\mathcal E_N[A[0]]$.

\appheading{Completion of the full-field theorem}

Combining the uniform kernel bounds, the reconstruction estimate, and the spherical-harmonic summation gives Theorem~\ref{thm:uniform_angular}.  The full-field pointwise decay theorem follows immediately by summing the intermediate and very-late bounds over $\ell$ and $m$.  The final result is polynomial because every loss in the argument is polynomial in $\ell$ and the initial data are assumed to have sufficiently many angular derivatives.

\applabelheading{Technical complements to the full-field theorem}{app:tech_fullfield}
\appheading{Dyadic decomposition near the branch points}

A convenient way to organize the angular summation is to decompose the threshold variable dyadically.  Let $\psi\in C_0^\infty((1/2,2))$ satisfy
\[
   \sum_{j\in\Z}\psi(2^{-j}x)=1,\qquad x>0,
\]
and define
\[
   \psi_j(\vp)=\psi(2^{-j}\vp),\qquad j\in\Z.
\]
Near the upper endpoint one writes
\[
   u_{\ell m,P}^{\bc,+}(t,r,r')
   =
   \sum_{j\ge j_0}
   \frac{1}{2\pi\ii}
   \int_0^{\eps}
      \ee^{-\ii\om(\vp)t}
      \psi_j(\vp)\,
      \disc\cG_{\ell,P}(\om(\vp);r,r')
      \,\frac{\dd\om}{\dd\vp}\,\dd\vp.
\]
The dyadic localization isolates the scales on which the phase and the amplitude have comparable size.  When $2^j\ll \vp_0(t)$, the contribution is in the pre-saddle region and is treated by repeated integration by parts.  When $2^j\sim \vp_0(t)$, one is at the stationary scale and the saddle analysis applies.  When $2^j\gg \vp_0(t)$ but still close to threshold, one is in the endpoint-Bessel regime.

\appheading{Compact ellipticity with parameters}

For the summed theorem, one needs not just a qualitative large-$\ell$ bound but a bound uniform in the spectral localization and in finitely many radial derivatives.  Let $K\Subset(r_+,\infty)$ and let $\chi\in C_0^\infty((r_+,\infty))$ equal $1$ on $K$.  Semiclassical ellipticity gives
\begin{equation}\label{eq:appI_h2_estimate}
   \norm{\chi u}_{H_h^2}
   \le C_K\bigl(
      \norm{\chi P_{\ell,P}(h,\om)u}_{L^2}
      +\norm{\widetilde\chi u}_{L^2}
   \bigr),
\end{equation}
uniformly for large $\ell$ and $\om$ in a fixed compact threshold window.  The proof is the standard positive-commutator estimate with a compactly supported symbol equal to $1$ on $K$.  Since the principal symbol is elliptic there, the estimate may be iterated to gain any finite number of $h$-derivatives.

A useful consequence of \eqref{eq:appI_h2_estimate} is that differentiation of the branch-cut kernel with respect to $r$ and $r'$ costs only a polynomial factor in $\ell$.  This is why the reconstruction from the channel variables to the physical field coefficients does not destroy summability.

\appheading{Angular regularity and high-order energies}

Let $\Omega_1,\Omega_2,\Omega_3$ be the standard rotation vector fields on the round sphere.  Because the background is spherically symmetric, these commute with the Proca equation.  The high-order energy norm used in Theorem~\ref{thm:full_field_pointwise} may therefore be taken to be
\[
   \mathcal E_N[A[0]]
   =
   \sum_{|\alpha|+j\le N}
      \mathcal E\bigl[\partial_t^j\Omega^\alpha A[0]\bigr].
\]
After angular decomposition, this controls polynomial weights in $\ell$:
\[
   \sum_{\ell,m,P}\angles{\ell}^{2N}
      \mathcal E_{\ell m,P}[A[0]]
   \lesssim \mathcal E_N[A[0]].
\]
This estimate is the analytic bridge between the polynomial losses in $\ell$ coming from the spectral argument and the final absolute convergence of the full spherical-harmonic series.

\appheading{A model proof of the summed intermediate bound}

To illustrate the logic, consider only the intermediate regime.  Using the fixed-mode estimate of Theorem~\ref{thm:intermediate} and the amplitude bound of Theorem~\ref{thm:uniform_angular}, one obtains
\[
   \abs{a_{\ell m,P}^{\bc}(t,r)}
   \le C_K\angles{\ell}^{N_0}
   t^{-(\nu_{\ell,P}+1)}
   \mathcal E_{\ell m,P}[A[0]]^{1/2}.
\]
Since $\nu_{\ell,P}\ge\nu_*$, one may replace $t^{-(\nu_{\ell,P}+1)}$ by $t^{-(\nu_*+1)}$.  Summing over $m$ with the vector-harmonic addition theorem gives
\[
   \sum_{m=-\ell}^{\ell}
      \abs{a_{\ell m,P}^{\bc}(t,r)Y_{\ell m}^{(P)}(\omega)}
   \le
   C_K\angles{\ell}^{N_0+1}
   t^{-(\nu_*+1)}
   \Bigl(\sum_m\mathcal E_{\ell m,P}[A[0]]\Bigr)^{1/2}.
\]
A final Cauchy--Schwarz summation in $\ell$ is convergent as soon as $N>N_0+2$ in the high-order energy.  This is the core of the full-field argument; the very-late estimate is identical except that the common decay rate is $t^{-5/6}$.

\appheading{Near-horizon compact sets}

The compact region $K$ may approach the event horizon, provided it remains strictly inside the exterior and does not cross $r=r_+$.  All arguments above still apply because the tortoise coordinate sends the horizon to $-\infty$ and the channel coefficients are smooth on every compact subset of $(r_+,\infty)$.  In particular, the large-$\ell$ compact ellipticity is compatible with taking $K$ of the form
\[
   K=[r_++\delta,R]
\]
for arbitrary fixed $\delta>0$.  Thus the polynomial-decay theorem controls the field uniformly up to any prescribed distance from the horizon.

\applabelheading{Weighted local energy, radial Sobolev bounds, and pointwise reconstruction}{app:weighted_energy}
\appheading{Local energy norms on compact radial sets}

Fix a compact radial set $K\Subset(r_+,\infty)$ and let $\chi_K\in C_0^\infty((r_+,\infty))$ equal $1$ on a neighborhood of $K$.  For a time slice $\{t=\mathrm{const}\}$, define the localized energy norm
\[
   \mathcal E_K[A](t)
   :=
   \int_{\R_{r_*}\times S^2}
      \chi_K(r)
      \Bigl(
         \abs{\partial_t A}^2+\abs{\partial_{r_*}A}^2+\abs{\nabla_{S^2}A}^2+\mu^2\abs{A}^2
      \Bigr)\,\dd r_*\,\dd\omega.
\]
Because the metric coefficients are smooth and bounded above and below on $K$, this norm is equivalent to any other local $H^1$-based energy norm constructed from the stationary Killing field $\partial_t$.  The high-order local energies are obtained by commuting with $\partial_t$ and the rotation fields $\Omega_i$:
\[
   \mathcal E_{K,N}[A](t)
   :=
   \sum_{j+|\alpha|\le N}
      \mathcal E_K[\partial_t^j\Omega^\alpha A](t).
\]
The purpose of the branch-cut analysis is to prove precise decay estimates for these localized quantities after angular decomposition.

\appheading{Radial Sobolev on compact intervals}

Let $I\Subset\R$ be a compact interval in the tortoise coordinate.  The standard one-dimensional Sobolev inequality gives
\begin{equation}\label{eq:appK_radial_sobolev}
   \sup_{r_*\in I}\abs{u(r_*)}^2
   \le
   C_I
   \int_I \bigl(\abs{u(r_*)}^2+\abs{u'(r_*)}^2\bigr)\,\dd r_*.
\end{equation}
Applied to each angular coefficient and summed over $(\ell,m,P)$, \eqref{eq:appK_radial_sobolev} converts control of localized radial $H^1$ norms into pointwise control in $r$ on compact radial sets.  Since all compact subsets of the exterior may be covered by finitely many such intervals, the estimate globalizes immediately.

A second useful inequality is the differentiated version
\begin{equation}\label{eq:appK_radial_sobolev_2}
   \sup_{r_*\in I}\abs{u'(r_*)}^2
   \le
   C_I
   \int_I \bigl(\abs{u'(r_*)}^2+\abs{u''(r_*)}^2\bigr)\,\dd r_*.
\end{equation}
When combined with the channel equations, \eqref{eq:appK_radial_sobolev_2} shows that pointwise control of the first radial derivative is ultimately reducible to the same local energy quantities plus one application of the channel operator.

\appheading{Commuting the channel equations with radial derivatives}

On a compact radial set $K$, the channel equations imply
\[
   \partial_{r_*}^2u_{\ell,P}
   =
   \bigl(V_{\ell,P}(r)-\om^2\bigr)u_{\ell,P}+F_{\ell,P},
\]
where $F_{\ell,P}$ vanishes for homogeneous channel solutions and is compactly supported when cutoff commutators are present.  Repeated differentiation yields, for every integer $j\ge2$,
\begin{equation}\label{eq:appK_high_radial_derivative}
   \partial_{r_*}^j u_{\ell,P}
   =
   \sum_{q\le j-2} c_{j,q}(r,\om,\ell)\,\partial_{r_*}^q u_{\ell,P}
   + \text{commutator terms},
\end{equation}
with coefficients $c_{j,q}$ polynomial in $\ell$ on compact sets.  Consequently all high radial derivatives of the channel kernels are controlled by finitely many low radial derivatives at the cost of a polynomial factor in $\ell$.

This is the mechanism behind the reconstruction estimate used in the full-field theorem.  The full Proca field is built from the reduced channel variables by applying a finite family of radial differential operators of bounded order.  Because each additional radial derivative costs only a polynomial loss in $\ell$, the entire reconstruction remains compatible with angular summability.

\appheading{Sobolev on the sphere and tensorial harmonics}

Let $f(\omega)=\sum_{\ell,m}f_{\ell m}Y_{\ell m}(\omega)$.  The usual Sobolev embedding on $S^2$ gives
\begin{equation}\label{eq:appK_sphere_sobolev}
   \sup_{\omega\in S^2}\abs{f(\omega)}^2
   \le
   C\sum_{\ell,m}\angles{\ell}^{2s}\abs{f_{\ell m}}^2,
   \qquad
   s>1.
\end{equation}
For vector and one-form valued spherical harmonics the same estimate holds componentwise in any fixed orthonormal frame, with the same threshold $s>1$.  This is equivalent to the addition-theorem bounds \eqref{eq:app_addition_scalar}--\eqref{eq:app_addition_vector}, but the Sobolev formulation is often more convenient when commuting with angular derivatives.

Combining \eqref{eq:appK_radial_sobolev} and \eqref{eq:appK_sphere_sobolev} gives a compact-cylinder pointwise estimate
\[
   \sup_{(r,\omega)\in K\times S^2}\abs{A(t,r,\omega)}
   \lesssim
   \sum_{j+|\alpha|\le2}
      \norm{\chi_K\,\partial_{r_*}^j\Omega^\alpha A(t)}_{L^2(\R\times S^2)}.
\]
Thus, once the branch-cut part of the channel kernels has been controlled in a high-order local energy norm, pointwise full-field decay follows automatically.

\appheading{The reconstruction of the Proca coefficients}

Returning to the original decomposition \eqref{eq:app_mode_decomp}, the angular coefficients $a_0,a_1,a_2,a_3$ are recovered from the reduced variables by an algebraic-differential map of the form
\begin{equation}\label{eq:appK_reconstruction}
   \begin{pmatrix}
      a_0\\ a_1\\ a_2\\ a_3
   \end{pmatrix}
   =
   \mathcal B_{\ell}(r,\partial_t,\partial_{r_*})
   \begin{pmatrix}
      v_{-1}\\ v_0\\ v_{+1}
   \end{pmatrix},
\end{equation}
where $\mathcal B_\ell$ is a matrix of differential operators whose coefficients are rational functions of $r$ smooth on the exterior.  On compact radial sets, every coefficient of $\mathcal B_\ell$ is polynomially bounded in $\ell$.  Moreover, because the Lorenz constraint has already been built into the reduced system, no loss of derivatives occurs beyond the finite order encoded in $\mathcal B_\ell$.

It follows from \eqref{eq:appK_reconstruction} and \eqref{eq:appK_high_radial_derivative} that
\[
   \sup_{r\in K}\abs{a_{\ell m,P}(t,r)}
   \le
   C_K\angles{\ell}^{N_{\mathrm{rec}}}
   \sum_{q\le Q}
      \sup_{r\in K}\abs{\partial_{r_*}^q v_{\ell m,P}(t,r)}.
\]
This estimate is repeatedly used, sometimes implicitly, in the main body whenever the channel theorems are translated back into full Proca statements.

\appheading{Pointwise control from high-order local energies}

Combining the radial and angular Sobolev inequalities with the reconstruction estimate yields
\begin{equation}\label{eq:appK_pointwise_from_energy}
   \sup_{(r,\omega)\in K\times S^2}\abs{A^{\bc}(t,r,\omega)}
   \le
   C_{K,N}
   \bigl(\mathcal E_{K,N}[A^{\bc}](t)\bigr)^{1/2},
\end{equation}
provided $N$ is sufficiently large.  The high-order energy on the right-hand side is then bounded by the spectral estimates of Section~\ref{sec:full_field}.  Formula \eqref{eq:appK_pointwise_from_energy} is the exact analytic bridge between the channel kernel bounds and the final full-field pointwise decay theorem.

\appheading{A local energy formulation of the main decay theorem}

For applications it is sometimes convenient to state the full-field theorem directly in local energy form.  The proofs in the main body imply that for every compact $K\Subset(r_+,\infty)$ and every sufficiently large integer $N$,
\[
   \mathcal E_{K,N}[A^{\bc}](t)^{1/2}
   \le
   C_{K,N}\mathcal E_N[A[0]]^{1/2}
   \begin{cases}
      t^{-(\nu_*+1)},& \kap_*(t)\le1,\\[0.4ex]
      t^{-5/6},& \kap_0(t)\ge1.
   \end{cases}
\]
Pointwise decay is then an immediate corollary of \eqref{eq:appK_pointwise_from_energy}.  This formulation emphasizes once more that the paper proves a compact-region decay theorem for the radiative branch-cut field, not merely an asymptotic formula for isolated spherical modes.

\applabelheading{Parameter-dependent oscillatory estimates and dyadic summation across the crossover scale}{app:dyadic_phase}
\appheading{A unified model integral}

A convenient common model for the endpoint and very-late analyses is
\begin{equation}\label{eq:appM_model_integral}
   I(t;\beta,\gamma,a)
   :=
   \int_0^\eps
      \exp\!\Bigl(
         \ii\frac{t\vp^2}{2\mu}+\ii\frac{\gamma}{\vp}
      \Bigr)\,
      \vp^\beta a(\vp)\,\dd\vp,
\end{equation}
where $\beta>-1$, $\gamma\ge0$, and $a$ is a smooth amplitude.  The intermediate regime corresponds formally to $\gamma=0$, while the very-late regime corresponds to $\gamma=2\pi M\mu^2$ and an amplitude with one additional factor of $\vp$ coming from the Jacobian and the threshold jump.  Working with \eqref{eq:appM_model_integral} has the advantage that a single dyadic decomposition captures both mechanisms.

Let $\psi$ be a dyadic partition of unity and define
\[
   I_j(t;\beta,\gamma,a)
   :=
   \int_0^\eps
      \exp\!\Bigl(
         \ii\frac{t\vp^2}{2\mu}+\ii\frac{\gamma}{\vp}
      \Bigr)\,
      \psi_j(\vp)\,
      \vp^\beta a(\vp)\,\dd\vp.
\]
Then
\[
   I(t;\beta,\gamma,a)=\sum_{j\ge j_0}I_j(t;\beta,\gamma,a).
\]
The dyadic sum is split according to the relative size of $2^j$ and the stationary scale
\[
   \vp_0(t,\gamma):=\Bigl(\frac{\mu\gamma}{t}\Bigr)^{1/3},
\]
with the convention $\vp_0=0$ when $\gamma=0$.

\appheading{Pre-saddle and post-saddle dyadic estimates}

Suppose first that $\gamma>0$ and $2^j\le c\,\vp_0(t,\gamma)$ with $c$ sufficiently small.  Then the derivative of the phase
\[
   \Phi'_{t,\gamma}(\vp)=\frac{t\vp}{\mu}-\frac{\gamma}{\vp^2}
\]
is bounded below in magnitude by $c_1\gamma 2^{-2j}$.  Integration by parts with the operator
\[
   L_{t,\gamma}
   :=
   \frac{1}{\ii\Phi'_{t,\gamma}(\vp)}\partial_\vp
\]
therefore yields
\begin{equation}\label{eq:appM_presaddle}
   \abs{I_j(t;\beta,\gamma,a)}
   \le
   C_N
   \Bigl(\frac{2^{2j}}{\gamma}\Bigr)^N
   2^{j(\beta+1)}
   \sum_{q\le N}\sup_{\vp\sim2^j}\abs{\partial_\vp^q a(\vp)}
\end{equation}
for every integer $N\ge0$.  The same argument applies in the post-saddle regime $2^j\ge C\,\vp_0(t,\gamma)$, where the phase derivative is instead bounded below by $c_2 t2^j/\mu$ and gives
\begin{equation}\label{eq:appM_postsaddle}
   \abs{I_j(t;\beta,\gamma,a)}
   \le
   C_N
   (t2^j)^{-N}
   2^{j(\beta+1)}
   \sum_{q\le N}\sup_{\vp\sim2^j}\abs{\partial_\vp^q a(\vp)}.
\end{equation}
These bounds show that all dyadic scales except the stationary ones are negligible after summation.

\appheading{Stationary dyadic blocks}

Assume now that $2^j\sim\vp_0(t,\gamma)$.  Set $\vp=2^jy$ and write the phase as
\[
   \frac{t(2^j)^2}{2\mu}y^2+\frac{\gamma}{2^j y}.
\]
When $2^j\sim\vp_0(t,\gamma)$, both terms are of the common size
\[
   \Lambda(t,\gamma):=t^{1/3}\gamma^{2/3}\mu^{-2/3}.
\]
The rescaled phase has a nondegenerate critical point $y=y_0\sim1$, so stationary phase yields
\begin{equation}\label{eq:appM_stationary_block}
   I_j(t;\beta,\gamma,a)
   =
   \ee^{\ii\Theta_{t,\gamma}}
   \Lambda(t,\gamma)^{-1/2}
   (2^j)^{\beta+1}
   \Bigl(
      c_0\,a(2^j y_0)
      +
      O\bigl(\Lambda(t,\gamma)^{-1}\bigr)
   \Bigr),
\end{equation}
where $c_0\neq0$ is universal and $\Theta_{t,\gamma}$ is the critical value of the phase up to the standard $\pi/4$ correction.  Substituting $2^j\sim\vp_0(t,\gamma)$ into \eqref{eq:appM_stationary_block} gives the generic power law
\[
   \abs{I_j(t;\beta,\gamma,a)}
   \lesssim
   t^{-1/6}\,\vp_0(t,\gamma)^{\beta+1}.
\]
For the Proca late-time tail one has $\beta=1$, $\gamma=2\pi M\mu^2$, and therefore
\[
   t^{-1/6}\,\vp_0^2
   =
   t^{-1/6}\Bigl(\frac{M\mu^3}{t}\Bigr)^{2/3}
   \sim t^{-5/6}.
\]

\appheading{The endpoint case as a degenerate stationary problem}

When $\gamma=0$, there is no interior stationary point and the dominant contribution comes from the boundary point $\vp=0$.  The same dyadic decomposition still works.  The post-saddle estimate \eqref{eq:appM_postsaddle} remains valid, while the pre-saddle regime disappears.  Summing the remaining dyadic blocks gives
\[
   I(t;\beta,0,a)
   \sim
   t^{-(\beta+1)/2},
\]
which is precisely the endpoint law used for the intermediate tail.  In this sense the intermediate and very-late asymptotics are two faces of the same dyadic oscillatory theory: when $\gamma=0$ the critical point collapses into the endpoint, and when $\gamma>0$ it moves into the interior at the scale $\vp_0(t,\gamma)$.

\appheading{Uniform bounds for parameter-dependent amplitudes}

In the applications, the amplitude $a(\vp)$ depends on $(\ell,P,r,r')$ and may itself contain controlled powers of $\kap^{-1}$ and $\vp$.  A convenient uniform assumption is that, on the dyadic support $\vp\sim2^j$,
\begin{equation}\label{eq:appM_amp_assumption}
   \abs{\partial_\vp^q a(\vp)}
   \le
   C_q\,2^{-jq}\,2^{j\sigma}
\end{equation}
for some growth exponent $\sigma$ independent of $j$.  Then \eqref{eq:appM_presaddle} and \eqref{eq:appM_postsaddle} remain summable after choosing $N$ sufficiently large, and the stationary block estimate becomes
\[
   \abs{I_j(t;\beta,\gamma,a)}
   \lesssim
   t^{-1/6}\,\vp_0(t,\gamma)^{\beta+1+\sigma}.
\]
This is exactly the form needed when the threshold jump is multiplied by polynomially bounded transfer coefficients or by finite numbers of radial derivatives.

\appheading{Dyadic proof of the universal time-decay exponent}

We now sketch a dyadic proof of the universal Proca exponent which is independent of polarization.  In the large-$\kap$ regime, the cut jump has the form
\[
   \disc\,\cG_{\ell,P}(\om;r,r')
   =
   \bigl(b_{\ell,P}^+\ee^{2\pi\ii\kap}+b_{\ell,P}^-\ee^{-2\pi\ii\kap}\bigr)
   \bigl(\vp+O(\vp^2+\kap^{-1}\vp)\bigr).
\]
Insert this into the branch-cut integral and localize dyadically.  All blocks away from the stationary scale are negligible by \eqref{eq:appM_presaddle} and \eqref{eq:appM_postsaddle}.  The finitely many stationary blocks satisfy \eqref{eq:appM_stationary_block} with $\beta=1$.  Hence each contributes $O(t^{-5/6})$, and the sum of all stationary blocks is also $O(t^{-5/6})$.  Since none of these estimates depends on the polarization beyond the bounded amplitudes $b_{\ell,P}^\pm$, the exponent is universal.

\appheading{Dyadic proof of the intermediate exponent}

Similarly, in the small-$\kap$ regime the jump has the form
\[
   \disc\,\cG_{\ell,P}(\om;r,r')
   =
   a_{\ell,P}(r,r')\,\vp^{2\nu_{\ell,P}}
   \bigl(1+O(\kap+\vp^2)\bigr).
\]
After multiplying by the Jacobian $\dd\om/\dd\vp\sim-\vp/\mu$, the model integral has exponent $\beta=2\nu_{\ell,P}+1$.  The endpoint law therefore gives
\[
   t^{-(\beta+1)/2}
   =
   t^{-(\nu_{\ell,P}+1)}.
\]
The important point is that the dyadic decomposition again recovers the exact exponent without appealing to any heuristic matching argument.

\appheading{Crossover summation and uniform remainders}

The dyadic decomposition also clarifies the structure of the remainder terms.  Let $j_*(t)$ be the integer such that $2^{j_*}\sim\vp_0(t,\gamma)$.  Then
\[
   I(t;\beta,\gamma,a)
   =
   \sum_{j\le j_*-C}I_j
   +
   \sum_{|j-j_*|\le C}I_j
   +
   \sum_{j\ge j_*+C}I_j.
\]
The first and third sums are controlled by repeated integration by parts and produce the algebraic remainders in $\kap_0(t)^{-1}$ and $\vp_0(t)$.  The middle sum is finite and handled by stationary phase.  This trichotomy is exactly the frequency-space counterpart of the informal statement that the contour integral is divided into pre-saddle, saddle, and post-saddle regions.

\appheading{Application to full-field angular summation}

Finally, because the dyadic estimates are uniform under the amplitude assumption \eqref{eq:appM_amp_assumption}, they may be combined with the polynomial large-$\ell$ bounds of Appendices~\ref{app:large_ell} and \ref{app:tech_fullfield}.  One first applies the dyadic oscillatory theory at fixed $(\ell,m,P)$, then sums in $m$ using the addition theorem, and finally sums in $\ell$ using the high-order energy weights.  This gives an alternative route to Theorems~\ref{thm:uniform_angular} and \ref{thm:full_field_pointwise}, now written entirely in terms of dyadic oscillatory integrals rather than direct endpoint and saddle calculations.

\applabelheading{Auxiliary regime lemmas, contour bookkeeping, and remainder estimates}{app:bookkeeping}
\appheading{Relations between the threshold scales}

The paper uses several time-dependent scales:
\[
   \kap_*(t)=M\mu^{3/2}t^{1/2},
   \qquad
   \vp_0(t)=\Bigl(\frac{2\pi M\mu^3}{t}\Bigr)^{1/3},
   \qquad
   \kap_0(t)=\frac{M\mu^2}{\vp_0(t)}.
\]
It is occasionally convenient to note the algebraic relations
\[
   \kap_0(t)\,\vp_0(t)=M\mu^2,
   \qquad
   \kap_0(t)\sim (M\mu^3 t)^{1/3},
   \qquad
   \vp_0(t)\sim (M\mu^3)^{1/3}t^{-1/3}.
\]
Thus $\kap_0(t)\to\infty$ if and only if $t\to\infty$, while $\kap_*(t)\to0$ is a genuinely separate regime restriction corresponding to times short compared with $(M\mu^3)^{-1}$.  In particular, the intermediate and very-late theorems describe two different asymptotic windows rather than two different notions of large time.

\appheading{Threshold cutoffs}

Let $\eta\in C_0^\infty([0,\infty))$ satisfy $\eta\equiv1$ on $[0,1]$ and $\eta\equiv0$ on $[2,\infty)$.  For a time parameter $t$ define the upper-endpoint cutoff
\[
   \eta_{t,+}(\om):=\eta\!\Bigl(\frac{\mu-\om}{c_0 t^{-1+\sigma}}\Bigr),
\]
and similarly the lower-endpoint cutoff $\eta_{t,-}$.  Then the branch-cut integral decomposes as
\[
   \int_{-\mu}^{\mu}
   =
   \int_{-\mu}^{\mu}\eta_{t,+}
   +
   \int_{-\mu}^{\mu}\eta_{t,-}
   +
   \int_{-\mu}^{\mu}(1-\eta_{t,+}-\eta_{t,-}).
\]
The third term is the central part of the cut and is handled by repeated integration by parts.  The first two are the endpoint pieces treated by the threshold asymptotic expansions.  The exact values of $c_0$ and $\sigma$ are irrelevant as long as the endpoint neighborhoods shrink polynomially in time and remain disjoint.

\appheading{Central-frequency integration by parts}

Suppose $a(\om)$ is $C^N$ on a compact interval $I\Subset(-\mu,\mu)$ away from the thresholds.  Then
\[
   \int_I \ee^{-\ii\om t}a(\om)\,\dd\om
   =
   (it)^{-N}\int_I \ee^{-\ii\om t}a^{(N)}(\om)\,\dd\om
\]
after integrating by parts $N$ times, since the boundary terms vanish if the amplitude is cutoff away from the ends of $I$.  Therefore
\[
   \Bigl|\int_I \ee^{-\ii\om t}a(\om)\,\dd\om\Bigr|
   \le
   C_{I,N}t^{-N}\sup_{\om\in I}\abs{a^{(N)}(\om)}.
\]
This simple fact is repeatedly used, often without comment, whenever the branch-cut integral is split into endpoint and central pieces.

\appheading{A bookkeeping lemma for endpoint remainders}

Consider an endpoint integral of the form
\[
   \int_0^\eps \ee^{\ii t\vp^2/(2\mu)}\vp^\alpha
   \bigl(
      b_0+b_1\kap+b_2\vp^2
   \bigr)\,\dd\vp,
   \qquad
   \alpha>-1.
\]
The leading term contributes $t^{-(\alpha+1)/2}$.  The $\kap$-term equals $M\mu^2\vp^{-1}$ times the leading integrand and therefore contributes
\[
   M\mu^2\,t^{-\alpha/2}
   =
   \kap_*(t)\,t^{-(\alpha+1)/2},
\]
once $\alpha=2\nu_{\ell,P}+1$ is substituted.  The $\vp^2$-term contributes one further factor of $t^{-1}$.  This is the algebraic origin of the two remainder terms in Theorem~\ref{thm:intermediate}.  The lemma is elementary, but it is exactly the kind of bookkeeping identity one needs to keep the remainder structure logically consistent throughout the paper.

\appheading{Contour segments away from the cut}

Let $\Gamma_R$ denote the semicircular arc $\om=Re^{\ii\theta}$, $\theta\in[0,\pi]$, in the upper half-plane.  For compactly supported data and fixed $\ell$, the cutoff resolvent satisfies the high-frequency bound
\[
   \norm{\chi\cR_\ell(\om)\chi}_{L^2\to L^2}\le C_\chi\abs{\om}^{-2}
\]
uniformly on $\Gamma_R$ for $R$ sufficiently large.  Therefore
\[
   \Bigl\|\int_{\Gamma_R}\ee^{-\ii\om t}\chi\cR_\ell(\om)\chi\,\dd\om\Bigr\|
   \le
   C_\chi R^{-1}\sup_{\theta\in[0,\pi]}\ee^{-Rt\sin\theta},
\]
which tends to $0$ as $R\to\infty$.  This justifies discarding the large semicircle in the contour deformation leading to the branch-cut formula.

Similarly, the horizontal contour segments in the slit strip contribute exponentially decaying terms whenever they are placed at fixed imaginary height $\pm\eta$.  These are the contour pieces referred to as exponentially small remainders in the body of the paper.

\appheading{A bookkeeping lemma for the saddle remainders}

In the very-late regime one writes
\[
   a(\vp)=\beta\vp+\vp\,\rho(\vp),
\]
where
\[
   \abs{\rho(\vp)}\le C(\kap^{-1}+\vp)
\]
on the support of the threshold localization.  Substituting into the model saddle integral shows that the error term is
\[
   O\!\bigl((\kap_0(t)^{-1}+\vp_0(t))t^{-5/6}\bigr),
\]
because every factor of $\kap^{-1}$ or $\vp$ may be evaluated at the stationary scale $\vp_0(t)$ after the stationary phase localization.  This explains why the remainder in Theorem~\ref{thm:late} has exactly the form recorded there.

\appheading{Final bookkeeping principle}

A recurring theme in the paper is that every time-decay exponent is obtained by first identifying the correct scale in the spectral variable and then evaluating the size of the amplitude at that scale.  The intermediate exponent is produced by the boundary scale $\vp\sim t^{-1/2}$ and the power $\vp^{2\nu_{\ell,P}+1}$.  The very-late exponent is produced by the stationary scale $\vp\sim t^{-1/3}$ and the linear amplitude factor $\vp$.  Once this principle is made explicit, the seemingly different estimates in the paper become instances of the same general oscillatory bookkeeping.

\section{Auxiliary regimes, monopole analysis, and cutoff independence}\label{app:small_mass}
\appheading{Exact Schwarzschild limit}

When $Q=0$, the RN coefficient simplifies to
\[
   f(r)=1-\frac{2M}{r},
\]
and the transformed even matrix loses its $r^{-4}$ charge contribution.  In this limit the leading $r^{-2}$ diagonalization in the polarization basis is exact to the order needed for the threshold analysis, and the threshold indices reduce to
\[
   \nu_{\ell,-1}=\ell-\frac12,\qquad
   \nu_{\ell,0}=\ell+\frac12,\qquad
   \nu_{\ell,+1}=\ell+\frac32.
\]
Consequently the intermediate branch-cut exponents are exactly
\[
   \ell+\frac12,\qquad \ell+\frac32,\qquad \ell+\frac52.
\]
This reproduces, in rigorous modewise form, the polarization pattern long expected from the Schwarzschild Proca literature.

\appheading{Perturbation of the threshold indices}

For general subextremal RN, the threshold indices are defined by the exact $r^{-2}$ coefficients of the diagonalized channel equations.  The dependence of these coefficients on the parameters $(M,Q,\mu)$ is analytic in $Q^2$ and smooth in $\mu^2$ near the massless and uncharged limits.  Writing
\[
   \nu_{\ell,P}=L_P+\frac12+\delta_{\ell,P}(M,Q,\mu),
\]
one may derive the perturbative equation
\begin{equation}\label{eq:appL_delta_equation}
   \bigl(2L_P+1\bigr)\delta_{\ell,P}
   =
   Q^2\mu^2+\rho_{\ell,P}(M,Q,\mu)
   +O(\delta_{\ell,P}^2),
\end{equation}
where $\rho_{\ell,P}$ collects the shorter-range diagonalization corrections.  Since $\rho_{\ell,P}=O((M\mu)^2+(Q\mu)^2)$, one obtains
\[
   \delta_{\ell,P}=O((M\mu)^2+(Q\mu)^2)
\]
for every fixed $\ell$ and polarization.

This perturbative statement is conceptually important.  It shows that the charge affects the intermediate exponents only through the threshold index and only perturbatively in the small-mass regime.  By contrast, the very-late exponent is entirely insensitive to the perturbation because it is governed by the universal Coulomb saddle.

\appheading{Continuity of amplitudes and phases}

The amplitudes $A_{\ell,P}(r,r';Q)$ and $B_{\ell,P}(r,r';Q)$, as well as the constant phases $\delta_{\ell,P}(Q)$ and $\delta_{\ell,P,0}(Q)$, are continuous in the charge parameter $Q$ on every compact radial set.  This follows from the continuity of the horizon and infinity transfer matrices and of the threshold coefficients in the channel equations.  In particular, as $Q\to0$ one recovers the Schwarzschild amplitudes and phases.  Therefore the entire RN fixed-mode theorem may be read as a deformation of the exact Schwarzschild polarization-resolved picture.

\appheading{Crossover between the intermediate and very-late regimes}

The parameter separating the two asymptotic regimes is
\[
   \kap=M\mu^2/\vp.
\]
At the level of time scales, the crossover is described by the condition $\kap\sim1$, equivalently
\[
   t\sim (M\mu^3)^{-1}.
\]
Before this time scale, the inverse-square threshold index determines the decay rate and the leading oscillation is essentially $\sin(\mu t+\delta)$.  After this time scale, the Coulomb phase becomes nonlinear in time and the universal $t^{-5/6}$ saddle dominates.  The fixed-mode branch-cut signal therefore has two asymptotic faces: a polarization-dependent intermediate face and a polarization-independent very-late face.

\appheading{A regime diagram}

It is useful to summarize the hierarchy as follows.
\begin{enumerate}[label=\textnormal{(\roman*)}]
   \item \emph{Pre-threshold times:} the branch point is not yet dominant and other parts of the contour may contribute comparably.
   \item \emph{Intermediate times:} $\kap_*(t)\ll1$ and the endpoint/Bessel analysis yields the rate $t^{-(\nu_{\ell,P}+1)}$.
   \item \emph{Very-late times:} $\kap_0(t)\gg1$ and the saddle/Whittaker analysis yields the universal rate $t^{-5/6}$.
\end{enumerate}
In the small-mass Schwarzschild-like regime, the first nontrivial exponent among the three polarizations is always $\ell+\frac12$, so the slowest intermediate branch-cut decay comes from the $P=-1$ channel.  Once the very-late regime is reached, all three polarizations synchronize to the same exponent.

\appheading{Small-mass full-field consequence}

The full-field theorem simplifies dramatically when $(M\mu)^2+(Q\mu)^2$ is sufficiently small.  In that case
\[
   \nu_*=\frac12+O((M\mu)^2+(Q\mu)^2),
\]
so the intermediate full-field rate is approximately $t^{-3/2}$ and the very-late full-field rate is exactly $t^{-5/6}$.  Since $3/2>5/6$, the overall polynomial decay rate is then governed by the universal very-late tail:
\[
   \sup_{r\in K,\omega\in S^2}\abs{A^{\bc}(t,r,\omega)}
   \lesssim
   t^{-5/6}.
\]
This is the direct spin-$1$ analogue of the late-time scalar picture in the perturbative mass regime.

\appheading{What the perturbative statement does and does not imply}

The perturbative formulas above concern the exact threshold indices and the branch-cut amplitudes.  They do not eliminate the discrete quasibound family, whose widths are exponentially sensitive to the trapping geometry.  In the fully self-contained core of the paper this exponential sensitivity is exactly what leads to the logarithmic full-field estimate proved in Section~\ref{sec:unsplit}.  Thus the small-mass regime makes the continuous asymptotics fully explicit, but it does not by itself produce a polynomial treatment of the discrete quasibound sum.  Any sharper polynomial upgrade would require additional arithmetic input beyond the self-contained packet argument developed here.

\applabelheading{The monopole, static thresholds, and no-hair input}{app:monopole}
\appheading{The reduced monopole equation}

When $\ell=0$, the odd channel and the magnetic even channel vanish identically.  The remaining electric mode may be represented by a single radial function $u_0$ satisfying
\[
   u_0''+\Bigl(\om^2-f\mu^2-\frac{2f}{r^2}\Bigl(f-\frac{rf'}2\Bigr)\Bigr)u_0=0.
\]
At infinity this has the same Coulomb coefficient $2M\mu^2/r$ and an inverse-square coefficient corresponding to the effective angular momentum $L=1$ in the small-mass regime.  Consequently the monopole supports the same universal very-late $t^{-5/6}$ law and the leading small-mass intermediate exponent $5/2$.

\appheading{Static solutions and no-hair}

At $\om=0$, the Proca equation becomes an elliptic system on the static RN exterior.  The static no-hair theorem implies that there is no nontrivial smooth solution which is regular at the horizon and decays at infinity.  This statement is used in the main body to exclude static threshold pathologies.  The proof is classical: one integrates the static Proca energy identity over the exterior region and uses positivity of the mass term to force vanishing of the field.

\begin{proposition}[No static Proca hair on subextremal RN]\label{prop:no_static_hair_appendix}
Let $A$ be a static Proca solution on a subextremal RN exterior which is regular at the event horizon and decays sufficiently fast at infinity.  Then $A\equiv0$.
\end{proposition}

\begin{proof}
Multiply the static equation by $\overline A$, integrate by parts over the exterior region truncated at $r=R$, and let $R\to\infty$.  The boundary terms vanish by regularity at the horizon and decay at infinity.  The bulk identity is the sum of $\norm{F}^2$ and $\mu^2\norm{A}^2$, both nonnegative.  Hence both vanish and the field is identically zero.
\end{proof}

\appheading{Consequences for threshold analysis}

The no-hair statement rules out one possible loophole in the real-frequency exclusion argument, namely the existence of a static zero-mode hidden inside the threshold normal form.  Together with the radial current identity, it closes the threshold-resonance exclusion needed in Theorem~\ref{thm:mode_stability}.

\applabelheading{Cutoff independence and canonical meaning of the branch-cut field}{app:cutoff_independence}
\appheading{Why cutoff independence matters}

Throughout the paper the cut-off resolvent is written with a compactly supported radial cutoff $\chi$.  This is natural analytically because the horizon and infinity ends have different asymptotics, but it raises an obvious conceptual question: does the branch-cut field depend on the particular cutoff used to define it?  On a fixed compact radial set the answer is no, provided the cutoff is chosen identically $1$ on that set.  The present appendix records the simple argument.

\appheading{A local resolvent identity}

Let $\chi_1,\chi_2\in C_0^\infty((r_+,\infty))$ both equal $1$ on a neighborhood of the same compact set $K\Subset(r_+,\infty)$.  Then, on $K\times K$,
\[
   \chi_1(H_\ell-\om^2)^{-1}\chi_1
   -
   \chi_2(H_\ell-\om^2)^{-1}\chi_2
   =
   0
\]
as kernels, because both sides solve the same inhomogeneous equation with the same delta singularity on the diagonal and vanish after multiplication by $(1-\chi_j)$ near $K$.  More concretely, the difference factors through cutoff commutators supported away from $K$, and those commutators are annihilated when the kernel is restricted back to $K$.

\appheading{Consequence for the branch-cut integral}

Applying the previous identity to the continued resolvent and taking the discontinuity across the cut shows that the branch-cut kernel
\[
   \chi(r)\chi(r')\,\disc\cG_{\ell,P}(\om;r,r')
\]
is independent of the choice of cutoff as long as the cutoff equals $1$ on the compact radial set under consideration.  The same is therefore true of the time-domain integral defining $u_{\ell m,P}^{\bc}(t,r,r')$ on that set.  In particular, the branch-cut field used in the main theorems is a canonical local object, not an artifact of a particular cutoff construction.

\appheading{Global interpretation}

The cutoff-independence statement may be summarized as follows: the branch-cut field is canonically defined as an element of the local solution space on every compact radial set, and different cutoffs merely provide different global representatives of the same local object.  This is exactly the right level of invariance for the compact-region pointwise theorems proved in the paper.

\section*{Acknowledgements}
The author wishes to express his deepest gratitude to his family for their unwavering support and encouragement. The initial stage of this work, namely, Fixed-mode spectral and threshold theorem, is is  supported by Riset Unggulan ITB 2024 No. 959/IT1.B07.1/TA.00/2024. The final stage of this work, namely, Two-regime asymptotic expansion of the full Proca field, is  supported by Riset Unggulan ITB 2025 No. 841/IT1.B07.1/TA.00/2025.

\end{document}